\documentclass[10pt]{amsart}
\usepackage[latin1]{inputenc}
\usepackage[T1]{fontenc}
\usepackage{bbm,bm}
\usepackage{amssymb}
\usepackage{amsmath}
\usepackage{colordvi}
\usepackage{enumerate}
\usepackage{graphicx,xspace}
\usepackage[section]{placeins}
\usepackage{color}
\usepackage{pstricks}
    \definecolor{mygreen}{rgb}{0,.6,0}
    \definecolor{myblue}{rgb}{0,0,.6}

    \definecolor{frank}{rgb}{1,.2,.2}
    \definecolor{stefan}{rgb}{.2,.7,.2}
    \definecolor{aaron}{rgb}{.2,.2,1}
\usepackage{xspace}

\usepackage[bookmarks=true,colorlinks,linkcolor=myblue,citecolor=mygreen]{hyperref}

\newcounter{FNC}[page]
\def\fauxfootnote#1{{\addtocounter{FNC}{2}$^\fnsymbol{FNC}$%
     \let\thefootnote\relax\footnotetext{$^\fnsymbol{FNC}$\Magenta{#1}}}}

 \numberwithin{equation}{section}
\newtheorem{theorem}{Theorem}[section]
\newtheorem{proposition}[theorem]{Proposition}
\newtheorem{lemma}[theorem]{Lemma}

\newtheorem{notation}[subsection]{Notation}
\newtheorem{defi}[theorem]{Definition}
\newtheorem{exam}[theorem]{Example}
\newtheorem{rema}[theorem]{Remark}
\newenvironment{definition}[1][]{\rm\begin{defi}[#1]\rm}{\end{defi}}

\newcommand{\C}{\mathbb{C}}

\def\ssym{\s\mathit{Sym}}

\def\csym{\c\mathit{Sym}}

\def\s{\mathfrak S}

\def\y{\mathcal Y}

\def\c{\mathfrak C}

\def\KK{{\mathbb K}}
\def\ot{\otimes}

\def\ov{\overline}

\def\ot{\otimes}
\def\Dot2{\ov{\D_m^+\ot \D_m^+}}

\def\Dot2{\ov{\D_m^+\ot \D_m^+}}
\def\St{\mbox {St}}

\newcommand{\Placeholder}%
    {\raisebox{.09ex}{\footnotesize $\bullet$}}
\newcommand{\placeholder}%
    {\raisebox{.04ex}{\tiny $\bullet$}}
\newcommand{\Fmn}{F_{m,n}}

\def\psplit{\stackrel{\curlyvee}{\to}}

\newcommand{\tubeA}{\raisebox{-.5ex}{\includegraphics{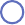}}}
\newcommand{\tubeB}{\raisebox{-.6ex}{\includegraphics{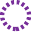}}}
\newcommand{\tubeC}{\raisebox{-.6ex}{\includegraphics{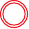}}}
\newcommand{\D} {\Delta}
\newcommand{\K} {\mathcal{K}}                           
\newcommand{\J} {\mathcal{J}}

\newcommand{\KG} {{\K} G}
\newcommand{\JG} {{\J} G}                

\newcommand{\JGr} {\JG_r}
\newcommand{\JGd} {\JG_d}

\author{Lisa Berry} \address{
    Bio-Med Science Academy,\\
    Rootstown, OH
    }
    \email{LBerry114@gmail.com}

\author{Stefan Forcey} \address{
    Department of Mathematics\\
    The University of Akron\\
    Akron, OH 44325-4002
    }
    \email{sf34@uakron.edu}  \urladdr{http://www.math.uakron.edu/\~{}sf34/}

    \author{Maria Ronco} \address{
    Department of Physics and Mathematics\\
    The University of Talca\\
    Talca, Chile.  }

 \author{Patrick Showers} \address{
    Department of Mathematics\\
    The University of Akron\\
    Akron, OH 44325-4002
    }

\title[Polytopes and Trees]{Species substitution, graph suspension, and graded Hopf algebras of painted tree polytopes.}

\begin{document}

\begin{abstract}
Combinatorial Hopf algebras of trees exemplify the connections
between operads and bialgebras. Painted trees were introduced
recently as examples of how graded Hopf operads can bequeath
Hopf structures upon compositions of coalgebras. We put these
trees in context by exhibiting them as the minimal elements of face
posets of certain convex polytopes. The full face posets themselves
often possess the structure of graded Hopf algebras (with one-sided
unit). We can enumerate faces using the fact that they are structure types of
substitutions of combinatorial species. Species considered
here include ordered and unordered binary trees and ordered lists
(labeled corollas). Some of the polytopes that constitute our main
results are well known in other contexts. First we see the classical
permutahedra, and then certain \emph{generalized permutahedra}:
specifically the graph-associahedra
 of suspensions of certain simple graphs. As an aside we show that
 the stellohedra also appear as \emph{liftings} of generalized permutahedra: graph composihedra for complete graphs.
  Thus our results give examples of Hopf algebras of tubings and marked tubings of graphs. We also show an alternative associative
algebra structure on the graph tubings of star graphs.

\includegraphics[width=4.83in]{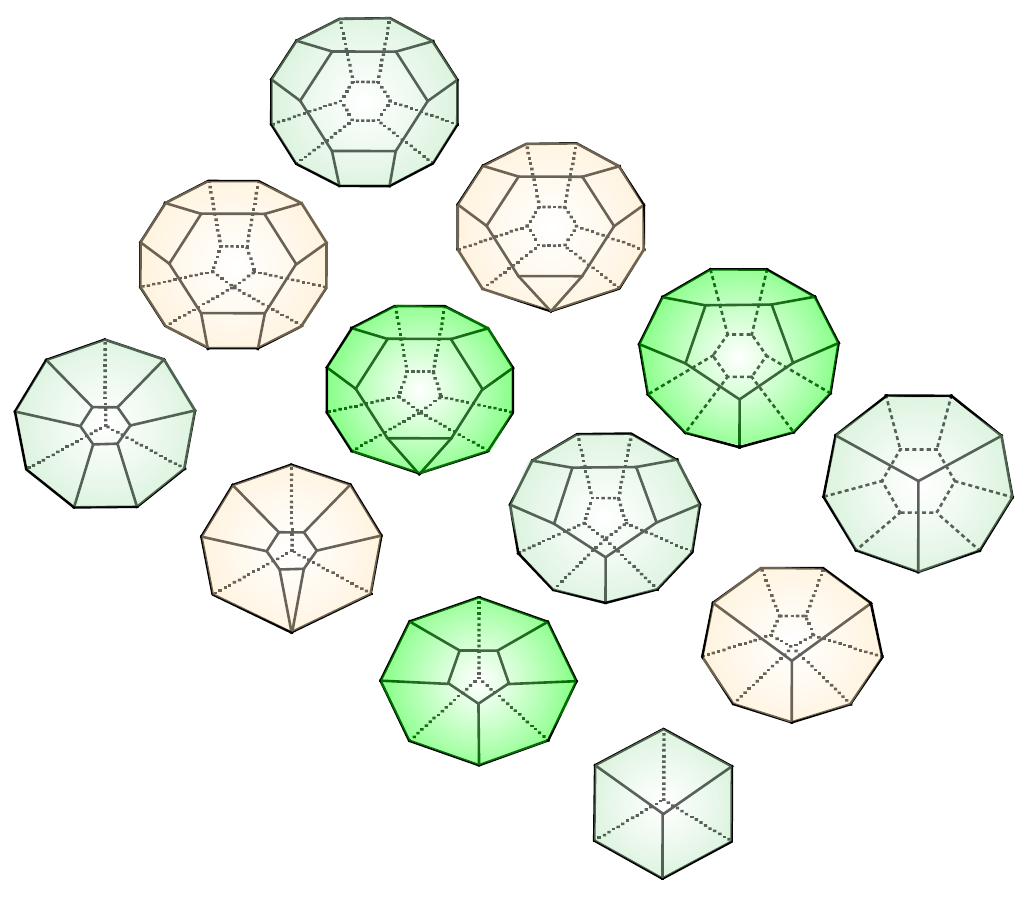}
\end{abstract}

\maketitle

\section{Introduction}
The mathematical operation of grafting trees, root to leaf, is a key
feature in the structure of several important operads and Hopf
algebras. Loday and Ronco initiated the study of these type of structures when they found a Hopf algebra of plane
binary trees in \cite{LR}. The Loday-Ronco algebra is based on the vertices of Stasheff's associahedra, the polytopes that
model homotopy associative spaces.

  There is a surjection from permutations to plane binary
trees: the Tonks projection, defined in \cite{tonks}, from the permutohedron to the associahedron.
 Chapoton found that the Hopf algebras of vertices of these polytopes are subalgebras of
larger ones based on the faces of the respective polytopes
\cite{chap}. Chapoton's algebras are the  differential graded
structures  corresponding to algebras described by Loday and Ronco
in \cite{LR3}. Using that
same surjection on basis elements, the Loday-Ronco algebra is the image
of the Malvenuto-Reutenauer Hopf algebra of permutations \cite{MR}.
There is also a projection from the Loday-Ronco Hopf algebra to the
algebra of quasisymmetric polynomials. The authors of \cite{LR}
showed Hopf algebra maps which factor the descent map from
permutations to Boolean subsets.

 In \cite{LR2} the authors describe the product of planar binary trees in terms of the Tamari order.
  In 2005 and 2006 Aguiar and Sottile characterized operations in these algebras by using M\"obius
functions (of the Tamari order and of the weak Bruhat order) to
obtain new bases, in respectively \cite{AS-LR},\cite{AS-MR}. Their
work gave a nice way to construct a basis of primitive elements,
using the irreducible trees.  In \cite{ForSpr:2010} the authors
characterize the same operations in terms of inclusions (into the
larger polytopes) of products of polytope faces.

 Alternatively, since the Loday-Ronco algebra is self-dual, it can
project to the divided power Hopf algebra. In \cite{FLS3} the
authors used the following notation: $\s Sym$ for the
Malvenuto-Reutenauer Hopf algebra,
   $\y Sym$ for the Loday-Ronco Hopf algebra, and $\c Sym$ for the divided power Hopf algebra. They defined the idea of
grafting with two colors, preserving the colors after the graft in
order to have two-tone, or painted, trees with various structures
possible in each colored region. Here we review the definitions,
adding some generality and defining poset structures on each set of
painted trees. We extend the coalgebra structure to twelve new
vector spaces, and we extend the Hopf algbra structure to nine of
those. We are also able to conclude that eight of the new coalgebras
defined have underlying geometries of polytope sequences.

The \emph{stellohedra}, or star-graph-associahedra, were first
defined using the latter terminology by Carr and Devadoss in
\cite{dev-carr}.  The former terminology was introduced in
\cite{post2}, where these polytopes were studied as special cases of
nestohedra. In \cite{FLS1} the 3-dimensional version of the stellohedron
appears graphically,
 as the domain
and range quotient of the multiplihedra for the complete graphs.
These quotients are the \emph{composihedra} and \emph{cubeahedra}
respectively, but this source does not identify them as stellohedra.
Also in \cite{FLS1} it is claimed without proof that grafted trees
represent these quotients in all dimensions, although the
corresponding trees in that source are associated in error to the
wrong polytope. (We correct the mistake here; compare our
Figures~\ref{pgdiamond} and~\ref{bigpoly_cccc_ext_lit} to Figures 3
and 4 of \cite{FLS1}.)

 In \cite{pilaud2} the authors do actually
prove that the stellohedra for all dimensions are in fact the
cubeahedra of complete graphs (which we will review). Also in
\cite{pilaud2} the stellohedron of dimension $n$ is recognized as
the secondary polytope of pairs of nested concentric $n$-dimensional
simplices. The stellohedra have also been seen as special cases of
signed-tree associahedra in \cite{vincent2}.

\subsection{ Main Results }

Our algebraic results are twelve new graded coalgebras of painted
trees, as described in Theorem~\ref{coalg}. Nine of those contain as
subalgebras the cofree graded coalgebras defined in \cite{FLS3},
shown here in Figure~\ref{pdiamond}. Eight of our new coalgebras
also possess new one-sided Hopf algebra structures, some in multiple
ways, as seen in Theorem~\ref{hopf_alg}.  See
Table~\ref{restab} for reference. Eleven of these new coalgebras are based on the structure types of substitutions (partitional compositions) of
species: nine seen in Theorem~\ref{species} (and the last two by bijection), leaving only one example without a way to calculate numbers of faces. For instance, the composihedra faces are counted by the ordinary generating function
$$(Y \circ L_+\circ L_+ )(x) = \frac{2x}{1 - x + \sqrt{1 - 10x + 17x^2}}   $$
$$=x + 3 x^2 + 11 x^3 + 49 x^4 + 253 x^5 + 1439 x^6 +\dots.$$

\begin{table}
\noindent
 \resizebox{\textwidth}{!}{
\begin{tabular}{||c|c||c|c|c|c|c|c||}
\hline\hline
$C=$& $D=$ & New  & New  & Face  & Composition  & Tubings  & Lifted \\
rooted&rooted& graded & Hopf & poset of & (Substitution)  &  of Graph  & Gen.\\
forests&trees&Coalgebra&Algebra&polytopes& of Species &Suspension& Perm.\\
\hline \hline
Corollas & Corollas & x & x & x& x & & x\\
\cline{2-8}  & Plane & x & x & x  & x & & x\\
\cline{2-8}
 & Weak order & x & x & x & x & x & x\\
\cline{1-8}
Plane & Corollas & x & x & x& x & & x \\
\cline{2-8} trees & Plane & x & x & x  & x &  & x\\
\cline{2-8}
 & Weak order & x & x & x & x & x & x\\
\cline{1-8}
Weakly   & Corollas & x & x & conj. & x && \\
\cline{2-8} ordered  & Plane & x & x & conj.  & x & & \\
\cline{2-8} trees & Weak order & x & conj. & conj. & x & &\\
\cline{1-8}
Weakly  & Corollas & x & x & x & x & x & x\\
\cline{2-8} ordered  & Plane & x & conj. & conj.  &  & & \\
\cline{2-8}forest & Weak order & x & conj. & x & x & x & x\\
 \hline\hline
\end{tabular}
 }\caption{Summary of known results about the sets $\widetilde{C/D}$, organized by tree
types as seen in Figure~\ref{pgdiamond}.}\label{restab}
\end{table}

We show that  eight sequences of our 12 sets of painted trees, with
defined relations, are isomorphic as posets to face lattices of
convex polytopes. Six of these are in the collection of Hopf
algebras just mentioned. Four of these isomorphisms are well known
from previous work: the associahedra, multiplihedra, composihedra
and cubes. Four of our new coalgebras are based on tubings of suspensions of simple graphs.
In Theorem~\ref{perma} we show that weakly ordered
forests grafted to weakly ordered trees are isomorphic to the
permutohedra. In Theorem~\ref{stella1} we show that forests of
corollas grafted to weakly ordered trees are isomorphic to the
star-graph-associahedra, or stellohedra. In Theorem~\ref{stella2}
and Theorem~\ref{stella3} we show that weakly ordered forests
grafted to a corolla are also isomorphic to stellohedra. In Theorem
~\ref{ptera} we show that forests of plane trees grafted to weakly
ordered trees are isomorphic to the fan-graph-associahedra, or
pterahedra. In Theorem~\ref{lift} we show that the stellohedra again
appear as graph-composihedra for the complete graphs.

In Section~\ref{def} we define the sets of trees and the surjective
functions between those sets. In Section~\ref{hopf} we define the
operations on our trees and explain which sets are graded coalgebras
and which are Hopf algebras. We give examples of products,
coproducts, and antipodes.

In Section~\ref{poly} we define a partial ordering of painted trees
and show which of our posets of trees represent combinatorial
equivalence classes of polytopes. In Section~\ref{s:alg} we describe
our Hopf algebra of faces of the stellohedra using graph tubings. In
Proposition~\ref{assofstell} we show that a new, less forgetful,
product on vertices of the stellohedra is associative.

\section{Definitions}\label{def}

Graphs with unlabeled vertices are isomorphism classes of graphs. In
this paper, {\it trees} are unlabeled, connected, acyclic, simple
graphs. A {\it rooted } tree is an oriented tree having one maximal
vertex or node, called the {\it root}. For any node $v$ of a tree,
the edges oriented towards $v$ are called {\it inputs} of $v$ and
the edges leaving from $v$ are called {\it outputs} of $v$. We
denote by ${\mbox {In}(v)}$ the set of inputs of $v$, and by ${\mbox
{Out}(v)}$ the set of outputs of $v$.

We denote by ${\mbox {Nod}(t)}$ the set of nodes of a tree $t$. All
the trees we work with satisfy that $\vert {\mbox {In}(v)}\vert >1$
and $\vert {\mbox {Out}(v)}\vert = 1$. We admit edges which are
linked to a unique node, one of them is the output of the root, the
others are called {\it leaves}. The {\it degree} of a tree is the
number of its leaves minus $1$.

We use the following terms:
\begin{itemize}
\item A {\it plane} tree is a rooted tree satisfying that
 the set ${\mbox {In}(v)}$ is totally ordered, for any node $v$. Sometimes this is also referred to as {\it planar}, and can be equivalently
 satisfied by requiring the leaves to lie in one horizontal line, in order, and the root at a lesser $y$-value.
\item A {\it binary} tree is a rooted tree such that $\vert {\mbox {In}(v)} \vert = 2$, for any node $v$.
\end{itemize}

An example of \emph{plane rooted binary tree}, often called a binary
tree when the context is clear, is the following, where the
orientation of edges is higher to lower on the plane:
\[
    \includegraphics[width = 3in]{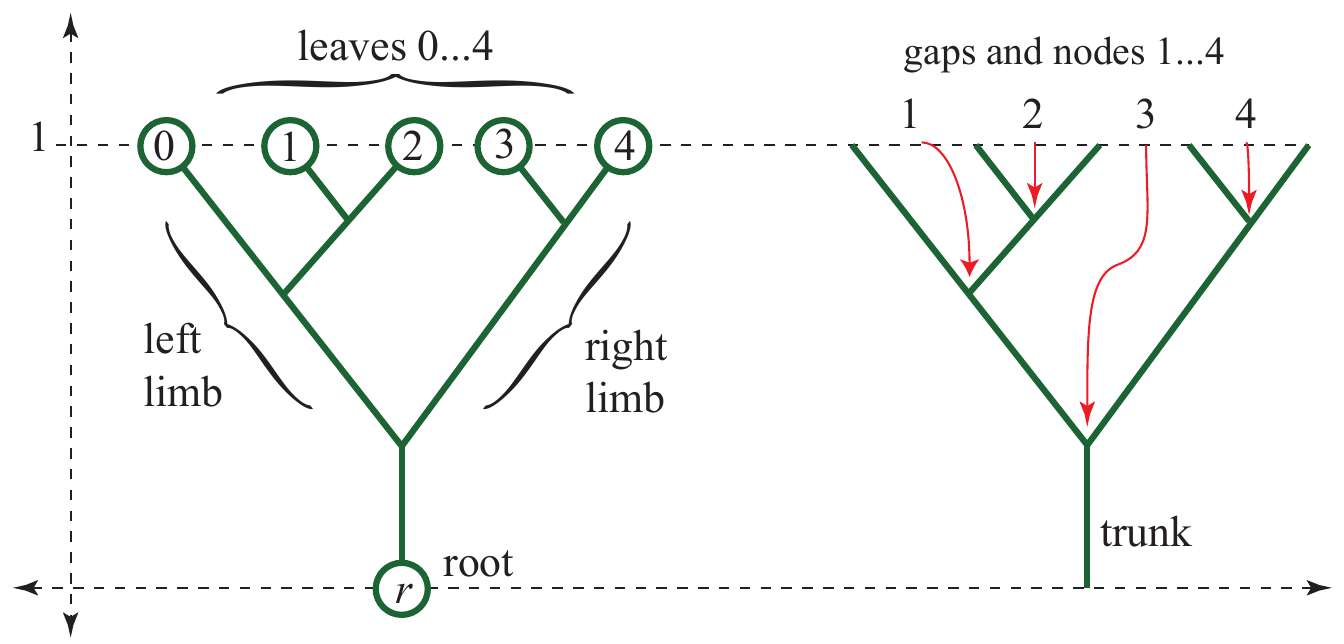}
\]

The leaves are ordered left to right as shown by the circled labels.
The horizontal node ordering  corresponds to the order of \emph{gaps
between leaves}: the $n^{th}$ gap is just to the left of the
$n^{th}$ leaf and the $n^{th}$ node is the one where a raindrop
would be caught which fell in the $n^{th}$ gap.  This ordering is
also described as a depth first traversal of the nodes. Non-leafed
edges are referred to as internal edges. The set of plane rooted
binary trees with $n$ nodes and $n+1$ leaves is denoted $\y_n.$ The
cardinality of these sets are the Catalan numbers:$$|\y_n| =
\frac{1}{n+1}{2n \choose n}.$$

We will also need to consider rooted plane trees whose vertices, or
nodes, have more than two inputs. We denote by ${\mathcal T}_n$ the
set of all plane rooted trees with $n+1$ leaves. The cardinal of
${\mathcal T}_n$ is the $n^{th}$ super-Catalan number (also called
the little Schr\" oder number).

 An $(n+1)$-leaved rooted tree with only
one node (it will have degree $n+2 \ge 3$), or, for $n=0,$ a single
leaf tree with zero nodes, is called a \emph{corolla}, denoted
$\c_n.$ This notation for the (set of one) corolla with $n+1$ leaves
is the same as used for the set of one \emph{left comb} in
\cite{FLS3}. In the current paper we have decided that the corollas
are more easily recognized than the combs.

\subsection{Ordered and painted plane trees}\label{F_sec: bi-leveled trees}
Many variations of the idea of the plane tree have proven useful in
applications to algebra and topology.

\begin{notation}\label{[n]} For any positive integer $n\geq 1$,
we denote by $[n]$ the ordered set $\{1, 2,\dots ,n\}$, and by
$[n]_0$ the set $\{0\}\cup [n]$.\end{notation}

\begin{defi} An \emph{ordered tree}
(sometimes called \emph{leveled}) is a plane rooted tree $t$,
equipped with a vertical linear ordering of ${\mbox {Nod}(t)}$, in
addition to the horizontal one. That is, an ordered tree is a plane
rooted tree $t$ equipped with a bijective map $L: {\mbox {Nod}(t)}
\longrightarrow [\vert{\mbox {Nod}(t)} \vert ]$, which respects the
order given by the vertical order. Clearly $L({\mbox {root}}) =
\vert{\mbox {Nod}(t)} \vert $. \end{defi}

This vertical linear ordering extends the partial vertical ordering
given by distance from the root.  This vertical ordering allows a
well-known bijection between the ordered trees with $n$ nodes,
denoted $\s_n,$ and the permutations on $[n].$

We may draw an ordered tree in three different styles:

\includegraphics[width=4.7in]{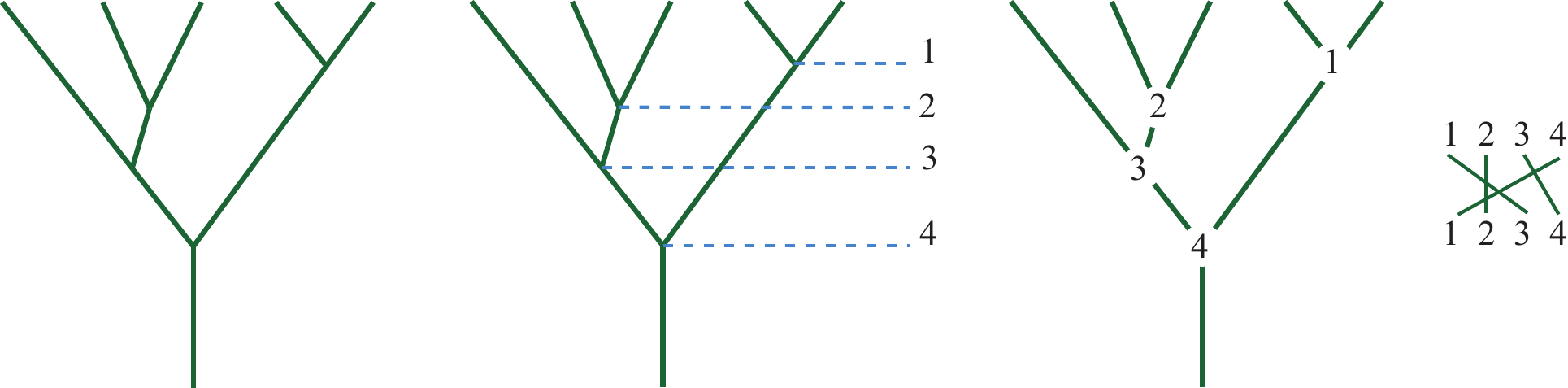}

The corresponding permutation in the above picture is $\sigma  =
(3,2,4,1)$, in the notation
$(\sigma(1),\sigma(2),\sigma(3),\sigma(4).)$


We will also consider \emph{forests} of trees. In this paper, all
forests will be a linearly ordered list of trees, drawn left to
right. This linear ordering can also be seen as an ordering of all
the nodes of the forest, left to right. On top of that, we can also
order all the nodes of the forest vertically, giving a
\emph{vertically ordered forest}, which we often shorten to
\emph{ordered forest}. This initially gives us four sorts of forests
to consider, shown in Figure~\ref{forest_forget}.

\begin{figure}[hb!]
\begin{center}
\includegraphics[width=4.7in]{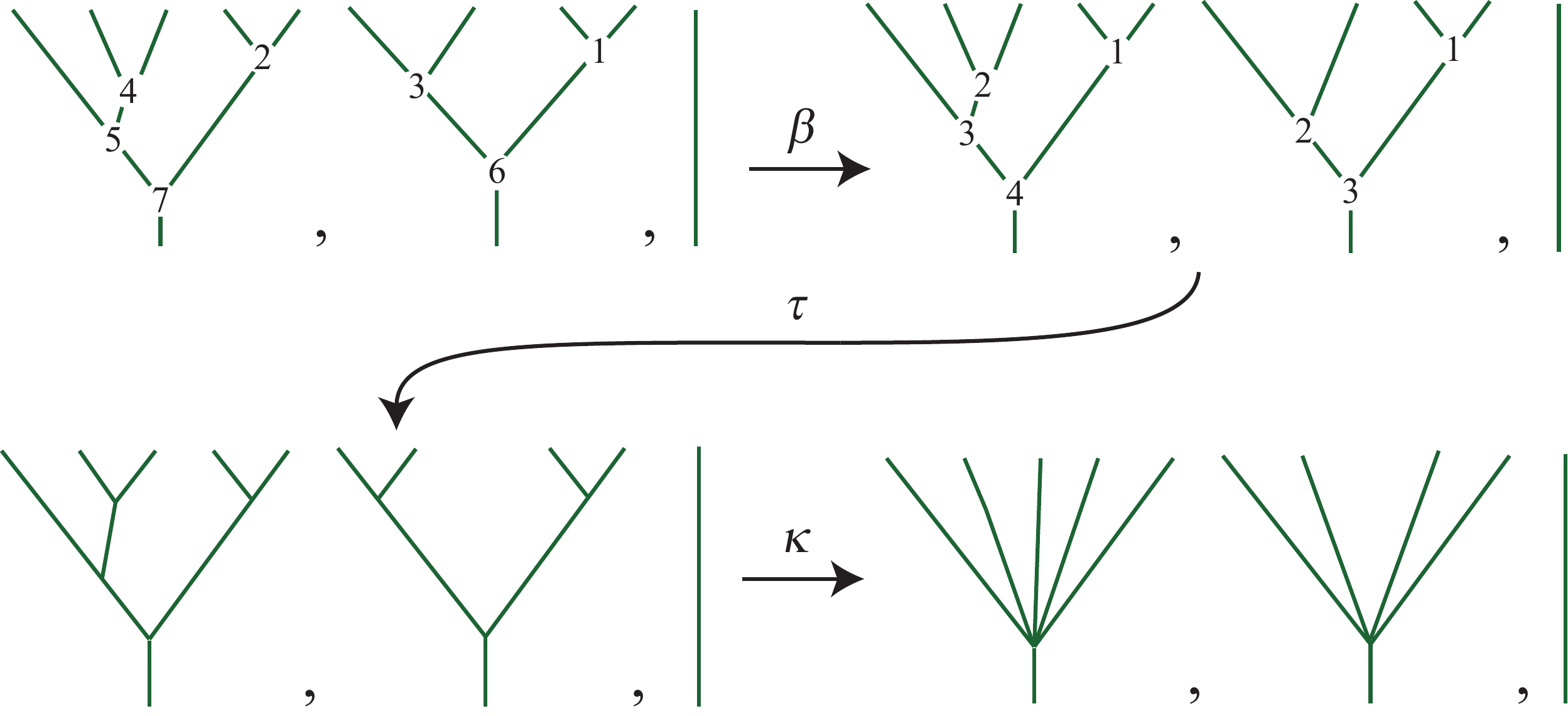}
\caption{Following the arrows: a (vertically) ordered forest, a
forest of ordered trees, a forest of binary trees, and a forest of
corollas.}\label{forest_forget}
\end{center}
\end{figure}

Also shown in Figure~\ref{forest_forget} are three canonical,
forgetful maps between the types of forests.
\begin{defi}
 We define $\beta$ to
be the function that takes an ordered forest $F$ and gives a forest
of ordered trees. The output $\beta(F)$ will have the same list of
trees as $F$, and for a tree $t$ in $\beta(F)$  the vertical order
of the nodes of $t$ will respect the vertical order of the nodes in
$F.$ That is, for two nodes $a,b$ of $t$ we have $a\le b$ in $t$ iff
$a\le b$ in $F.$

We define $\tau$ to be the function that takes an ordered tree and
outputs the tree itself, forgetting all of the vertical ordering of
nodes (except for the partial ordering based on distance from the
root.) We define $\kappa$ to be the function that takes a tree and
gives the corolla with the same number of leaves.
\end{defi}

Note that $\tau$ and $\kappa$ are immediately both functions on
forests, simply by applying them to each tree in turn. Also note
that $\tau$ and $\kappa$ are described in \cite{FLS3}, but that
there $\kappa$ yields a left comb rather than a corolla.

Now we define larger sets of trees that generalize the binary ones.
First we drop the word binary; we will consider plane rooted trees
with nodes that have any degree larger than two. Then, from the
non-binary vertically ordered trees we further generalize by
allowing more than one node to reside at a given level.
 Instead of
corresponding to a permutation, or total ordering, these trees will
correspond to an ordered partition, or weak ordering, of their
nodes.
\begin{defi}
 A \emph{weakly ordered tree} is a plane rooted tree with a weak
 ordering of its nodes that respects the partial order of proximity
 to the root.
\end{defi}

Recall that this means all sets of nodes are comparable--but some
are considered as tied when compared, forming a block in an ordered
partition of the nodes. The linear ordering of the blocks of the
partition respects the partial order of nodes given by paths to the
root.

For a weakly ordered tree with $n+1$ leaves the ordered partition of
the nodes determines an ordered partition of  $S=\{1,\dots,n\}$, as
described in \cite{tonks}. Here we see $S$ as the set of gaps
between leaves. (Recall that a gap between two adjacent leaves
corresponds to the node where a raindrop would eventually come to
rest; $S$ is partitioned into the subsets of gaps that all
correspond to nodes at a given level.) Weakly ordered trees are
drawn using nodes with degree greater than two, and using numbers
and dotted lines to show levels.

\includegraphics[width=\textwidth]{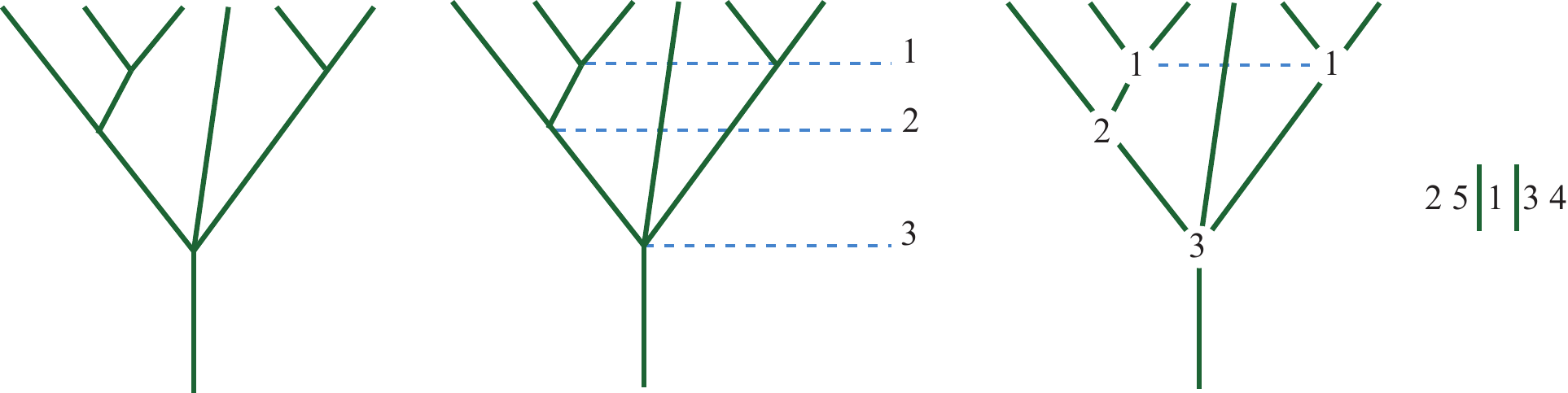}

 The ordered partition corresponding to the above pictures is $(\{2,5\},\{1\},\{3,4\}).$
  Note
that an ordered tree is a (special) weakly ordered tree.


 As well as
forests of weakly ordered trees we also consider weakly ordered
forests.
  This gives
us three more sorts of forests to consider, shown in
Figure~\ref{forest_forget_gen}. As indicated in that figure, the
maps $\beta, \tau$  and $\kappa$ are easily extended to forests of
the non-binary and/or weakly ordered trees: $\beta$ forgets the weak
ordering of the forest to create a forest of weakly ordered trees,
$\tau$ forgets the weak ordering, and $\kappa$ forgets the partial
order to create corollas.

\begin{figure}[hb!]
\begin{center}
\includegraphics[width=4.7in]{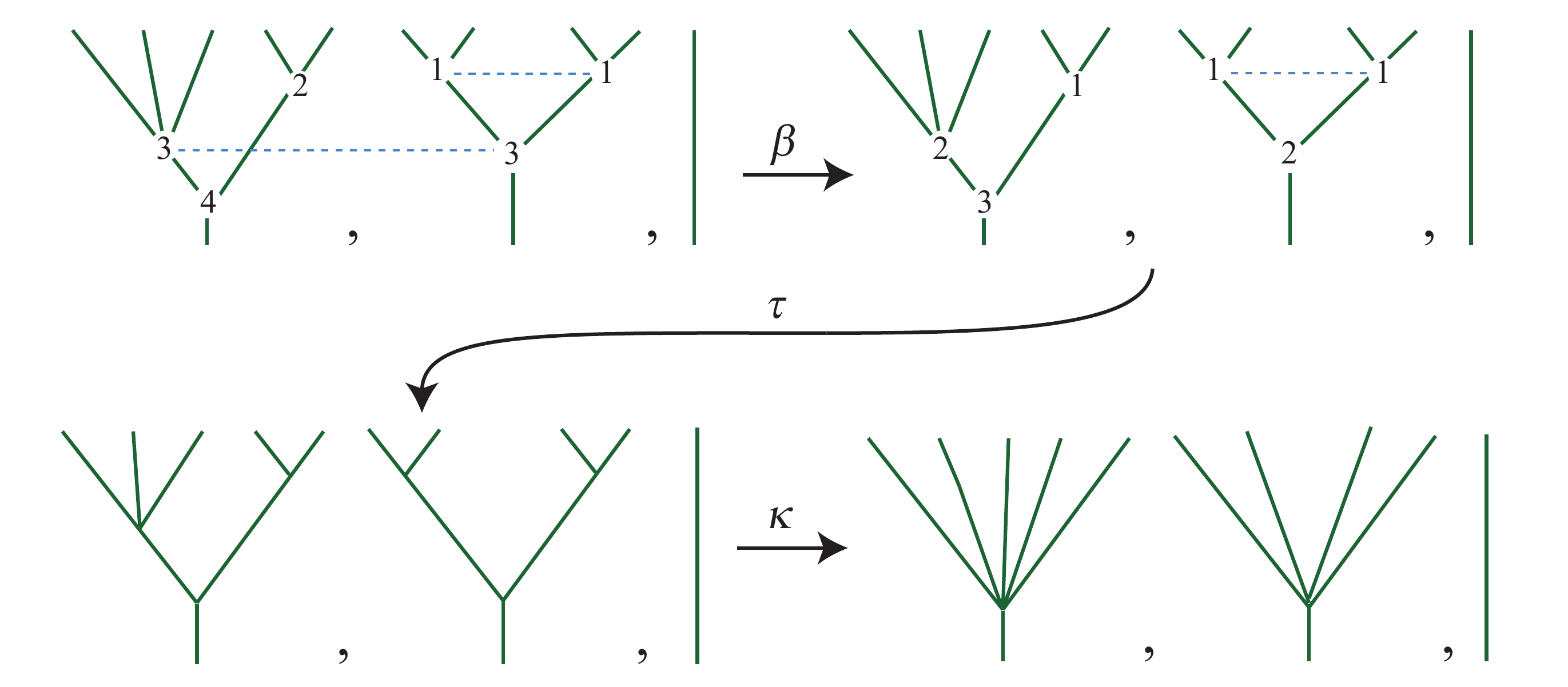}
\caption{Following the arrows: a (vertically) weakly ordered forest,
a forest of weakly ordered trees, a forest of plane rooted trees,
and a forest of corollas. Note that the forests in
Figure~\ref{forest_forget} are special cases of
these.}\label{forest_forget_gen}
\end{center}
\end{figure}

 The trees we focus on in this paper generalize those introduced in \cite{FLS3}. They are constructed by grafting together combinations of ordered
trees, binary trees, and corollas.  Visually, this is accomplished
by attaching the roots of one of the above forests $w\in C$ to the
leaves of one of the above types of trees $v\in D$, but remembering
the originals $w$ and $v$. The result is denoted $w/v \in C/D.$  We
use two colors, which we refer to as ``painted'' and ``unpainted.''
The forest is described as unpainted, and the base tree (which the
forest is grafted to) is painted. At a graft the leaf is identified
with the root, and in the diagram that point is no longer considered
a node, but is rather drawn as a change in color (and thickness, for
easy recognition) of the resulting edge. (Note that in some papers
such as \cite{multi} our mid-edge change in color is described
instead as a new node of degree two.)

With regard to the partial ordering of nodes by proximity to the
root (with the closest to the root being least), we can describe a
painted tree as having a distinguished order ideal of painted nodes.

We refer to the result as a \emph{(partly) painted tree}, regardless
of the types of upper (unpainted) and lower (painted) portions.
Notice that in a painted tree the original trees (before the graft)
are still easily observed since the coloring creates a boundary,
called the \emph{paint-line} halfway up the edges where the graft
was performed. Thus the paint line separates the painted tree into a
single tree of one color and a forest of trees of another color. In
Figure~\ref{pdiamond} we show all 12 ways to graft one of our types
of partially ordered forest with one of our types of tree.

\begin{defi}The maps $\beta, \tau$ and $\kappa$ are now extended to the painted
trees, just by applying them to the unpainted forest and/or to the
painted tree beneath. We indicate this by writing a fraction:
$\displaystyle{\frac{f}{g}}$ for two of our three maps, or the identity map, as
seen in Figure~\ref{pdiamond}. That is, $\displaystyle{\frac{f}{g}}$ indicates
applying $f$ to the forest and $g$ to the painted base tree, for
$f,g\in\{\beta,\tau,\kappa,1\}$.
\end{defi}

\begin{figure}[ht!]
\includegraphics[width=\textwidth]{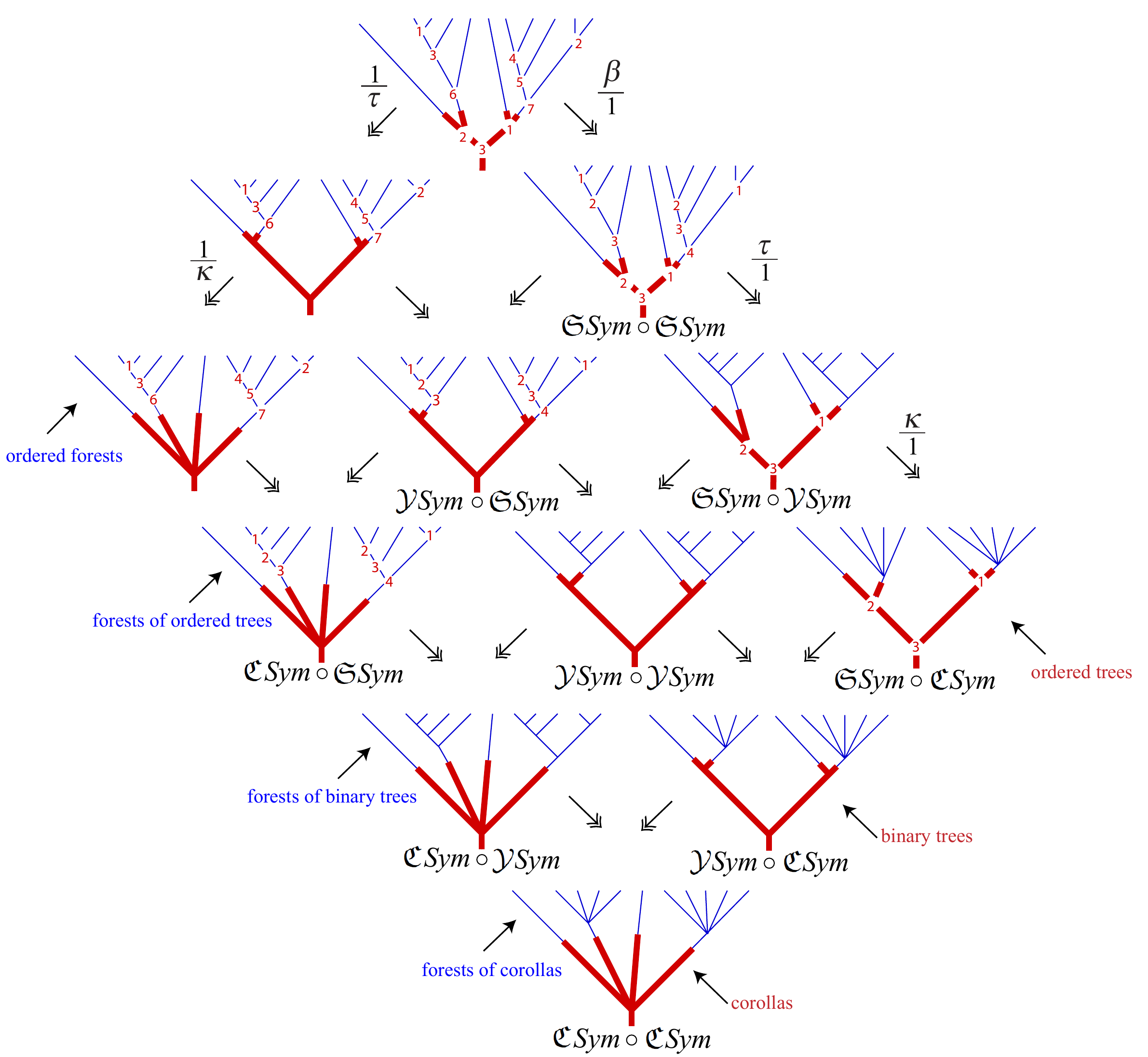}
\caption{Varieties of grafted, painted trees. Each diagonal shares a
type of tree on the bottom (painted) or a type of forest grafted on,
as indicated by the labels. Nine are labeled with the compositions
of coalgebras from \cite{FLS3}.  Trees in this figure correspond to
vertex labels of the 10-dimensional polytopes in sequences whose
3-dimensional versions are shown in
Example~\ref{bigpoly_cccc_ext_lit}. The forgetful maps are shown
with example input and output. Parallel arrows all denote the same
map, except of course that the identity is context
dependent.}\label{pdiamond}
\end{figure}

\subsection{General painted trees.}
Now our definition of painted trees is expanded to include any of our types of
forest grafted to any of our types of tree.
 On top of that we will also permit a further
broadening of the allowed structure of our painted trees. The
paint-line, where the graft occurs, is allowed to coincide with
nodes, where branching occurs. We call it a \emph{half-painted
node}. In terms of the grafting of a forest onto a tree our
description depends on the type of forest. If the forest is weakly
ordered, or is a forest of weakly ordered trees, then we see each
half-painted node as grafting on a single tree at its least node,
after removing its trunk and root. If the forest is only partially
ordered (i.e. of binary trees or corollas) then we see the
half-painted nodes as (possibly) several roots of several trees
simultaneously grafted to a given leaf. For a choice $C$ of forests
and a choice $D$ of trees, the resulting general trees are denoted
$\widetilde{C/D}.$ See the examples in Figures~\ref{pgdiamond}
and~\ref{detail}.

%

\begin{figure}[ht!]
\includegraphics[width=\textwidth]{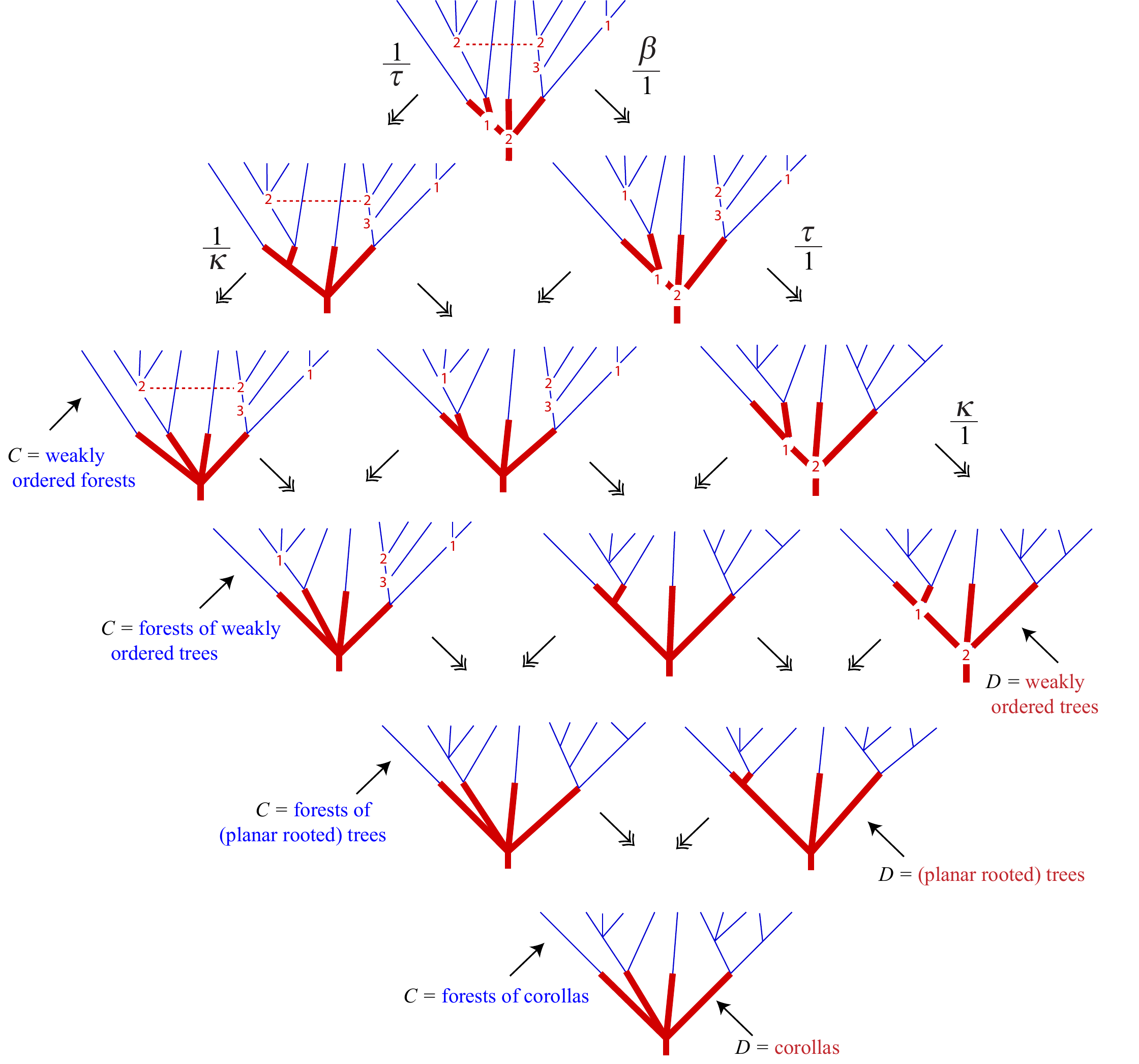}
\caption{General varieties of grafted, painted trees
$\widetilde{C/D}$. For those proven to be polytopes in
Section~\ref{poly}, these correspond to face labels of the
10-dimensional polytope sequences whose 3-dimensional versions are
shown in Example~\ref{bigpoly_cccc_ext_lit}. Parallel arrows all
denote the same map. Note that the trees in Figure~\ref{pdiamond}
are special cases---vertex trees, or minimal in the face
lattice---of the types illustrated in this figure.
}\label{pgdiamond}
\end{figure}

\begin{figure}[ht!]
  \includegraphics[width=4in]{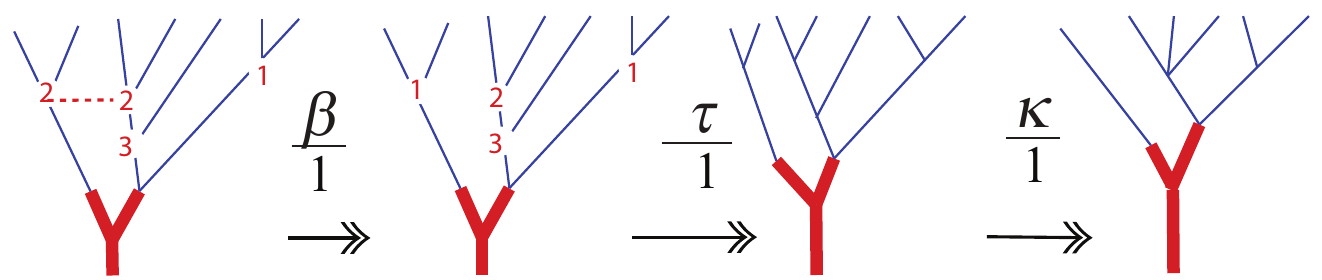}
\caption{Action of the projections, detail from
Figure~\ref{pgdiamond}. At first on the left there is a weakly ordered
forest. Then there are weakly ordered trees, one attached at each half-painted node. Next there are two or three plane trees attached at the paint line;
 we can see the righthand half painted node  as incident with either one or two trees
without any contradiction. Finally there is a forest of 3 separate corollas. }\label{detail}
\end{figure}

For these general painted trees we can again extend the
``fractional'' maps using $\beta, \tau$ and $\kappa.$ We reiterate
from above how the half-painted nodes are interpreted, since that
determines the input for the ``numerator'' map. Specifically
$\displaystyle{\frac{\beta}{g}}$ operates by taking as input for $\beta$ the weakly
ordered forest of trees, one tree for each half-painted node. That
is, $\displaystyle{\frac{\beta}{g}}$ treats the half-painted nodes as being the
location of  a single tree that is grafted on without a trunk. This
description is the same for $\displaystyle{\frac{\tau}{g}}.$ In contrast however,
the map $\displaystyle{\frac{\kappa}{g}}$ takes as input the forest found by
listing all the unpainted trees while assuming each has a visible
trunk, some of which are simultaneously grafted at the same
half-painted node. Examples of these maps are shown in
Figure~\ref{pgdiamond}, where we show 12 general painted trees that
consist of one of the four general types of
 forest and
  one of the three general types of trees.
 Figure~\ref{detail} is a detail from Figure~\ref{pgdiamond} showing how the
  actions of the projections differ.

\section{Hopf Algebras}\label{hopf}

Let $\KK $ denote a field. For any set $X$, we denote by $\KK[X]$
the $\KK$-vector space spanned by $X$. As in \cite{FLS3} we work
over a fixed field of characteristic zero, and our vector spaces
will be constructed by using the sets of trees as graded bases.

Recall from \cite{FLS3} the concept of splitting a tree, given a
multiset of its leaves. Here, modified from an example in
\cite{FLS1}, is a 4-fold splitting into an ordered list of 5 trees:

\includegraphics[width=\textwidth]{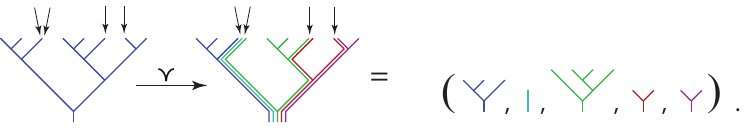}

Also recall the process of grafting an ordered forest to the leaves
of a tree:

\includegraphics[width=\textwidth]{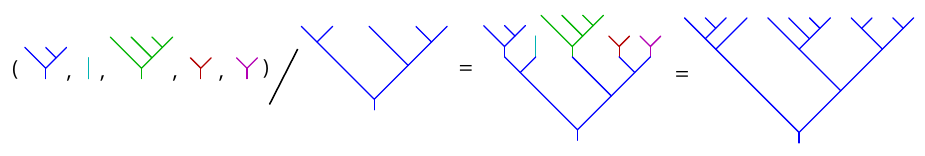}

In \cite{FLS3} there are defined coproducts on nine of the families
of painted trees shown in Figure~\ref{pdiamond} (the ones with
labels denoting their membership in a composition of coalgebras).
Eight of these, all but the composition of coalgebras $\ssym \circ
\ssym$, are shown to possess various Hopf algebraic structures in
\cite{FLS3}. Now we show which of those structures can be extended
to our generalized painted trees in Figure~\ref{pgdiamond}.

The Hopf algebras we are interested in first are the algebra of
corollas, called $\c Sym$ and shown to be identical to the divided
power Hopf algebra in \cite{FLS3}; second the algebra of rooted
planar binary trees $\y Sym$ which is known as the Loday-Ronco Hopf
algebra, and finally the algebra of rooted planar trees
$\y\widetilde{Sym}.$ The latter is the Hopf algebra of faces of the
associahedra as described in \cite{chap}, and in terms of graph
tubings in \cite{ForSpr:2010}.

The coproducts and products are all defined in \cite{FLS3} using
subscripts: the element of the vector space is $F_w$ where $w$ is a
tree of the given type. The coproduct is defined by splitting:

$$\Delta(F_w) = \sum_{w\psplit(w_0,w_1)} F_{w_0}\otimes F_{w_1}, $$
where the sum is over all ways to split the tree $w$ at one leaf; so
has $n$ terms. Multiplication on the left is defined by splitting
the left operand and grafting to the right operand:

$$F_w\cdot F_v = \sum_{w\psplit(w_0,\dots, w_n)} F_{(w_0,\dots, w_n)/v}, $$
where the sum is over all ways to split the tree $w$ at a multiset
of $n-1$ leaves (where $n$ is the number of leaves of $v.$)

We will often eliminate the subscript notation and simply draw the
basis element. For example, here is the coproduct in $\y Sym$:

\includegraphics[width=5.5in]{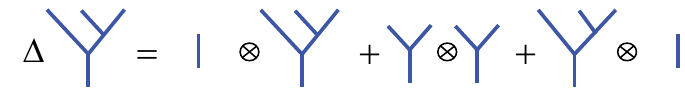}

 Here is how to multiply two trees in $\y Sym$:

\includegraphics[width=5.5in]{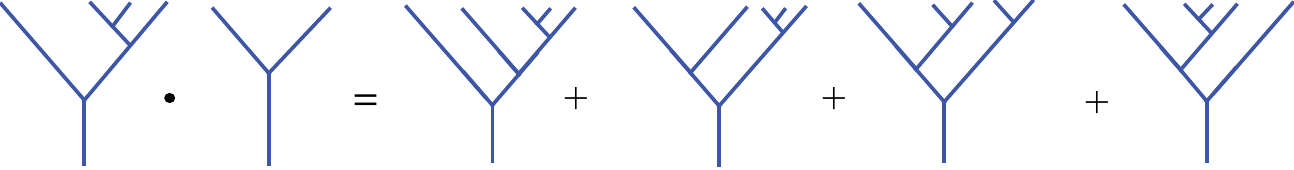}

For more examples see \cite{FLS3}.

 Given any of the 12 types of painted tree from
Figure~\ref{pgdiamond}, we get a graded vector space where trees
with $n$ leaves comprise the basis of degree $n-1.$ The basis of
degree 0 is the single-leaved painted tree--this is the same for
each of the 12 cases. The degree 1 basis is also identical for all
12 cases: the three painted trees with 2 leaves. The 12 cases differ
when it comes to the degree 2 bases, as seen in
Figure~\ref{2s_labels}. Note that while most of the trees in
Figure~\ref{pdiamond} can be seen as coming from a composition of
coalgebras, as labeled, the general trees in Figure~\ref{pgdiamond}
cannot since the unpainted forests can be grafted in multiple ways.
However, they can often still possess a coproduct, given by
splitting the trees leaf to root. Splitting a tree of a given type
always produces two trees of that same type. The weakly ordered
trees and weakly ordered forests can be split into two weakly
ordered trees or forests. In fact we have the following:

\begin{theorem}\label{coalg}
The action of splitting trees leaf to root makes each of the tree
types in Figure~\ref{pgdiamond} into the basis of a graded
coassociative coalgebra.
\end{theorem}
\begin{proof} The coproduct of a basis element is the sum of pairs of trees which are
formed by splitting at each leaf. Note that the degrees of the pairs
each sum to $n-1.$ Coassociativity is seen by comparing both orders
of applying the coproduct to the result of choosing any two leaves
at which to split at the same time:
$$(\Delta\otimes 1)(\Delta(F_w)) = (1\otimes\Delta)(\Delta(F_w)) = \sum_{w\psplit(w_0,w_1,w_2)} F_{w_0}\otimes F_{w_1} \otimes F_{w_2}.$$
\end{proof}
The compositions of coalgebras labeled in Figure~\ref{pdiamond} are
subcoalgebras of the corresponding generalizations in
Figure~\ref{pgdiamond}. We will denote these larger coalgebras by
${\mathcal E} = \widetilde{C/D}$ where $C$ and $D$ are the
corresponding sets of trees.  For example here is a coproduct in
$\widetilde{C/D};$ in this picture the painted trees could be any of
our twelve varieties.

\includegraphics[width=5.5in]{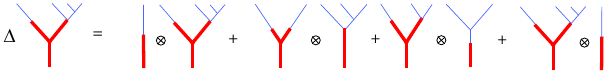}

Next we point out the \emph{actions} of certain Hopf algebras on
many of our coalgebras of generalized painted trees.

\subsection{Hopf algebra modules}
Using the same operations of splitting and grafting, we can often
show that the Hopf algebras ${\mathcal D} = \y S\widetilde{ym},$ and
${\mathcal D} = \csym$ possess actions on painted trees which make
our coalgebras of the latter into ${\mathcal D}$-modules.

\begin{theorem} For each
${\mathcal E} = \widetilde{C/D}$ with $C$ being the planar trees or
corollas, the coalgebra ${\mathcal E}$ with basis $C/D$ is
respectively a ${\mathcal D}=\y S\widetilde{ym}$-module coalgebra or
${\mathcal D}=\csym$-module coalgebra.
\end{theorem}
\begin{proof}We show that ${\mathcal E}$ is an associative left
module, and that the action of ${\mathcal D}$ (denoted $\star$)
commutes with the coproducts as follows: $\Delta_{\mathcal E}(d\star
e) = \Delta_{\mathcal D}(d)\star\Delta_{\mathcal E}(e).$ We consider
the action on basis elements. The action of a planar tree $d\in \y
S\widetilde{ym}$ (or corolla  in $\csym$ respectively) on a painted
tree $e$ involves splitting $d$ and grafting the resulting forest
onto the leaves of $e$.  In the case of $C$ being corollas, the
result of the grafting is then subjected to the application of
$\kappa/1.$ Note that this application of $\kappa$ is equivalent to
the composition in the operad of corollas, as pointed out in
\cite{FLS3}. For example, where $C$ is the set of corollas:

\includegraphics[width=3in]{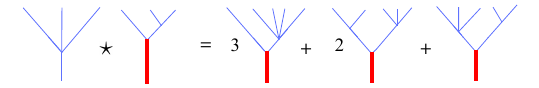}

Note in the above example that six terms result from choosing any
two splits of the three leaves in the corolla. After applying
$\kappa,$ there are duplicates as enumerated by the coefficients.

 The associativity of the action is then
straightforward to show on basis elements: given three layers of
trees (the bottom layer is the painted tree) the result does not
depend upon the order in which one makes the grafts.

The commutativity property is also straightforward on basis
elements. Recall that the coproduct is applied linearly to each term
in a sum, on the left side of the equation: $\Delta_{\mathcal
E}(d\star e)$. Also recall that the action of a tensor product on a
tensor product is performed componentwise: $(x\otimes y) \star
(z\otimes w) = (x\star z)\otimes(y\star w)$ on the right side. Thus
each term on the left-hand side is a pair of painted trees, formed
by splitting after grafting a splitting of $d$ onto $e$. That pair
is found on the right-hand side: all the splits are just performed
before the grafting occurs. See the following example, where we pick
out the matching terms.

\includegraphics[width=\textwidth]{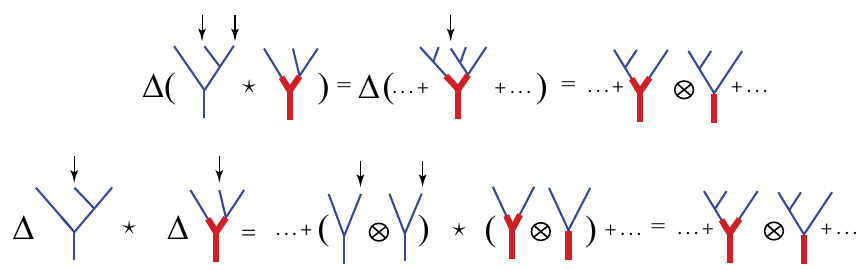}
\end{proof}

\begin{theorem} For each
${\mathcal E} = \widetilde{C/D}$ for $D$ being the planar trees or
corollas, and $C$ being either corollas, planar trees, or weakly
ordered trees, the coalgebra ${\mathcal E}$ with basis $C/D$ is
respectively a ${\mathcal D}=\y S\widetilde{ym}$-module coalgebra or
${\mathcal D}=\csym$-module coalgebra.
\end{theorem}
\begin{proof}
The same features need to be checked as in the proof of the previous
theorem, which is straightforward. Now however the action of the
Hopf algebra is on the right, so the product $e\star d$ involves
splitting $e\in {\mathcal E}$ and then grafting to $d\in {\mathcal
D}.$
\end{proof}

The fact that one sided Hopf algebras exist for the generalized
painted trees follows from the use of the maps $\beta, \tau,$ and
$\kappa$ defined on trees and forests. We recall the definition of
the sort of map we need from \cite{FLS3}. We let ${\mathcal D}$ be
one of our connected graded Hopf algebras with product $v\cdot w.$

\begin{definition}
A map $f:{\mathcal E}\to{\mathcal D}$ of connected graded coalgebras
is a \emph{connection} on ${\mathcal D}$ if ${\mathcal E}$ is a left
(right) ${\mathcal D}$-module coalgebra and $f$ is both a coalgebra
map and a module map:
$$(f\otimes f)\Delta_{\mathcal E}(e) = \Delta_{\mathcal D}f(e) ~~and~~ f(d\star e) =  d\cdot f(e).$$
\end{definition}

We have examples of connections $f$ using the maps $\tau$ and
$\kappa$. If the target is corollas, we apply first $\tau$  and then
$\kappa$ to a painted tree $w$. Then we forget the painting and
apply $\kappa$ once more. The result is just a corolla with the same
number of leaves as $w.$ If the target is planar trees we apply
$\tau$ only, and then forget the painting. The result is a planar
tree with the same branching structure as $w.$ These example
connections are seen to be coalgebra and module maps by inspecting
their action on basis elements: the result is the same if $f$ is
applied before or after splitting and grafting.

Here is an example connection from planar trees over weakly ordered
trees to planar trees:

\includegraphics[width=4in]{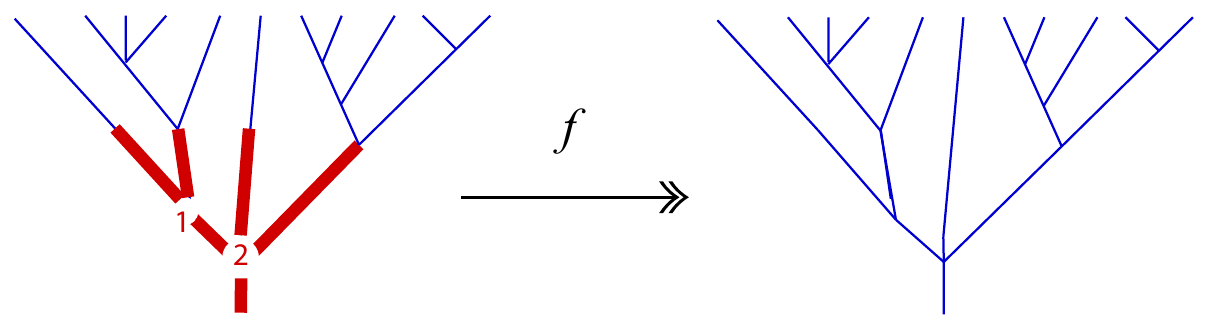}

Here is an example connection from corollas over weakly ordered
trees to corollas:

\includegraphics[width=4in]{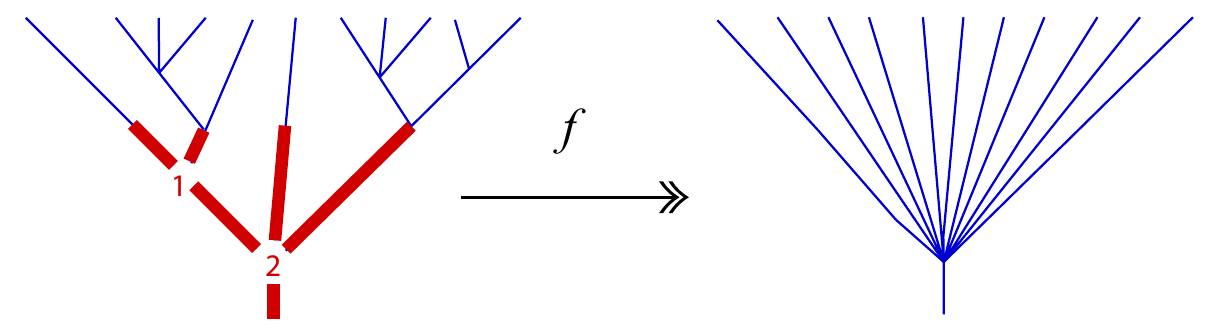}

Here is an example connection from corollas over planar trees to
planar trees:

\includegraphics[width=4in]{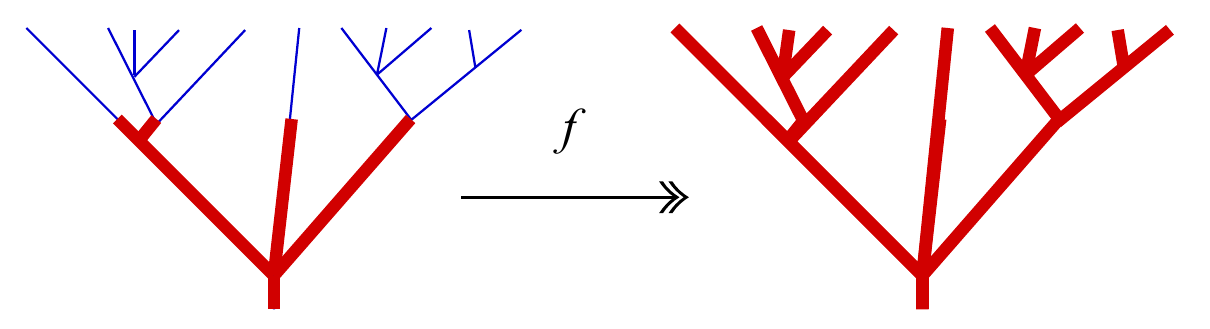}

\begin{theorem}\label{hopf_alg}
Consider the coalgebras ${\mathcal E}$ with graded bases the painted
trees $C/D$ with $C$ (top) consisting of forests of planar trees or
forests of corollas; and those with $D$ being planar trees or
corollas (bottom) and with forests of planar trees, corollas or
weakly ordered trees on top. Each of these eight coalgebras are
one-sided Hopf algebras (they possess a one-sided unit and one-sided
antipode.)
\end{theorem}
\begin{proof}
We rely on Theorem 4.1 of \cite{FLS3}, which states that when there
is a connection $f:{\mathcal E}\to{\mathcal D}$ then ${\mathcal E}$
is a Hopf module and a comodule algebra over ${\mathcal D},$ and
also a one-sided Hopf algebra in its own right. For those coalgebras
$\widetilde{C/D}$ with planar trees or corollas on top (as $C$), the
connection $f$ is the map to $C,$ the planar trees or corollas
respectively. Note that in this case the product  $\cdot:{\mathcal
E}\otimes{\mathcal E}\to {\mathcal E}$ will be on the left: for
$e,e' \in {\mathcal E}$ we have, from the proof of Theorem 4.1 of
\cite{FLS3}, that $e\cdot e' = f(e) \star e'.$

 For those coalgebras
$\widetilde{C/D}$ with planar trees or corollas on bottom (as $D$),
the connection $f$ is the map to $D,$ the planar trees or corollas
respectively. Note that in this case the product  $\cdot:{\mathcal
E}\otimes{\mathcal E}\to {\mathcal E}$ will be on the right: for
$e,e' \in {\mathcal E}$ we have, from the proof of Theorem 4.1 of
\cite{FLS3}, that $e\cdot e' = e \star f(e').$

We note that the one-sided unit $\eta=\eta(1)$ for either left or
right products is the painted corolla with one leaf. The counit
$\epsilon$ is a projection from ${\mathcal E}$ onto the base field:
its value is the coefficient of the painted corolla with one leaf.
\end{proof}

Notice that some of our structures (the four painted trees that use
no ordered trees) are Hopf algebras in two different ways. One has a
left-sided unit and the other has a right-sided unit. Here is an
example of the product in the Hopf algebra with left-side unit on
the coalgebra $\widetilde{C/C}$ for $C$ the binary trees:

\includegraphics[width=5.5in]{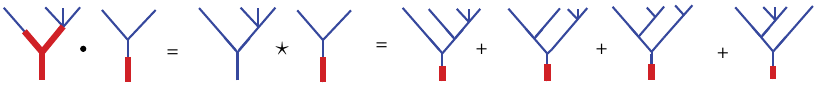}

Here is an example of the product in the Hopf algebra with
right-side unit on the coalgebra $\widetilde{C/C}$ for $C$ the
binary trees:

\includegraphics[width=5.5in]{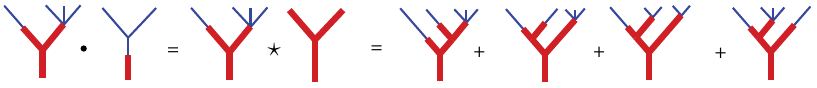}

Finally, we exhibit some antipodes $S:
\widetilde{C/D}\to\widetilde{C/D}$. These are guaranteed to exist in
connected graded bialgebras. For example, for left multiplication,
the left antipode may be calculated with the recursive formula
developed in \cite{FLS3}, section 4: $$ S(\eta) =
\eta;~~~~~~~~~~~~~~~~~~~ S(e) = -\eta\cdot e - \sum S(e_1)\cdot e_2
$$ where the sum is over all splittings of $e$ into two painted
trees $e_1$ and $e_2$, both with more than a single leaf.

Here are some examples, where the antipodes of larger trees can be
found recursively using antipodes of their splittings:

\includegraphics[width=3.5in]{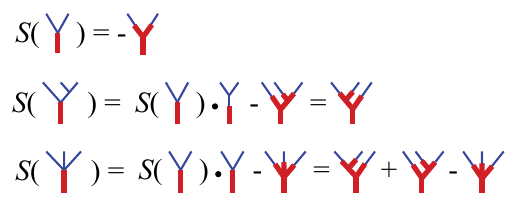}

\subsection{Enumeration using species composition}
Next we calculate the graded dimensions $d_n = $ dim$(\widetilde{C/D})_n$.
 For each coalgebra $\widetilde{C/D}$ this dimension is the sequence of numbers of basis elements.
 We find generating functions for the
sequences of cardinalities of sets of trees, for most of the 12 tree varieties. In each case we want to count the trees themselves, but achieve this by
 finding the cardinality of a species $G$  which has structures those trees with labeled leaves. Recall that the set of \emph{structure types}
 of a species is the collection of unlabeled versions of the combinatorial objects, as described in \cite{species}.  That source is also our reference for \emph{species substitution},
 also known as \emph{(partitional) composition}.

However, since $n$ labels are assigned to the $n$ leaves in a row left to right, counting the number of structures of one of our species $G$
just means overcounting the structure types by a factor of $n!$. Thus
 the exponential generating function $\sum|G(n)|x^n/n!$ of the species is the same as the ordinary generating function of the graded dimension: $\sum d_nx^n.$
Note that for a given type of tree, the graded dimension is also the numbers of faces of the polytopes in the sequence, plus one for the polytope itself.
 The first few numbers in each sequence can be checked by hand-counting the faces
in Example~\ref{bigpoly_cccc_ext_lit} and Figure~\ref{2s_labels}.

Consider the combinatorial species $L_+$ of non-empty lists, $F$ of
weakly ordered trees with labeled leaves, and $Y$ of plane rooted
trees with labeled leaves. The exponential generating functions of
two of these have nice closed forms: they are $\displaystyle{L_+(x)
= \frac{x}{1-x},}$ whose derivatives at 0 give the factorials; and
$$\displaystyle{Y(x) = \frac{2x}{1 + x + \sqrt{1 - 6x + x^2}},}$$
whose derivatives at 0 give the super-Catalan numbers (numbers of faces of associahedra, A001003).

For the species $F$ of weakly ordered trees with $n$ labeled leaves,
the structure types are counted by the Fubini numbers $f(n-1)$,
since the number of gaps between leaves is $n-1.$ We are forced to
use the exponential generating function:
$$F(x) = \sum_{n=1}^\infty  f(n-1)x^n = x+ x^2 + 3x^3+ 13x^4 + 75x^5 +\dots$$
Notice that $F(x)$ is an exponential generating function for the series  $n!f(n-1),$ which is the sequence counting the
number of ordered trees with $n$ labeled leaves.

Now we are able to set up ordinary generating functions for the sets
$\widetilde{C/D},$ for the cases in which $\widetilde{C/D}$ is
describable as the structure types of a composition of
species. These are listed in Table~\ref{tspecies}. In each case the
middle species is $L_+.$ Here the nonempty lists are of one or
several trees grafted at the paint line to single leaves. For
instance, when $C = D =$ corollas, then the set $\widetilde{C/D}$ of trees is the set of
structure types of the composite species $L_+\circ L_+\circ L_+.$ Thus the exponential generating function is

$$(L_+\circ L_+\circ L_+) (x) = \frac{\frac{x}{1-x}}{1-2\frac{x}{1-x}} = \frac{x}{1-3x} = x + 3x^2 + 9x^3 + 27x^4 +\dots$$

Notice that this is the ordinary generating function of the faces of the cubes. In our
examples, the ordinary generating function of the structure types is
the exponential generating function of the species, since the labeled leaves always contribute
a factorial.

For the trees whose half-painted nodes are seen as grafting on a single tree at its least node,
after removing its trunk and root, we instead compose with $2G-x$. The factor of two accounts for the options of grafting with or without a trunk,
and subtracting $x$ corrects for the over-counting of grafting on a tree with one leaf--with or without a trunk are both the same.
  These facts just recounted imply the following:

\begin{theorem}\label{species} The sets $\widetilde{C/D}$ for $C$ and $D$ forests of corollas,
plane trees, or weakly ordered trees are
in bijection with the structure types of the substitutions of species as listed in Table~\ref{tspecies}.
\end{theorem}
\begin{table}
\noindent
 \resizebox{5.5in}{!}{
\begin{tabular}{||c|c|c|c|c||}
\hline\hline
$C=$& $D=$ & Species  & Sequence of & OEIS \\
rooted&rooted&Composition & dimensions & entry\\
\hline \hline
Corollas & Corollas & $L_+\circ L_+\circ L_+$ & $1, 3, 9, 27, 81, \dots $ & A000244\\
\cline{2-5}  & Plane & $Y\circ L_+\circ L_+$ & $1, 3, 11, 49, 253, 1439, \dots $ & \\
\cline{2-5}
 & Weak order & $F \circ L_+\circ L_+$ & $1, 3, 11, 51, 299, \dots$ & A007047\\
\cline{1-5}
Plane & Corollas & $L_+\circ L_+\circ Y = L_+\circ(2Y-x)$ & $1, 3, 11, 45, 197, 903, \dots$ & A001003\\
\cline{2-5} trees & Plane & $Y\circ L_+\circ Y = Y\circ(2Y-x)$ & $1, 3, 13, 67, 381, 2311, \dots$ &  \\
\cline{2-5}
 & Weak order & $F \circ L_+\circ Y = F\circ(2Y-x)$ &$1, 3, 13, 69, 427, \dots$ & \\
\cline{1-5}
 Weakly   & Corollas & $L_+\circ(2F-x)$ & $1, 3, 11, 49, 265, 1739, \dots$ & \\
\cline{2-5} ordered  & Plane & $Y\circ(2F-x)$ & $1, 3, 13, 71, 449, \dots$ & \\
\cline{2-5} trees & Weak order & $F\circ(2F-x)$ & $1, 3, 13, 73, 495, \dots$ & \\
 \hline\hline
\end{tabular}
 }\caption{Compositions (substitutions) of species whose structure types enumerate the sets $\widetilde{C/D}$, with reference to the OEIS \cite{Slo:oeis} if known.}\label{tspecies}
\end{table}

For example, when $C = $ forests of corollas and $D=$ weakly ordered
trees we have:

 $$(F \circ L_+\circ L_+ )(x) = \frac{x}{1-2x} +\left(\frac{x}{1-2x}\right)^2 +3\left(\frac{x}{1-2x}\right)^3
+13\left(\frac{x}{1-2x}\right)^4+75\left(\frac{x}{1-2x}\right)^5
+\dots$$
$$ = x + 3 x^2 + 11 x^3 + 51 x^4 + 299 x^5 + \dots$$
which has coefficients the numbers of faces of the stellohedra, which is sequence A007047 in the OEIS \cite{Slo:oeis}.

Another example, when $C = $ forests of plane trees and $D=$ corollas we
have:

 $$(L_+ \circ L_+\circ Y )(x) = \frac{\frac{2x}{1 + x + \sqrt{1 - 6x + x^2}}}{1-2\frac{2x}{1 + x + \sqrt{1 - 6x + x^2}}} = \frac{2 x}{\sqrt{x^2 - 6 x + 1} - 3 x + 1}= \frac{2}{1 + x + \sqrt{1 - 6x + x^2}}-1$$
$$ = x + 3 x^2 + 11 x^3 + 45 x^4 + 197 x^5 + 903 x^6 + \dots$$

Notice here that the numbers of faces of the associahedra (A001003) are seen as coefficients, but shifted as expected.

For another example, we count the faces of the composihedra:
$$(Y \circ L_+\circ L_+ )(x) = \frac{2\frac{x}{1-2x}}{1 + \frac{x}{1-2x} + \sqrt{1 - 6\frac{x}{1-2x} + \left(\frac{x}{1-2x}\right)^2}} = \frac{2x}{1 - x + \sqrt{1 - 10x + 17x^2}}  $$
$$=x + 3 x^2 + 11 x^3 + 49 x^4 + 253 x^5 + 1439 x^6 +\dots.$$
which does not yet appear in the OEIS.

For a final example, we show the generating function for weakly ordered trees grafted onto corollas.

$$L_+\circ(2F-x) = \frac{x+ 2x^2 + 6x^3+ 26x^4 + 150x^5+\dots}{1 -(x+ 2x^2 + 6x^3+ 26x^4 + 150x^5+\dots) }$$
$$= x + 3 x^2 + 11 x^3 + 49 x^4 + 265 x^5 + \dots$$

Note that as seen in Table~\ref{restab} there is only one of our tree types that we cannot yet count.
 That is because in the three cases of grafting on a weakly ordered forest the resulting structures do not fit the
definition of species composition--there is an ordering structure that encompasses multiple ``substituted'', or grafted on, trees.
For two of those cases however the resulting structures
are in bijection with well known species, leaving the one open case.

In the next section we will define twelve poset structures, one for each of our sets of trees. In most cases this poset is equivalent to the face poset of a polytope.
\newpage

\section{Polytopes}\label{poly}

\subsection{Partial ordering of nodes and gaps.} Each painted tree $t$ (any of our 12
types of painted tree with $n+1$ leaves) determines a
partial ordering of a partition $\pi(t)$ of the set $\{0,\dots,n\}$. The numbers from 1 to $n$ denote the gaps between leaves, and 0 is adjoined.
One part of the partition $\pi(t)$ consists of 0 and the gaps of $t$ which end at half-painted nodes. Other parts of $\pi(t)$ consist of
all gaps which end at nodes sharing a given vertical level, in the original tree and forest.
 Elements in a part of the partition $\pi(t)$ are considered tied, that is, equal in a weak partial ordering of the set $\{0,\dots,n\}$ of gaps between leaves plus 0.
We observe that 1) all half-painted nodes must be forced to remain at
the same level, that is, tied in a
weak order; and 2) that gaps ending at nodes above the paint line will never
surpass gaps ending at half painted nodes, and neither of the former will surpass
gaps ending at painted nodes.

For example consider the tree $t\in \widetilde{C/D}$ , with $C=$ weakly ordered forests, $D=$ ordered trees, and $n=10,$  at the top of Figure~\ref{pgdiamond}. This tree is
also shown in part (a) of Example 4.6 with $\pi(t) = \{\{0,3,9\},\{1\},\{2,3,7\},\{5,6\},\{8\},\{10\}\}$.
The partial ordering of $\pi(t)$ here is a single chain: $\{10\} < \{2,3,7\} < \{8\} < \{0,3,9\} < \{1\} < \{5,6\}.$
 We can consider the gaps 2,3, and 7, for instance, as equivalent or tied in
the weak order of gaps. Note that all the parts of $\pi$ are comparable here because of the type $\widetilde{C/D}$.

Now we can define 12 separate posets whose elements are painted trees: one
poset on each of our 12 types of painted trees shown in
Figure~\ref{pgdiamond}. Note that the simplest painted tree with $n$
leaves has one half-painted node: $n$ single leafed unpainted trees
all grafted to a painted trunk, the node coinciding with the paint
line. This \emph{half-painted corolla} can be interpreted as one of
any of the 12 painted tree varieties, and it will be the unique
maximal element in all 12 posets. Its ordered partition $\pi$ has only one part.

\begin{definition}
 Given two painted trees $s$ and $t$ that are of the same painted type (i.e.
they share the same types of tree and forest, below and above the
paint line) we define the \emph{painted growth preorder}, (shortly proven to be a partial order) in which :
$$s \prec  t$$ if $s=t$ or if $s$ is formed from $t$ using a series of pairs $(a,b)_i$ of the following two moves, each pair performed in the following order:
\end{definition}
\begin{enumerate}
\item[a)] \emph{growing} internal edges of $t:$ introducing new internal edges or increasing the length of some internal edges
(either painted or unpainted). This growth results in a
 refinement of the (weak) partial order on the set $\{0,\dots,n\}$ of gaps between leaves plus 0, by adding relations between
previously incomparable (or tied) elements.

\item[b)] ...followed by  \emph{forgetting} any superfluous structure introduced by the edge
growing. This is described precisely by taking the tree that results
from growing edges, and applying to it the forgetful map (from the
set of $\beta, \tau, \kappa, 1$ and their fractions and
compositions) that is needed to ensure that the result is in the
original type of the painted tree $t$.

 For example if the original type of $t$ had
weakly ordered forests grafted to weakly ordered trees, we only
apply the identity. However if $t$ originally was a forest of weakly
ordered trees grafted to a weakly ordered tree we should apply
$\displaystyle{\frac{\beta}{1}}.$

\end{enumerate}

 Note that internal edges can grow where there was no
internal edge before, such as at a half-painted node or any node
that had degree larger than three. Note also that the rules for
painted trees must be obeyed by the growing of edges--for instance an
unpainted edge cannot grow from a completely painted node (and vice
versa), and if some painted edges are grown from a half-painted node this cannot result in a node where painted and unpainted edges exist at the same level.

To find examples of (non-covering) relations in the 12 posets see the
trees in respective locations of figure~\ref{pdiamond} and
Figure~\ref{pgdiamond}: certain of the latter are comparable to the former
in the same positions. (Not all: for instance the uppermost respective trees in these two figures are not comparable.) Several more covering relations for some of
our 12 classes of general painted trees are shown next:

                  \includegraphics[width=\textwidth]{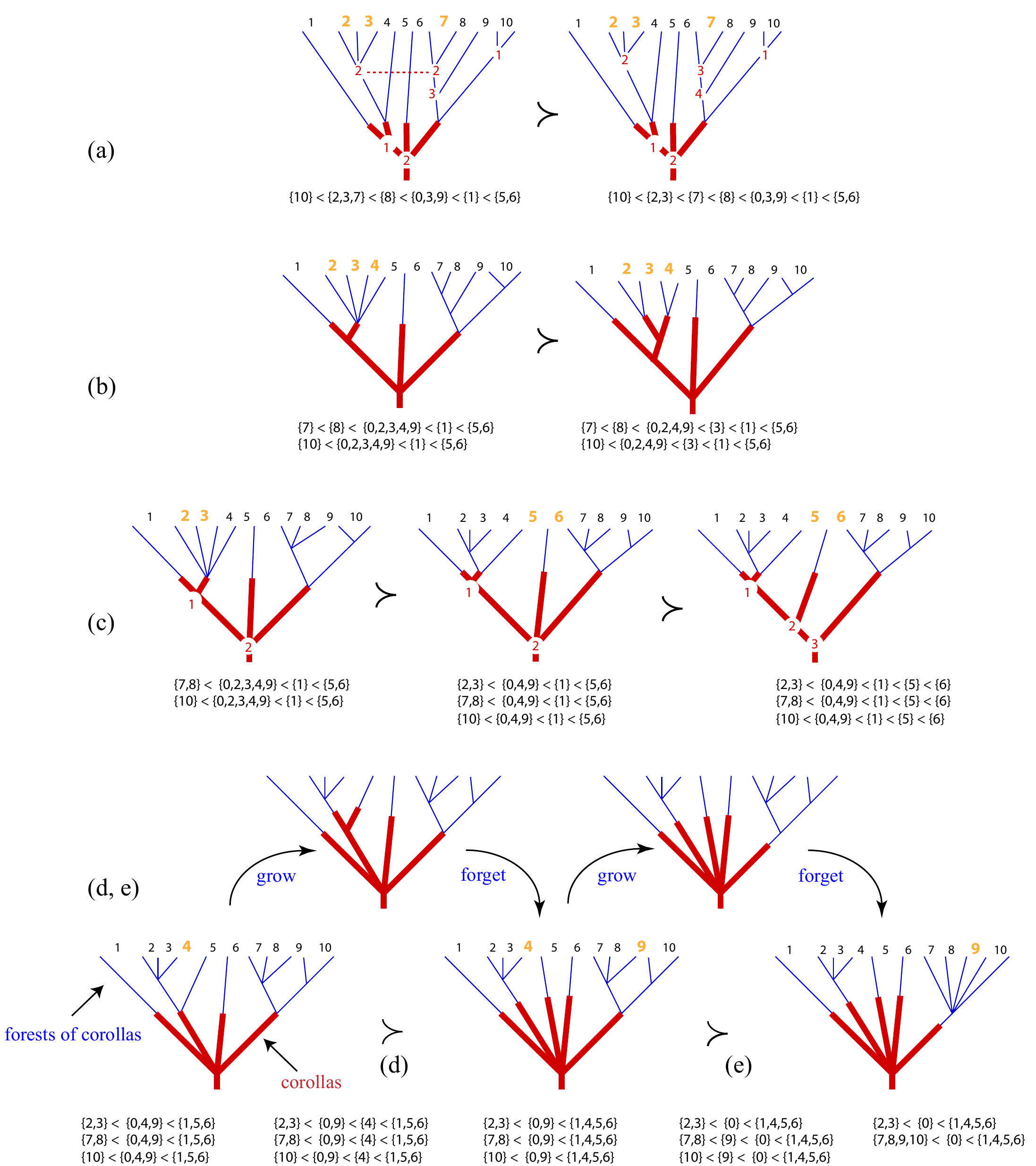}
\begin{exam}\label{covering}
Above we see covering relations, with the associated partially ordered partitions $\pi$ shown below each tree. In the first (a) we are looking at
weakly ordered forests grafted to a weakly ordered tree, so growing
an edge is a covering relation. In the next relation (b) we are
looking at rooted trees above and below the graft--again no
forgetting is needed. Relations (c) are in the stellohedron. At the
bottom, relations (d,e) are in the cube. For both covering relations the forgetful map is $\kappa.$
\end{exam}

                  \includegraphics[width=\textwidth]{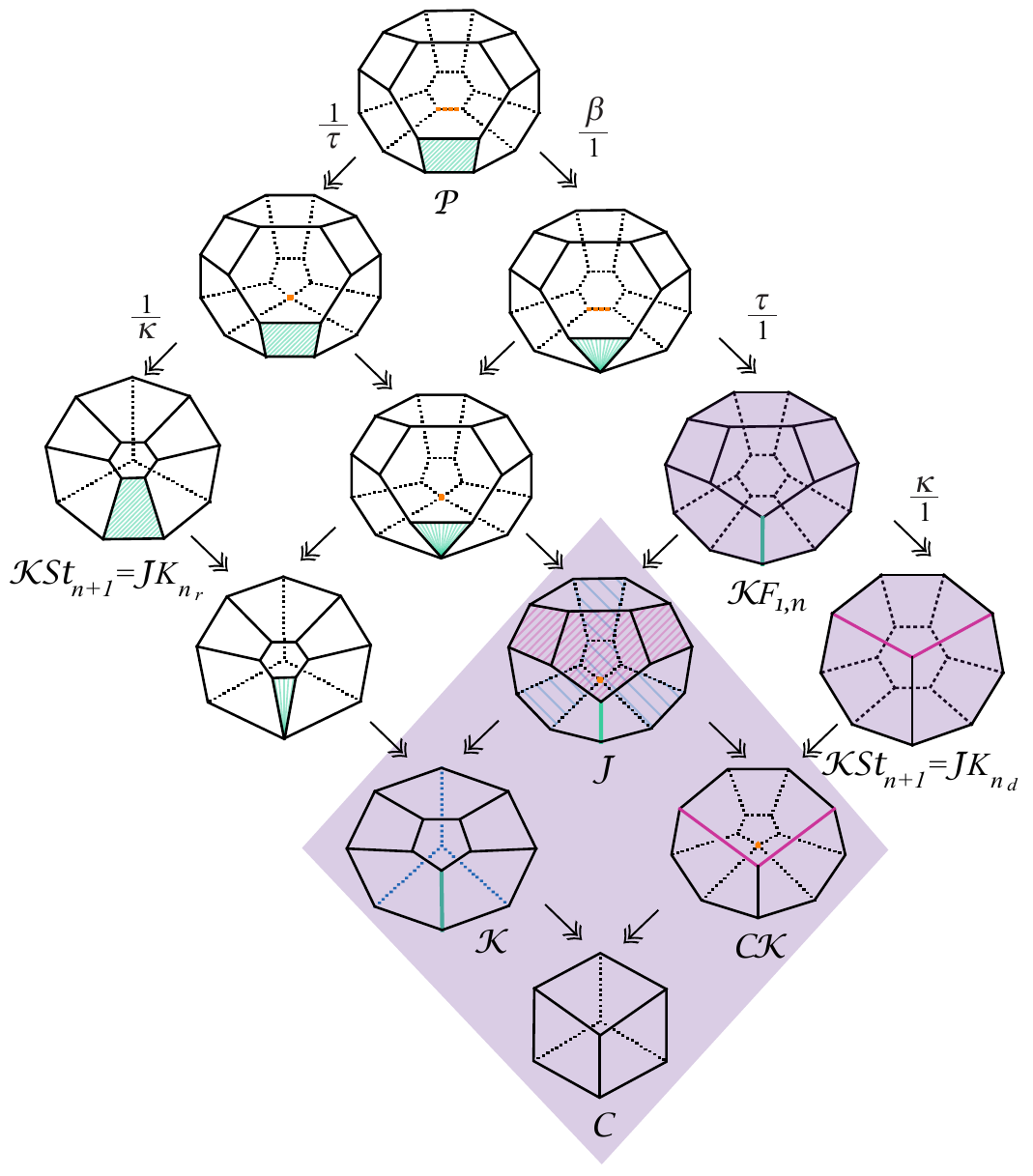}
\begin{exam}\label{bigpoly_cccc_ext_lit}
~ \end{exam} The 3-dimensional polytopes which represent the painted
trees in our 12 sequences. The four in the shaded diamond are the
cube $\mathcal {C}$, associahedron $\mathcal{K},$ multiplihedron
$\mathcal {J}$ and composihedron $\mathcal{CK}.$ The other two
shaded  polytopes are the pterahedron $\K F_{1,n}$ (fan graph
associahedron) and the stellahedron $\K St.$ The topmost is the
permutohedron. The furthest to the left is again the stellohedron.
The other four, unlabeled, are conjectured to be polytopes (clearly
they are in three dimensions--the conjecture is about all
dimensions.) Each of these corresponds to the type of tree shown in
Figure~\ref{pdiamond}, in the corresponding position. Hatching indicates that a face will collapse
under the action of the fractional map.

\begin{theorem}
The growth preorder on painted trees results in a poset for each of our 12 types of painted trees.
\end{theorem}
\begin{proof} Reflexivity and transitivity are by construction, since relations are any series of
 moves $(a,b)_i$ (including the empty series). Antisymmetry is straightforward
for most of the 12 types since the relations correspond purely to the refinement of the
 partially ordered partitions of the gaps between leaves. For the relations
that use the non-trivial forgetful maps, we note that these always involve a refinement of the partition
which subdivides the part of the partition containing 0.
This suffices to demonstrate antisymmetry since a relation cannot result in increasing
the size of the part containing 0 (that is, the half painted nodes never increase in number). For example
see (d) and (e) in Example~\ref{covering}.
\end{proof}

Here is a detail from Example~\ref{bigpoly_cccc_ext_lit}, in which
we label the vertices (and one complete triangular facet) of the
polytope of weakly ordered trees over corollas on its Schlegel
diagram.
\begin{center}\includegraphics[width=3in]{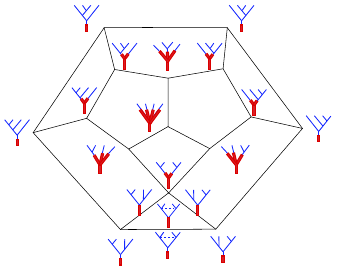}
\end{center}
In what follows we argue that most of the posets just described are
realized by inclusion of faces in a convex polytope.

\subsection{ Bijections } We conjecture
that in all 12
 cases the the painted growth partial order is realized as the face
posets of sequences of convex polytopes. Four of the cases have been
proven in previous work. These four appear as the highlighted
diamond in Example~\ref{bigpoly_cccc_ext_lit}. The polytope
sequences are the cubes, associahedra, composihedra and
multiplihedra. The latter three are shown (with pictures of painted
trees) in \cite{multi}; the fact that the cubes result from
forgetting all the branching structure is equivalent to the fact
that cubes arise when both of two product spaces are associative, as
pointed out in \cite{BV1}, also (with design tubings) in
\cite{devcube}.

\begin{figure}[htb!]
\begin{center}
\includegraphics[width=4.7in]{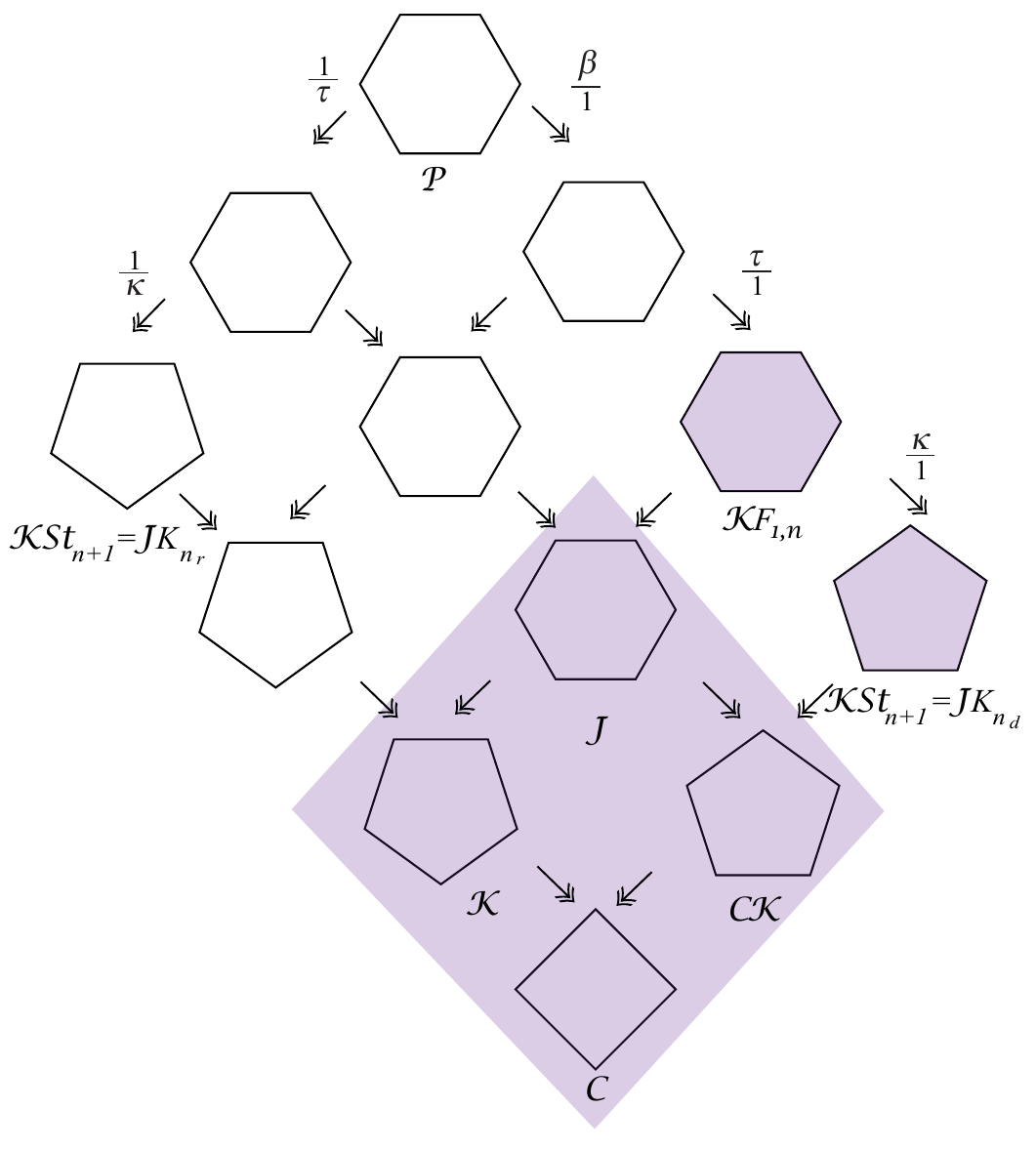}
\caption{These are the 2-dimensional terms in the same sequences as
in Example~\ref{bigpoly_cccc_ext_lit}.}\label{bigpoly_ext_2d}
\end{center}
\end{figure}

\begin{figure}[htb!]
\begin{center}
\includegraphics[width=\textwidth]{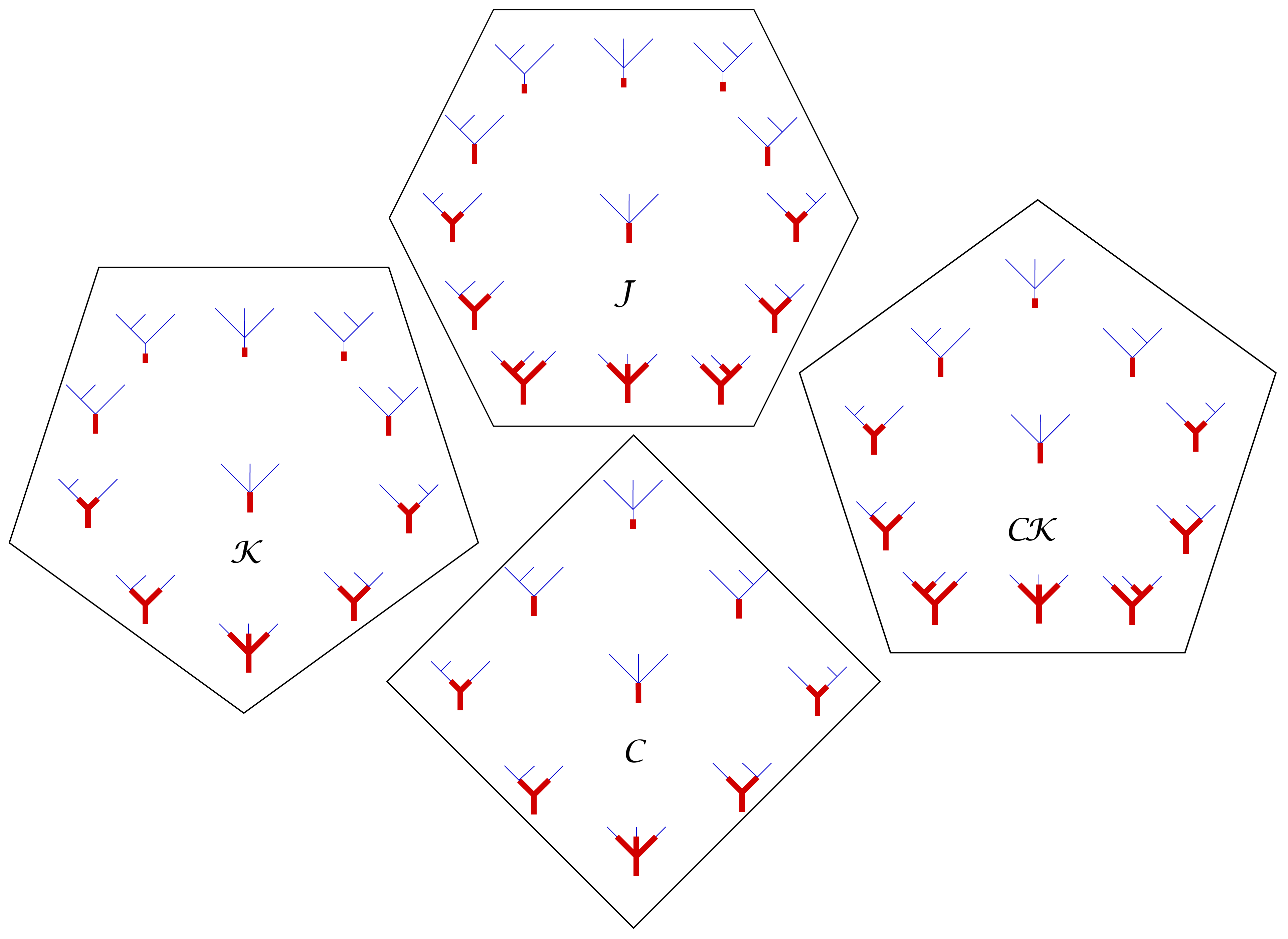}
\caption{These are the 2-dimensional terms with their faces labeled.
The same labels are used no matter where the shape occurs in
figure~\ref{bigpoly_ext_2d}. }\label{2s_labels}
\end{center}
\end{figure}

In this section four more sequences of our sets of painted trees,
with their relations, will be shown to be isomorphic as posets to
face lattices of convex polytopes. Two of these are the species
whose structure types are:   a forest of corollas grafted to a
weakly ordered tree (stellohedra) or a weakly ordered forest grafted
to a corolla (stellohedra again). A third is the species whose
structure type is the weakly ordered forest grafted to a weakly
ordered tree (permutohedra). Finally the species whose structure
type is a forest of plane rooted trees grafted to a weakly ordered
tree (pterahedra). There remain four cases in
Example~\ref{bigpoly_cccc_ext_lit} that we leave as a conjecture.
(These latter four are the ones which do not have a label naming
them under their picture).

 Some of our proofs and corollaries will
use the concept of tubings, which we review next.

\subsection{Tubes, tubings and marked tubings.}
The definitions and examples in this section are largely taken from
\cite{ForSpr:2010} and \cite{dev-real}. They are based on the
original definitions in \cite{dev-carr}, with only the slight change
of allowing a universal tube, as in \cite{dev-forc}.

\begin{definition}
Let $G$ be a finite connected simple graph.  A \emph{tube} is a set
of nodes of $G$ whose induced graph is a connected subgraph of $G$.
For a given tube $t$ and a graph $G$, let $G(t)$ denote the induced
subgraph on the graph $G$. We will often refer to the induced graph
itself as the tube. Two tubes $u$ and $v$ may interact on the graph
as follows:
\begin{enumerate}
\item Tubes are \emph{nested} if  $u \subset v$.
\item Tubes are \emph{far apart} if $u \cup v$ is not a tube in $G,$ that is, the induced subgraph of the union
 is not connected, (equivalently
none of the nodes of $u$ are adjacent to a node of $v$).
\end{enumerate}
Tubes are \emph{compatible} if they are either nested or far apart. We call $G$
itself the \emph{universal tube}.
 A \emph{tubing} $U$ of $G$ is a set of tubes of $G$ such that every pair of tubes in $U$ is
 compatible; moreover, we force every tubing of $G$ to contain (by default) its universal tube.
By the
 term $k$-\emph{tubing} we refer to a tubing made up of $k$ tubes, for $k \in \{1,\dots,n\}.$
 \end{definition}~\\\\

\begin{rema}For connected graphs our definition here is equivalent to
that in \cite{post2}. In \cite{dev-carr} and \cite{dev-real} the
universal tube is not considered a tube, nor included in tubings.
This however leads to a poset of tubings which is isomorphic to the
one in this paper.~\end{rema}

When $G$ is a disconnected graph with connected components $G_1$,
\ldots, $G_k$, there are alternate definitions in the literature. In
\cite{post} and \cite{post2}, as well as in \cite{ardila}, the
connected components are all tubes and must all be included in every
tubing. We will refer to this as a \emph{building set tubing} since
it contains all maximal elements.

Alternatively, in \cite{dev-carr}, \cite{dev-forc} and
\cite{forc-spring}, as well as in \cite{FLS3}, the additional
condition for disconnected graphs is as follows: If $u_i$ is the
tube of $G$ whose induced graph is $G_i$, then no tubing of $G$
 contains all of the tubes $\{u_1, \ldots, u_k\}$. However, the
universal tube is still included in all tubings despite being itself
disconnected.

 Parts (a)-(c) of Figure~\ref{f:legaltubing} from \cite{dev-forc} show examples of allowable tubings,
 whereas (d)-(f) depict the forbidden ones.

\begin{figure}[htb!]
\includegraphics[width=\textwidth]{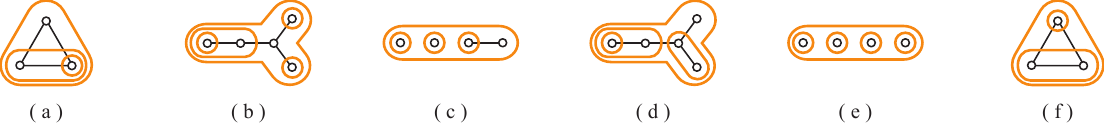}
\caption{(a)-(c) Allowable tubings and (d)-(f) forbidden tubings, figure from
\cite{dev-forc}.} \label{f:legaltubing}
\end{figure}


Let ${\mbox {Tub}(G)}$ denote the set of tubings on a graph $G.$ As
shown in {\textup{\cite[Section 3]{dev-carr}}}, for a graph $G$ with
$n$ nodes, the \emph{graph associahedron} $\KG$ is a simple, convex
polytope of dimension $n-1$ whose face poset is isomorphic to
${\mbox {Tub}(G)}$, partially ordered by the relationship $U \prec
U'$ if $U'\subseteq U.$

 The vertices of the graph associahedron are the $n$-tubings of $G.$
 Faces of dimension $k$ are indexed by $(n-k)$-tubings of $G.$ In fact,
 the barycentric subdivision of ${\mathcal K}G$ is precisely the geometric realization of the described poset of tubings.

To describe the face structure of the graph associahedra we need a
definition from \cite[Section 2]{dev-carr}.

\begin{defi}
For graph $G$ and a collection of nodes $t$, construct a new graph
$G^{*}(t)$ called the \emph{reconnected complement}: If $V$ is the
set of nodes of $G$, then $V-t$ is the set of nodes of $G^{*}(t)$.
There is an edge between nodes $a$ and $b$ in $G^{*}(t)$ if $\{a,b\}
\cup t'$ is connected in $G$ for some $t'\subseteq t$.
\end{defi}

\noindent Figure~\ref{f:recon} illustrates some examples of graphs
along with their reconnected complements.
\begin{figure}[htb!]
\includegraphics {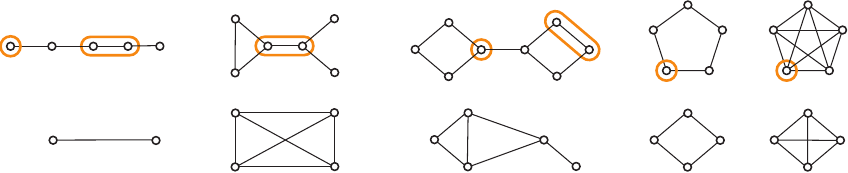}
\caption{Examples of tubes and their reconnected complements. Figure
from \cite{ForSpr:2010}.} \label{f:recon}
\end{figure}


\begin{theorem} \label{t:facet} \cite[Theorem 2.9]{dev-carr}
Let $V$ be a facet of $\KG,$ that is, a face of dimension $n-2$ of
${\mathcal K}G$, where $G$ has $n$ nodes. $V$ corresponds to  $t$, a
single, non-universal, tube of $G$ . The face poset of $V$ is
isomorphic to $ \KG(t) \times \K G^{*}(t)$.
\end{theorem}

We will consider a related operation on graphs. The {\it suspension}
of $G$ is the graph ${\mathfrak S} G$ whose set of nodes is obtained
by adding a node $0$ to the set ${\mbox {Nod}(G)}$, of nodes of $G$,
and whose edges are defined as all the edges of $G$ together with
the edges $\{0,v\}$ for ${ v\in \mbox{Nod}(G)}.$

The reconnected complement of $\{0\}$ in ${\mathfrak S} G$ is the
complete graph ${\mathcal K}_n$ for any graph $G$ with $n$ nodes.
Note that the star graph $St_n$ is the suspension of the graph
which has $n$ nodes and no edge, while the fan graph $F_{1,n}$ is the
suspension of the path graph on $n$ nodes. The suspension of a
complete graph $K_n$ is the complete graph $K_{n+1}.$

It turns out that this construction of the graph multiplihedra is a
special case of a more general construction on certain polytopes
called the \emph{generalized permutahedra} as defined by Postnikov
in \cite{post}. The \emph{lifting} of a generalized permutahedron,
and a nestohedron in particular, is a way to get a new generalized
permutahedron of one greater dimension from a given example, using a
factor of $q\in[0,1]$ to produce new vertices from some of the old
ones \cite{ardila}. This procedure was first seen in the proof that
Stasheff's multiplihedra complexes are actually realized as convex
polytopes \cite{multi}.\vspace{.05in}

 Soon afterwards the lifting
procedure was applied to the graph associahedra--well-known examples of
nestohedra first described by Carr and Devadoss. We completed
    an initial study of the resulting polytopes, dubbed \emph{graph multiplihedra}, published as \cite{dev-forc}.

 This application raised the question of a
general definition of lifting using $q.$ At the time it was also unknown
whether the results of lifting, then just the multiplihedra and the
graph-multiplihedra, were themselves generalized permutahedra. These questions
were both answered in the recent paper of Ardila and Doker \cite{ardila}. They
defined nestomultiplihedra and showed that they were generalized permutohedra
of one dimension higher in each case.


We refer the reader to Ardila and Doker \cite{ardila} for the
general definitions. Here we need only the following definitions,
from \cite{dev-forc}. Combinatorially, lifting of a graph
associahedron occurs when the notion of a tube is extended to
include markings.

\begin{definition}
A \emph{marked tube} of a graph $G$ is a tube with one of three
possible markings:
\begin{enumerate}
\item a \emph{thin} tube \ \tubeA \ \ given by a solid line,
\item a \emph{thick} tube \ \tubeC \ \ given by a double line, and
\item a \emph{broken} tube  \ \tubeB \ \ given by fragmented pieces.
\end{enumerate}
Marked tubes $u$ and $v$ are \emph{compatible} if
\begin{enumerate}
\item they form a tubing and
\item if $u \subset v$ where $v$ is not thick, then $u$ must be thin.
\end{enumerate}
A \emph{marked tubing} of $G$ is a tubing of pairwise compatible
marked tubes of $G$.
\end{definition}

A partial order is now given on marked tubings of a graph $G$.  This poset
structure is then used to construct the \emph{graph multiplihedron} below.

\begin{definition} \label{d:poset}
The collection of marked tubings on a graph $G$ can be given the
structure of a poset.   For two marked tubings $U$ and $U'$ we have
$U \prec U'$ if $U$ is obtained from $U'$ by a combination of the
following four moves. Figure~\ref{f:poset} provides the appropriate
illustrations, with the top row depicting $U'$ and the bottom row
$U$.
\begin{enumerate}
\item \emph{Resolving markings}:  A broken tube becomes either a thin tube (\ref{f:poset}a) or a thick tube (\ref{f:poset}b).
\item \emph{Adding thin tubes}:  A thin tube is added inside either a thin tube (\ref{f:poset}c) or broken tube (\ref{f:poset}d).
\item \emph{Adding thick tubes}:  A thick tube is added inside a thick tube (\ref{f:poset}e).
\item \emph{Adding broken tubes}:  A collection of compatible broken tubes $\{u_1, \ldots, u_n\}$ is added simultaneously inside a broken tube $v$ only when $u_i \Subset v$ \emph{and} $v$ becomes a thick tube; two examples are given in (\ref{f:poset}f) and (\ref{f:poset}g).
\end{enumerate}
\end{definition}

\begin{figure}[ht!]
\begin{center}
\includegraphics[width=\textwidth]{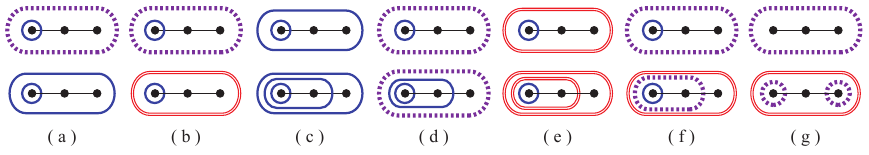}
\end{center}
\caption{The top row are the tubings and bottom row their
refinements. Figure based on original in \cite{dev-forc}.}
\label{f:poset}
\end{figure}

Here is the key idea from \cite{dev-forc}: for a graph $G$ with $n$
nodes, the \emph{graph multiplihedron} $\JG$ is a convex polytope of
dimension $n$ whose face poset is isomorphic to the set of marked
tubings of $G$ with the poset structure given above.

There are two important quotient polytopes mentioned in
\cite{dev-forc}: $\JGd$ and $\JGr$ for a given graph $G.$ The former
is called the \emph{graph composihedron}. Its faces correspond to
marked tubings, but for which no thin tubes are allowed to be inside
another thin tube.  In terms of equivalence of tubings, the face
poset of $\JGd$ is isomorphic to the poset $\JG$ modulo the
equivalence relation on marked tubings generated by identifying any
two tubings $U\sim V$ such that $U\prec V$ in $\JG$ precisely by the
addition of a thin tube inside another thin tube, as in
Figure~\ref{f:poset}(c). Thus an equivalence class of tubings can be
represented by its maximum member: a tubing with no thin tubes
inside any other thin tube. The graph composihedron is defined via
geometric realization in \cite{dev-forc}. The relations in
Figure~\ref{f:poset} still hold, but some of them appear
differently, and one (c) is no longer present in
Figure~\ref{f:poset_composi}.

\begin{figure}[ht!]
\begin{center}
\includegraphics[width=\textwidth]{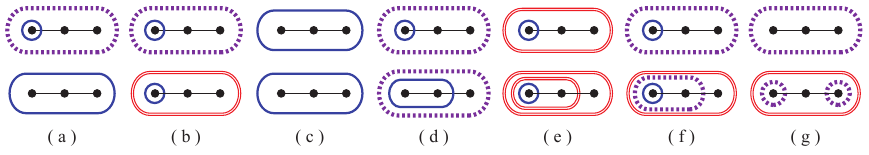}
\end{center}
\caption{The top row are the tubings and bottom row their
refinements, in the graph composihedron. These are altered versions
(shown up to equivalence) of the relations in Figure~\ref{f:poset}.
In fact (c) is the reflective relation.} \label{f:poset_composi}
\end{figure}

The polytope $\JGr$ has faces which correspond to marked tubings,
but for which no
 thick tubes are allowed to be inside another thick tube.
  In terms of equivalence of tubings, the face poset of
$\JGr$ is isomorphic to the poset $\JG$ modulo the equivalence
relation on marked tubings generated by identifying any two tubings
$U\sim V$ such that $U\prec V$ in $\JG$ precisely by the addition of
a thick tube, as in Figure~\ref{f:poset}(e).Thus an equivalence
class of tubings can be represented by its maximum member: a tubing
with no thin tubes inside any other thin tube. $\JGr$ is defined via
geometric realization in \cite{dev-forc}. For connected graphs $G$,
the polytope $\JGr$ is combinatorially equivalent to the graph
cubeahedron $\C G,$ as defined in \cite{devcube}.

The graph cubeahedron $\C K_n$ is described in \cite{devcube} as
comprising the \emph{design-tubings} on the complete graph. In
Figure~\ref{fig: biject5} we show the correspondence between labels
of vertices: range-equivalence classes of marked tubings and design
tubings.
 The isomorphism claimed in
\cite{devcube} is easily described: design tubes (square tubes)
correspond to the nodes not inside any thin or broken tube; while
round tubes in the design tubing correspond to thin tubes. Broken
tubes contain any nodes not in any tube of the design tubing.

 For
this reason we refer to the entire class of polytopes $\JGr$ as the
\emph{(general) graph cubeahedra.} In fact the description of $\C G$
using design tubings which is given in \cite{devcube} is not
difficult to extend to graphs with multiple components: we only need
to introduce the universal (round) tube. For example, the graph
cubeahedron for the edgeless graph is the hypercube with a single
truncated vertex.

The four well-known examples of polytopes from
Example~\ref{bigpoly_cccc_ext_lit} can be seen as tubing posets, as
pointed out in \cite{dev-forc}. The multiplihedra $\J = \J P$ have
face posets equivalent to the  marked tubes on path graphs $P$. The
composihedra are the domain quotients of these: $\J P_d$; and the
associahedra are the range quotients of these: $\J P_r.$  The cubes
show up as the result of taking both quotients simultaneously.

\begin{figure}[hbt!]
                  \centering
                  \includegraphics[width=\textwidth]{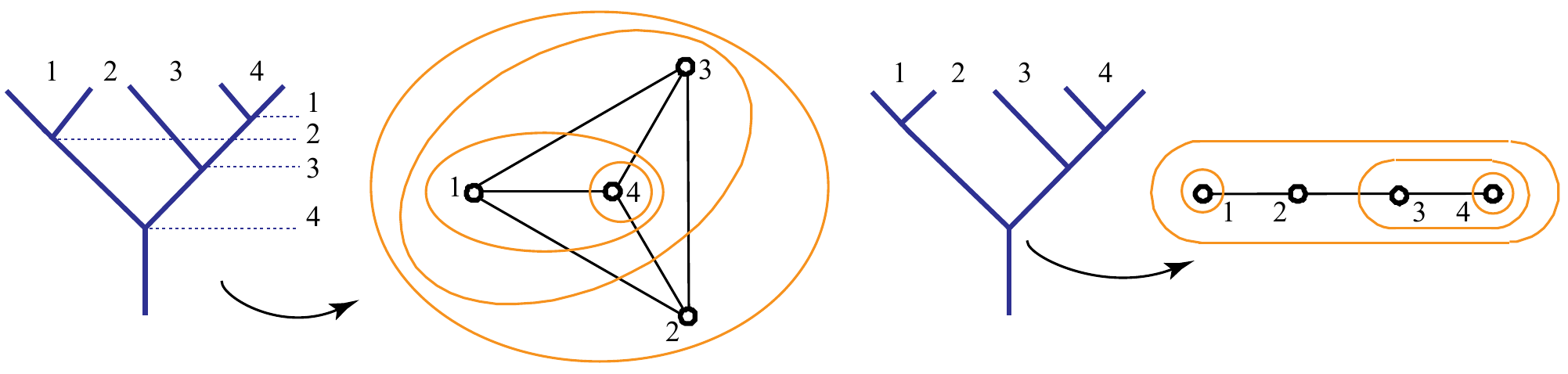}
     \caption{The permutation $\sigma = (2431) \in S_4$ pictured as an ordered tree and
                   as a tubing of the complete graph; An unordered binary tree, and its corresponding tubing. Figure from \cite{ForSpr:2010}.}\label{perm1}
  \end{figure}

  \begin{figure}[hbt!]
                  \centering
                  \includegraphics[width=\textwidth]{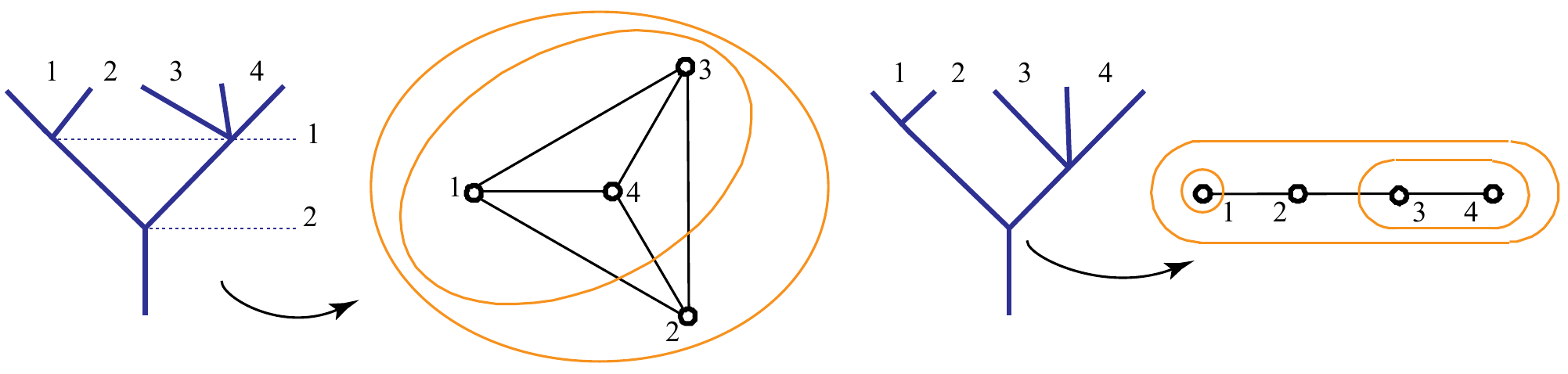}
     \caption{The ordered partition $(\{1,2,4\},\{3\})$ pictured as a leveled tree and
                   as a tubing of the complete graph; the underlying tree, and its corresponding tubing. Figure from \cite{ForSpr:2010}.}\label{perm2}
  \end{figure}

\subsection{Permutohedra}

First we prove that the poset of painted trees made by grafting a
weakly ordered forest to a weakly ordered base tree  is the face
poset of a polytope. It turns out that for painted trees with $n$
leaves this polytope is the permutohedron $\mathcal{P}_n.$ It is
well known (see \cite{LR}) that the permutohedron has faces indexed
by the weak orders, which in turn may be represented by weakly
ordered trees. The face poset is the partial ordering of these trees
by refinements.


\begin{theorem}There is an isomorphism $\varphi$ from the poset of $(n+1)$-leaved weakly ordered
trees to the painted growth preorder of $n$-leaved weakly ordered
forests grafted to weakly ordered trees.

\end{theorem}\label{perma}
\begin{proof}
The isomorphism and its inverse are described as switching between
the paint line and an extra branch. Given a weakly ordered tree $t$,
we find $\varphi(t)$ by adding a paint line at the level of
left-most node of $t$, and then deleting the left-most branch of
$t$. Finally the remaining nodes are ordered, above and below the
paint line, according to their original vertical order in $t.$ The
inverse is straightforward. Here is a picture of the process:

\includegraphics[width=\textwidth]{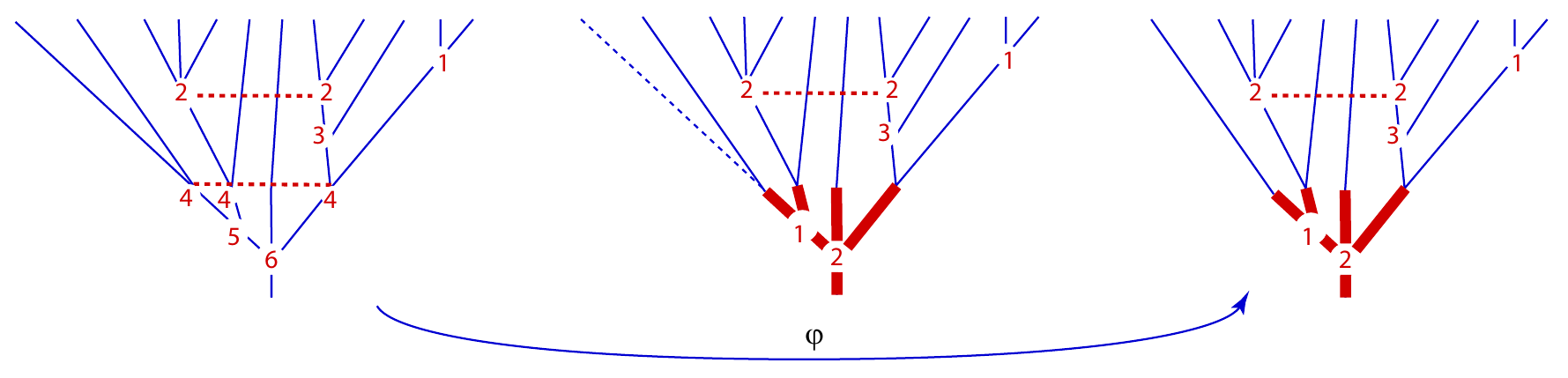}

Next we argue that the isomorphism just described respects the poset
structures. If $a\prec b$ for two weakly ordered trees, we have that
the weak ordering of the nodes of $a$ is a refinement of the weak
ordering for $b.$ We can visualize this refinement as the growing of
some internal edges of $a$ to break ties between nodes that were at
the same level.  If the refinement involves breaking a tie that does
not include the left-most node (see level 2 in the above picture),
then the same growing produces the same relation between the painted
tree images $\varphi(a)$ and $\varphi(b).$ If the growing does break
a tie involving the left-most node (see level 4 in the above
picture), then the image of $b$ may differ from that of $a$ only in
that the set of nodes of $\varphi(a)$ which coincide with the paint
line will be a subset of those in $\varphi(b)$. This can be seen as
growing edges at some half-painted nodes. Here is an example of the
latter case, with trees related to those in the above pictured
example:

\includegraphics[width=4.5in]{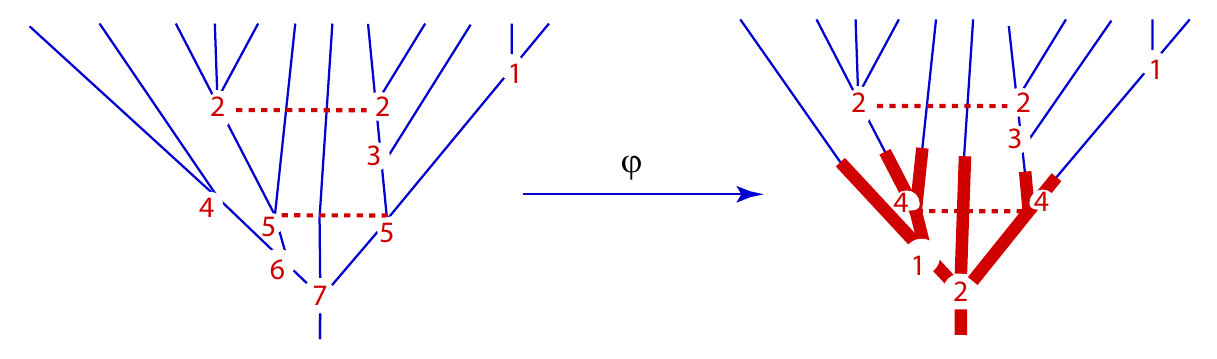}
\end{proof}

This theorem immediately implies that the poset of $n$-leaved weakly
ordered forests grafted to weakly ordered trees is isomorphic to the
face poset of the $n$-dimensional permutohedron. That is because the
poset of $(n+1)$-leaved weakly ordered trees is well known to
represent the face poset of the permutohedron (via seeing each tree
as a weak order of $[n]$, that is, an ordered partition.)

A corollary, from \cite{dev-forc}, is that the poset of $n$-leaved
weakly ordered forests grafted to weakly ordered trees is isomorphic
to the face poset of the $n$-dimensional graph multiplihedron of the
complete graph.

\subsection{Stellohedra}

Now we prove that the poset of painted trees made by grafting a
forest of corollas to a weakly ordered base tree is the face poset
of a polytope. It turns out that for painted trees with $n+1$ leaves
this polytope is the graph-associahedron $\KG$ where $G$ is the star
graph $St_{n}.$ ~

Recall that the star graph $St_{n}$ is defined as follows: we use
the set $\{0,1,2,\dots,n\}$ as the set of nodes. Edges are
$\{0,i\}$ for $i=1,\dots,n.$

\begin{theorem}\label{stella1}
The poset of tubings on the star graph $St_{n}$ is isomorphic to the
poset of $n$-leaved forests of corollas grafted to weakly ordered
trees.
\end{theorem}
\begin{proof}
We first note that any tubing $T$ of the star graph includes a
unique smallest tube $t_0$ which contains node 0.  All other tubes
of $T$ are either contained in $t_0$ or contain $t_0,$ since the
node 0 is adjacent to all other nodes. The tubes contained in $t_0$
form a tubing of an edgeless graph. The tubes containing $t_0$ form
a tubing on the reconnected
 complement of $t_0$, which is the complete
graph on the nodes not in $t_0.$ Here the key idea is that the tube
$t_0$ is analogous to the half-painted nodes. See
Figure~\ref{fig:stello_bij_big}.

Now we use two facts shown in \cite{dev-carr}: that the
permutohedron is combinatorially equivalent to the
graph-associahedron of the complete graph, and that the simplex is
combinatorially equivalent to the graph-associahedron of the
edgeless graph, which in turn is equivalent to the Boolean lattice
of subsets of its nodes. Recall that the permutohedron is also
indexed by the weakly ordered trees, leading to an isomorphism
between tubings and trees as seen in Figures~\ref{perm1}
and~\ref{perm2}.
\begin{figure}[ht!]\centering
                  \includegraphics[width=\textwidth]{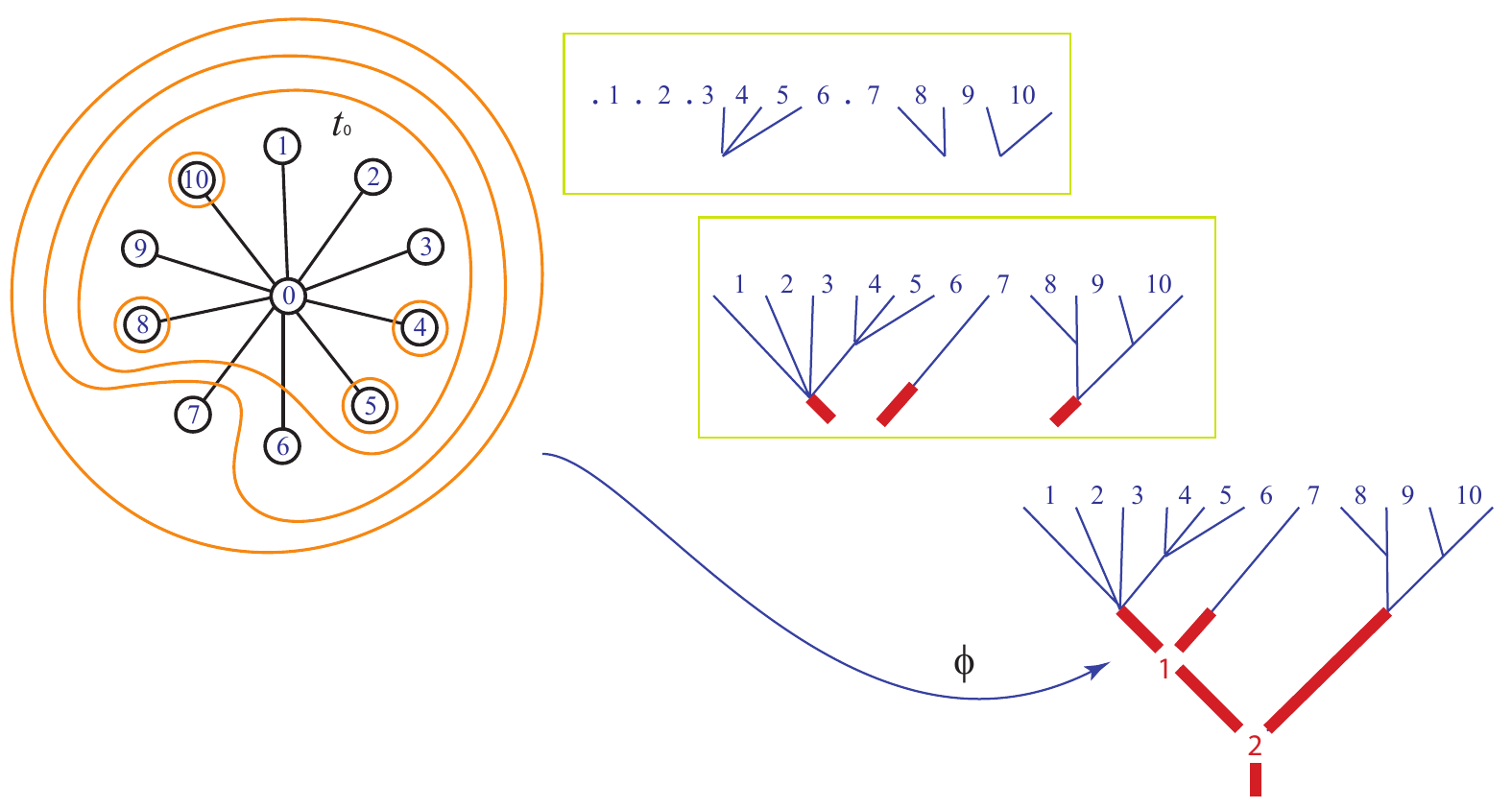}
\caption{A tubing $T$ on the star graph and its bijective image in
the corollas over weakly ordered trees. The three steps are shown
for constructing $\phi(T).$}\label{fig:stello_bij_big}
\end{figure}

Thus the bijection $\phi$ we want takes a tubing $T$ on the star
graph $S = S_n$ to a painted tree. This bijection is constructed
from the bijection $\alpha$ from tubings on an edgeless graph to
subsets of $n$ gaps between leaves (corresponding to nodes in a
unpainted forest of corollas); together with the bijection $\beta$
from tubings on a complete graph with $j$ vertices to weakly ordered
trees with $j+1$ leaves.

The construction of $\phi$ proceeds as follows.
 First the nodes $1,\dots,n$ of the star graph
correspond to the gaps (between leaves) $1,\dots,n$ of the output
tree.
 The tubing of the subgraph inside of $t_0$ maps via $\alpha$ to a subset of $[n],$ and that subset is
 precisely the
subset of the gaps which correspond to unpainted nodes (of corollas)
in our output tree. Second, nodes that are inside $t_0$ but not
inside any smaller tube determine the gaps that coincide with the
paint line, half-painted nodes on our output tree.
 Finally the
tubing outside of $t_0$  maps via $\beta$ to the painted  weakly
ordered tree. The inverse of $\phi$ is the straightforward reversal
of these steps. An example is seen in
Figure~\ref{fig:stello_bij_big}.

Checking that this bijection preserves the ordering is
straightforward. Covering relations in the stellohedron are face
inclusions, which each correspond to adding one tube to a tubing.
The addition of a singleton tube  inside of $t_0$ corresponds under
our bijection to growing an unpainted edge at a half-painted node
(and then applying $\kappa.$)

The addition of a tube just inside of $t_0$ that contains all the
singleton tubes, so in effect creating a new $t_0$, corresponds to
growing some painted edges from half-painted nodes.

The addition of a tube outside of $t_0$ corresponds to growing a
painted edge at a painted node. The three possibilities are
illustrated here: the first has a singleton tube added (around
vertex 9) compared with the original tubing in
Figure~\ref{fig:stello_bij_big}.

\includegraphics[width=\textwidth]{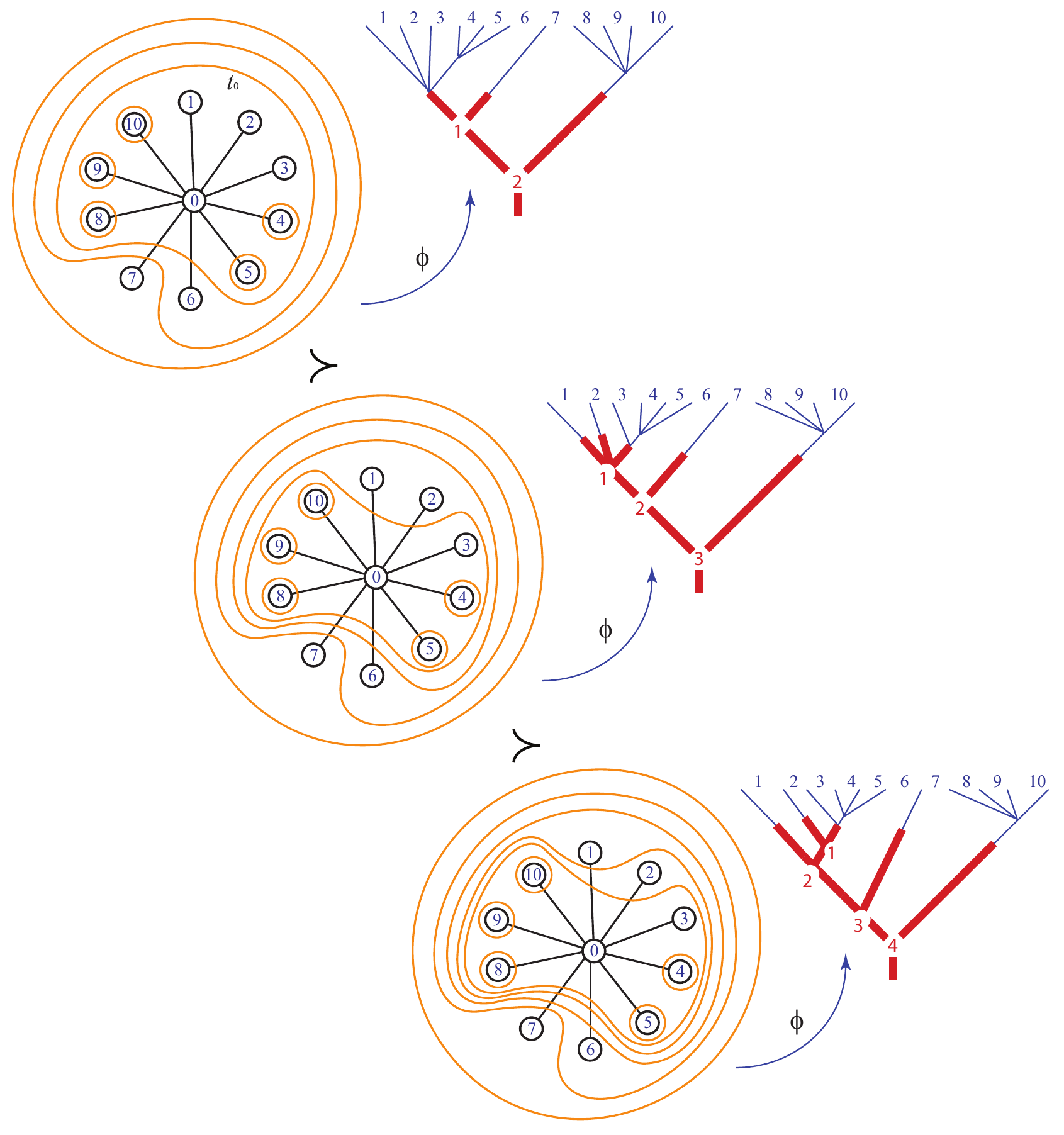}
\end{proof}

The isomorphism of vertices of the polytopes in 3 dimensions is shown
pictorially in Figure~\ref{fig:biject1}.

\begin{figure}[htb!]\centering
                  \includegraphics[width=\textwidth]{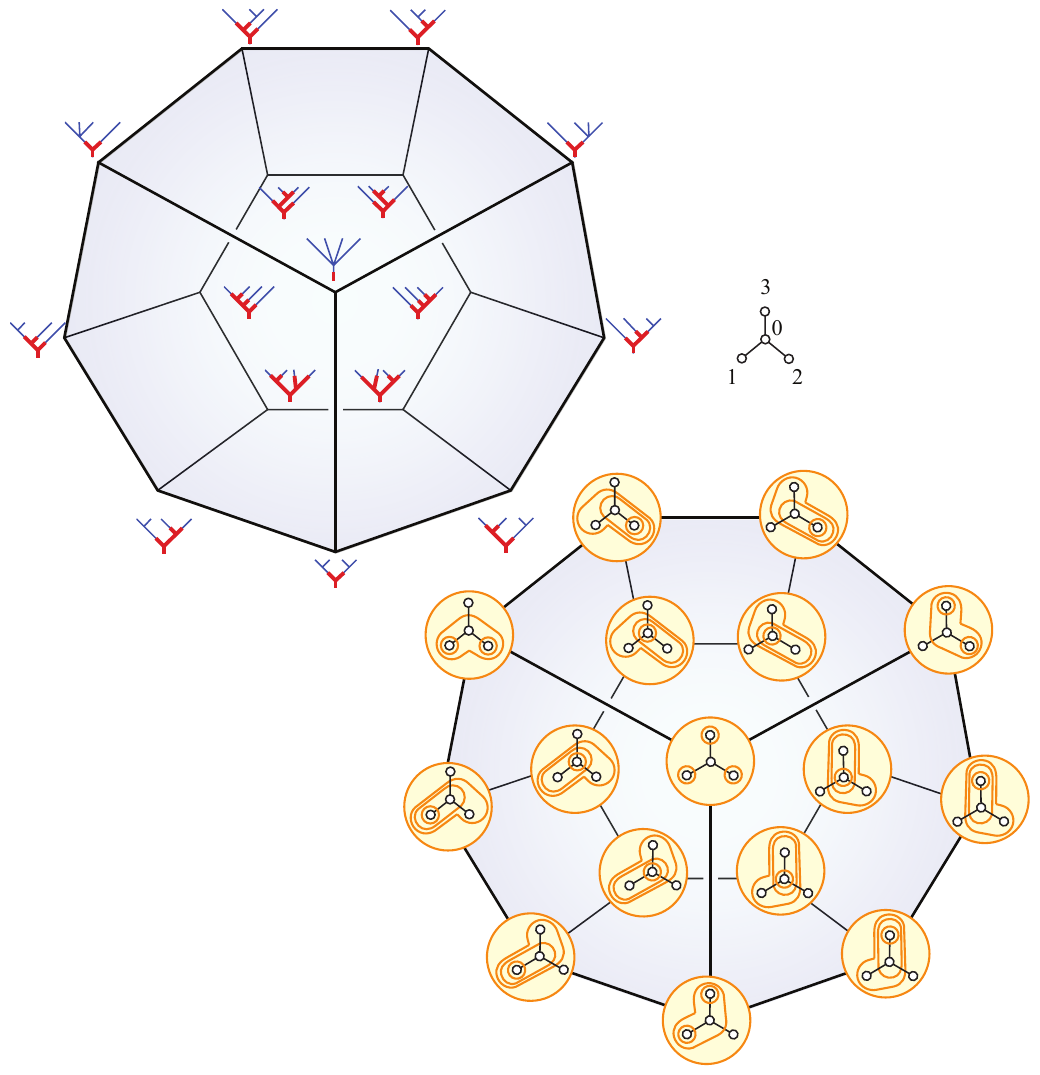}
\caption{Two pictures of the stellohedron, via the bijection in
Theorem~\ref{stella1}.}\label{fig:biject1}
\end{figure}

Next we show that the stellohedra can also be seen as the domain and
range quotients $\JG_d$ and $\JG_r$ of the multiplihedron $\JG$
where $G$ is the complete graph.

\begin{theorem}\label{lift}
The graph-composihedron for a complete graph $\K_n$ is
combinatorially equivalent to the stellohedron for the star-graph
$St_{n+1}.$

\end{theorem}
~
\begin{proof}
We can most easily see the isomorphism by using the stellohedra just
found in Theorem~\ref{stella1},   that is, by showing an isomorphism
to painted trees.


 We show a bijection $\phi'$ from the graph-composihedron of the complete graph to the set $\widetilde{C/D}$ of forests
of corollas grafted to weakly ordered trees.  The nodes $1,\dots,n$
of the complete graph correspond to the gaps (between leaves)
$1,\dots,n$ of the tree.
 Here the key idea is that now a broken tube $t_0$ plays the
same role as the half-painted nodes in the corresponding tree, and a
single thin tube the role of the unpainted nodes. The steps in the
construction of the bijection $\phi'$ are analogous to those in the
proof of Theorem~\ref{stella1}, as follows:

The bijection $\phi'$ takes as input a marked tubing on the complete
graph with no thin tubes inside another thin tube. It outputs a
painted tree as follows: if there is a single thin tube
$t=\{v_1,v_2,\dots,v_k\}$ then the like-numbered gaps of the output
will correspond to nodes of  unpainted corollas. Nodes that are
inside a broken tube but not inside any thin tube of the input
correspond to the gaps of the output that correspond to half-painted
nodes. Any nodes outside of all the thin or broken tubes in the
input correspond to nodes of the weakly ordered base tree in the
output, and this mapping is via the previously mentioned bijection
between weak orders and tubings on the complete graph. Note that the
reconnected complement of the largest thin or broken tube is a
complete graph. The inverse of $\phi'$ is the straigtforward
reversal of these steps. An example of the bijection $\phi'$ is seen
in Figure~\ref{fig:stello_bij_composi}.

We check that this bijection $\phi'$ preserves the ordering. Note
that the relations are simpler than in general for marked tubes
since the tubings must all be completely nested, and since thin
tubes inside of thin tubes are ignored (via the equivalence). Thus
the relations in the  Figure~\ref{f:poset_composi}(c)
and~\ref{f:poset_composi}(g) need not be checked. The relations in
Figure~\ref{f:poset_composi}(a) and (d) correspond to growing
unpainted edges from half-painted nodes. The relations in
Figure~\ref{f:poset_composi}(b) and (f) correspond to growing
painted edges from half-painted nodes.  The relation in
Figure~\ref{f:poset_composi}(e) corresponds to growing  a painted
edge from a painted node. Examples of the preservation of ordering
via $\phi'$ re seen in Figure~\ref{fig:stello_bij_composi_ord}. 3-dimensional
examples are seen in Figure~\ref{fig: biject2}. See
Figure~\ref{chains} for some isomorphic chains.
\end{proof}

\begin{figure}[ht!]\centering
                  \includegraphics[width=4.75in]{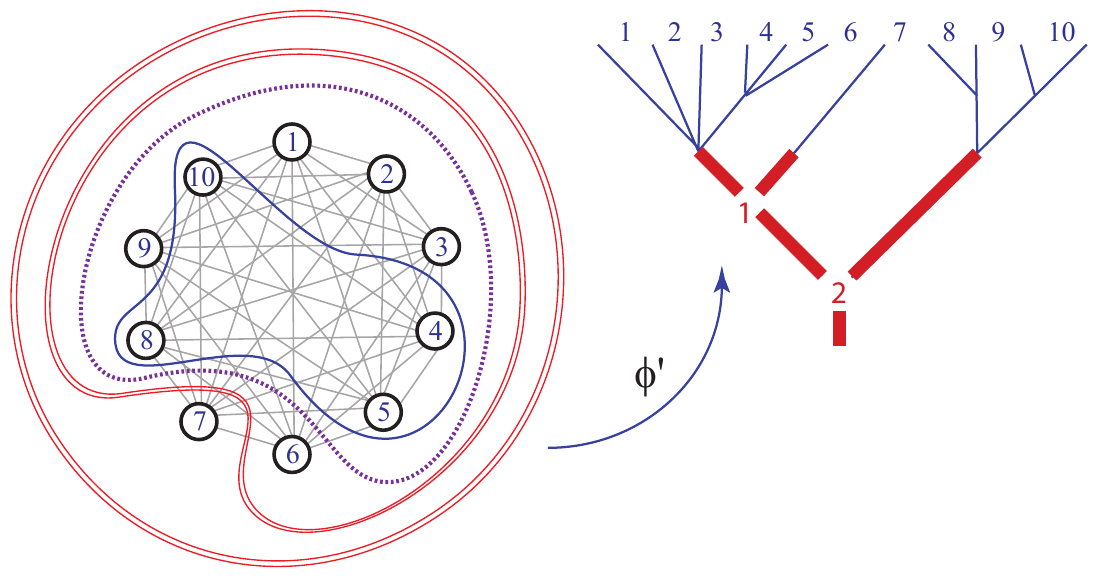}
\caption{A marked tubing on the complete graph, representing an
element of the complete-graph-composihedron (no structure is shown
inside the thin tube) and its bijective image in the forest of
corollas over a weakly ordered tree.}\label{fig:stello_bij_composi}
\end{figure}

\begin{figure}[ht!]\centering
                  \includegraphics[width=\textwidth]{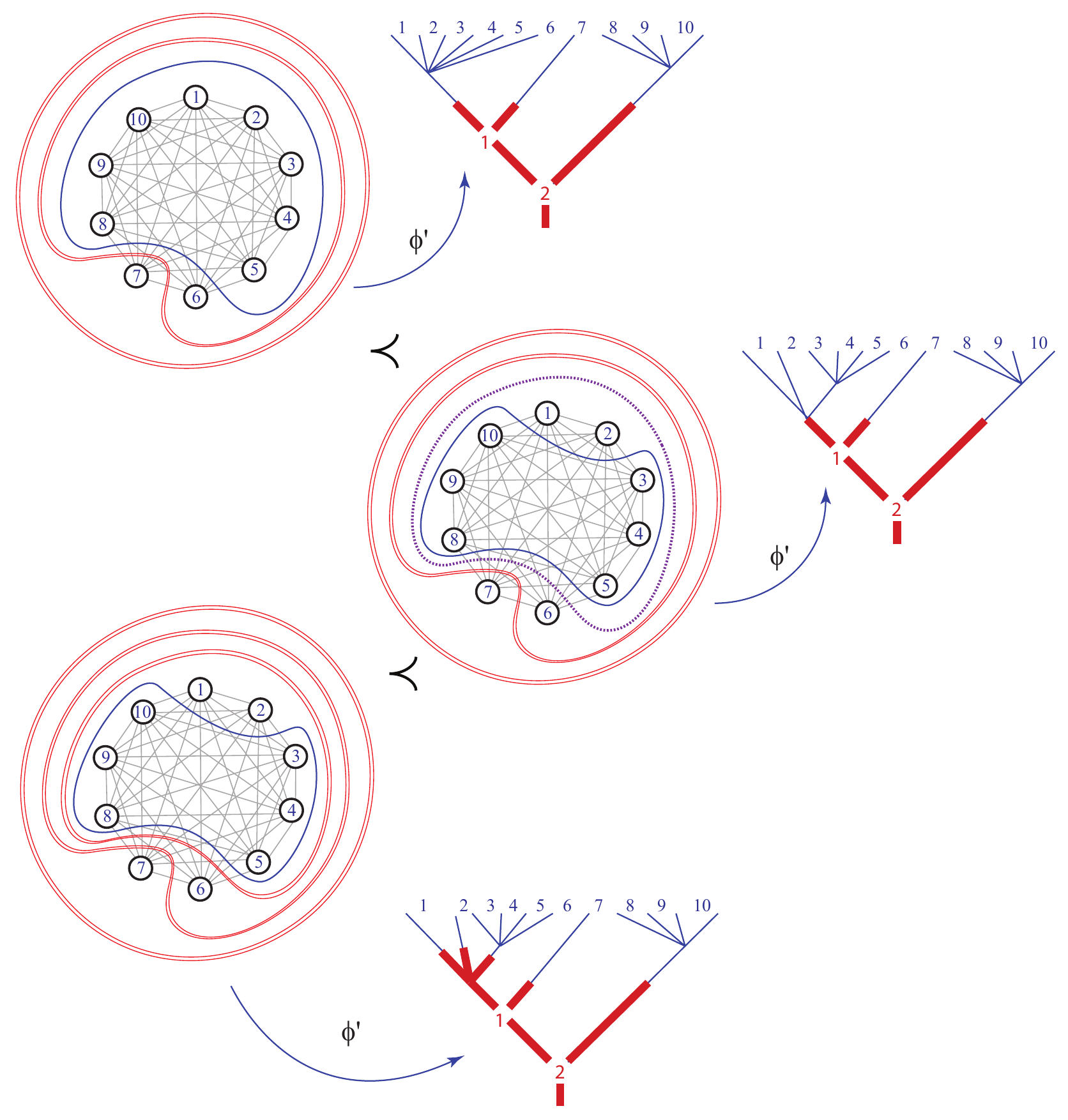}
\caption{The first (upper) example of $\phi'$ has source and target
related to Figure~\ref{fig:stello_bij_composi} by a adding a tube as
in Figure~\ref{f:poset_composi}(d) and growing an unpainted edge
from a half-painted node, and then forgetting structure in both
pictures. The other two relations are from
Figure~\ref{f:poset_composi}(a) and (b) (right to left).
}\label{fig:stello_bij_composi_ord}
\end{figure}

Moreover, we will show that the poset of painted trees made by
grafting a weakly ordered forest to a base corolla is the face poset
of a polytope. It turns out that for painted trees with $n+1$ leaves
this polytope is again the graph-associahedron $\KG$ where $G$ is
the star graph $St_{n}.$

 First, however, we show a bijection from the range-quotients of the
  complete graph multiplihedron (the complete graph-cubeahedron) to
the weakly ordered forests grafted to corollas.
\begin{theorem}\label{stella2}
The poset of $n+1$-leaved weakly ordered forests grafted to corollas
is combinatorially equivalent to the graph-cubeahedron for a
complete graph $\K_n.$
\end{theorem}

\begin{proof}
This proof follows the pattern of the previous one, so we leave most
of it to the reader. Note that any nodes outside of the broken $t_0$
tube in the input correspond to the painted corolla base tree of the
ouput. The tubing inside a largest thin tube (which contains a
clique) in the input corresponds to the gaps (between leaves) that
end in nodes of the unpainted weakly ordered forest of the output.
Nodes that are inside a broken tube but not inside any thin tube
determine the gaps that coincide with the paint line. An example of
the bijection is seen in Figure~\ref{fig:stello_bij_cuba}.

The fact that this bijection preserves the ordering is seen just as
in the proof of Theorem~\ref{lift}. 3-dimensional examples are seen in
Figure~\ref{fig: biject4}. See Figure~\ref{chains} for some
isomorphic chains. \end{proof}

\begin{figure}[ht!]\centering
                  \includegraphics[width=4.75in]{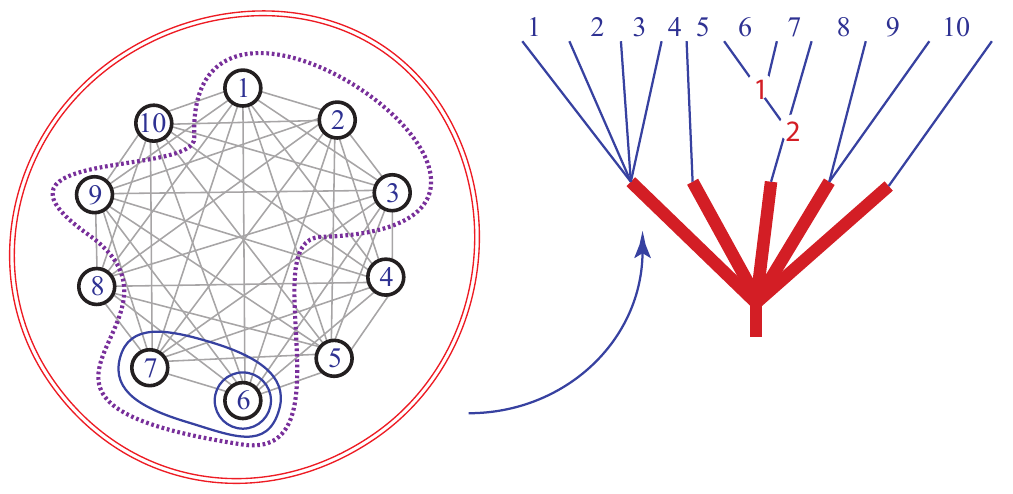}
\caption{A marked tubing on the complete graph, representing an
element of the complete-graph-cubahedron (no structure is shown
outside the broken tube) and its bijective image in the weakly
ordered forests over corollas.}\label{fig:stello_bij_cuba}
\end{figure}

\begin{figure}[h]\centering
                  \includegraphics[width=\textwidth]{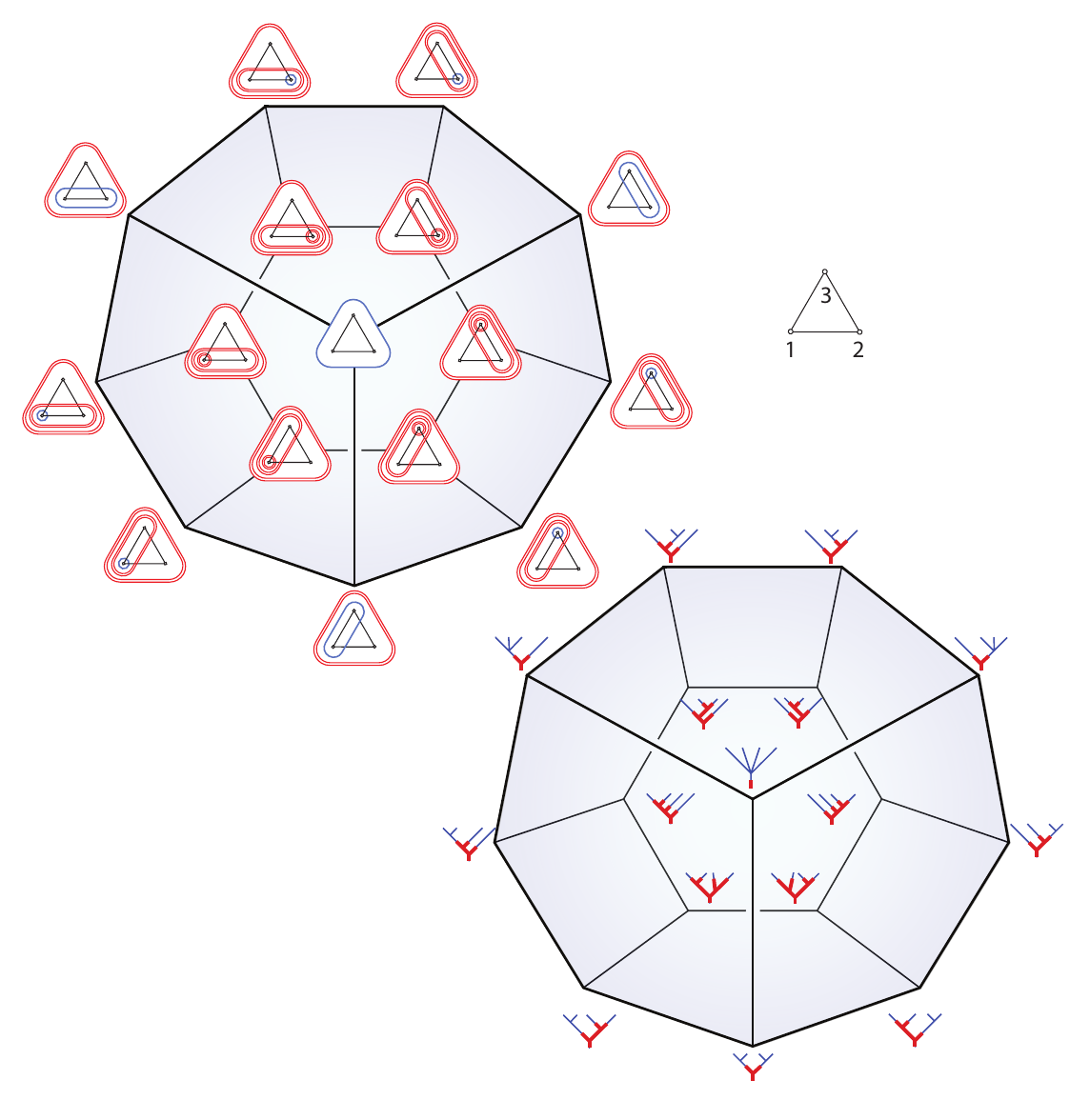}
\caption{Another stellahedra bijection via Theorem~\ref{lift}: when
$K_n$ is the complete graph then $\J K_{n_d}$ (the complete graph
composihedron) is the stellohedron.}\label{fig: biject2}
\end{figure}


\begin{figure}[h]\centering
                  \includegraphics[width=\textwidth]{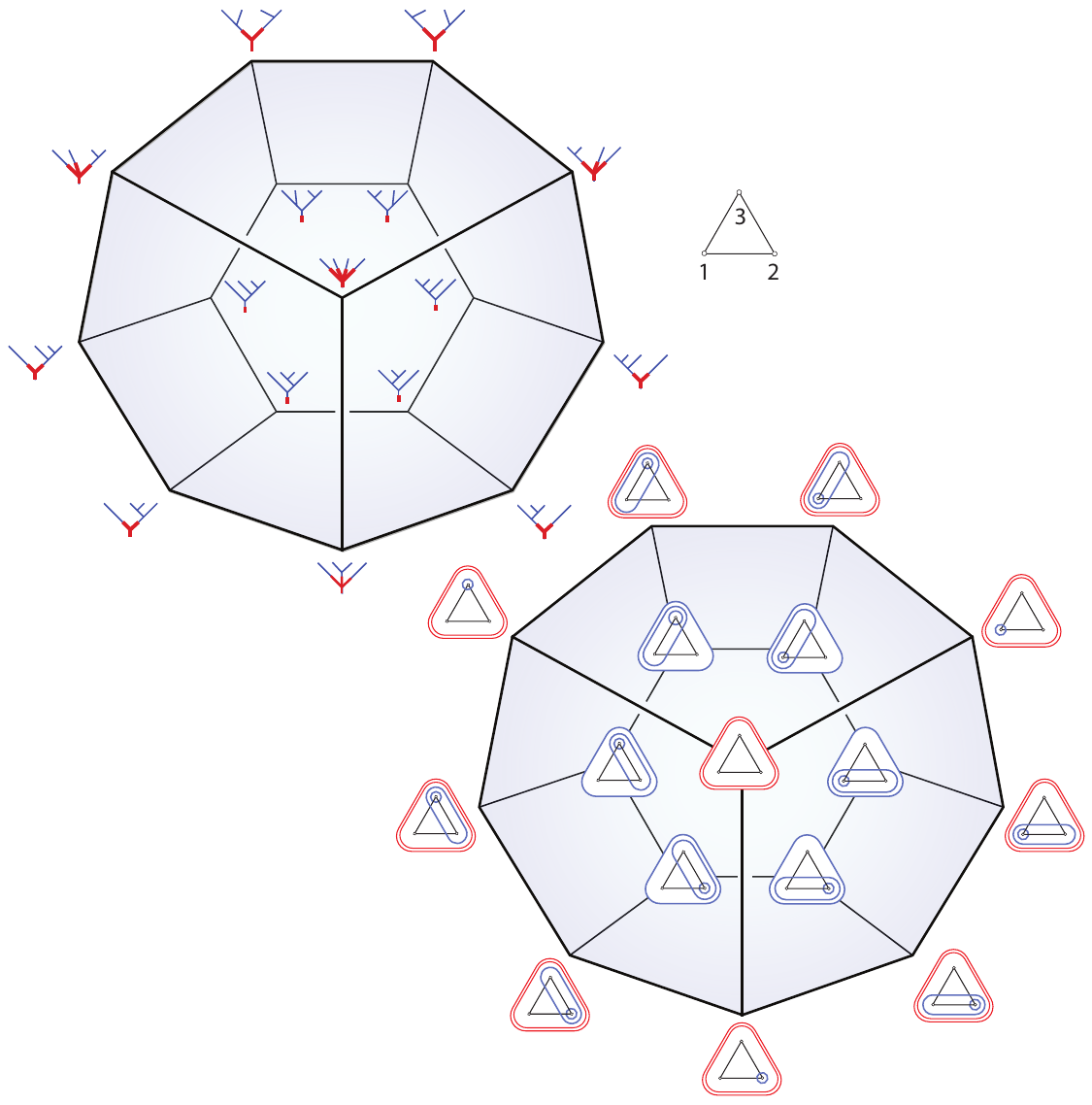}
\caption{Another stellahedra bijection via Theorem~\ref{stella3}:
the composition of ordered forests with corollas, seen in bijection
with the complete graph cubeahedron.}\label{fig: biject4}
\end{figure}

We now can finish with the following:

\begin{theorem}\label{stella3} The weakly ordered forests grafted to corollas are
isomorphic to the stellohedra. \end{theorem}

\begin{proof}
  By Theorem 62 of \cite{pilaud2}, the graph-cubeahedron for a
complete graph $\K_n$ is combinatorially equivalent to the
stellohedron for the star-graph $St_{n+1}.$ Here is a brief
description of the poset isomorphism described in that paper: if the
star graph $St_{n+1}$ has node 0 as its center, and the nodes of the
complete graph are $1,\dots,n$, then a square tube on the complete
graph is mapped to itself, as a round tube; and round tubes on the
complete graph are mapped to their complement plus the node 0 on the
star graph.  We demonstrate this isomorphism in Figure~\ref{chains}.
Thus the theorem is shown, by composition with the isomorphism in
our Theorem~\ref{stella2}.
\end{proof}

\begin{figure}[h]\centering
                  \includegraphics[width=\textwidth]{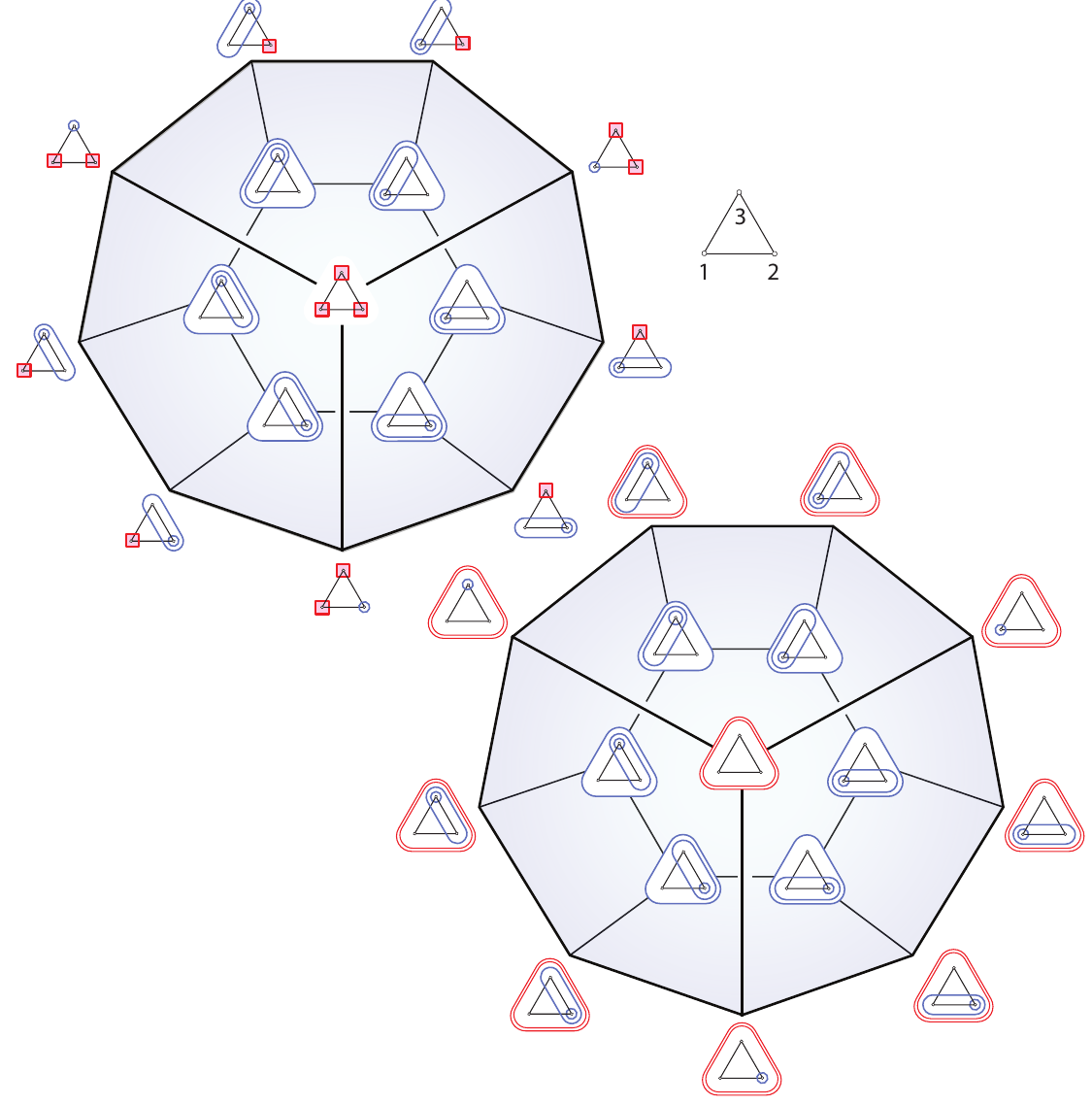}
\caption{Another stellahedra bijection: the complete graph
cubeahedron indexed by design tubings.}\label{fig: biject5}
\end{figure}

\begin{figure}[h]\centering
                  \includegraphics[width=5in]{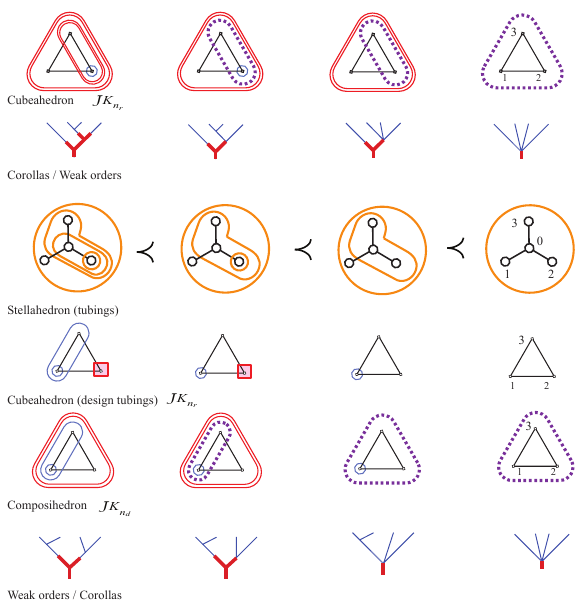}
\caption{Here we bring together six isomorphic flags from the 3-dimensional
stellohedra shown above.}\label{chains}
\end{figure}

\subsection{Pterahedra}
\medskip

The aim of this subsection is to prove that the poset of painted
trees made by grafting a forest of plane rooted trees to a weakly
ordered base tree is the face poset of a polytope. It turns out that
for painted trees with $n+1$ leaves this polytope is the
graph-associahedron $\K F_{1,n}$ where the fan graph $F_{1,n}$ is
the suspension of the path graph $P_n$.

More precisely, the fan graph $F_{1,n}$ is defined as follows: the
set of nodes of $F_{1,n}$ is $[n]_0$, while an edge of the graph is
given either by the pair $\{i,i+1\}$, for some $i= 1,\dots,n-1$, or
by a pair $\{0,i\}$, for some $i=1,\dots,n$.

\begin{theorem}\label{ptera}
The poset of tubings on the fan graph $F_{1,n}$ is isomorphic to the
poset of $n$-leaved forests of plane trees grafted to weakly ordered
trees.
\end{theorem}
\begin{proof}

Recall that any tubing $T$ of the fan graph includes a unique
smallest tube $t_0$ which contains the node $0$.  As the node $0$ is
adjacent to all other nodes, the other tubes of $T$ are either
contained in $t_0$ or contain $t_0$. The tubes contained in $t_0$
form a tubing of a graph which is a (possibly) disconnected set of
line graphs. The tubes containing $t_0$ form a tubing on the
reconnected
 complement of $t_0$, which is the complete
graph on the nodes which do not belong to $t_0$.

There exists a canonical bijection between the poset of weakly
ordered trees with $n+1$ leaves and the poset ${\mbox {Tub}(K_n)}$
of tubings on the complete graph: pictured in Figures~\ref{perm1}
and~\ref{perm2}. The restriction of this map to the set of plane
trees, gives a bijection between this poset and the poset ${\mbox
{Tub}(P_n)}$ of tubings on the path graph.

Thus the bijection from the poset ${\mbox {Tub}(F_{1,n})}$ of
tubings on the fan graph to our set of painted trees is obtained
from the bijection between ${\mbox {Tub}(K_m)}$ and the set of
weakly ordered trees, together with the bijections between the set
${\mbox {Tub}(P_m)}$ of tubings on the path graph  and the set of
plane rooted trees with $m+1$ leaves, for $m\geq 1$. The tube $t_0$
plays the same role as the paint line in the corresponding tree. The
nodes $1,\dots,n$ of the fan graph $F_{1,n}$ correspond to the gaps
(between leaves) $1,\dots,n$ of the plane rooted tree. For any
tubing $T\in {\mbox {Tub}(F_{1,n})}$ tubing outside of $t_0$ maps to
the painted weakly ordered tree, the tubings inside $t_0$ map to the
unpainted trees, and nodes that are inside $t_0$ but not inside any
smaller tube determine the gaps that coincide with the paint line.
Examples are seen in Figure~\ref{ptera_bij_big}.

\begin{figure}[hbt!]
                  \centering
                  \includegraphics[width=\textwidth]{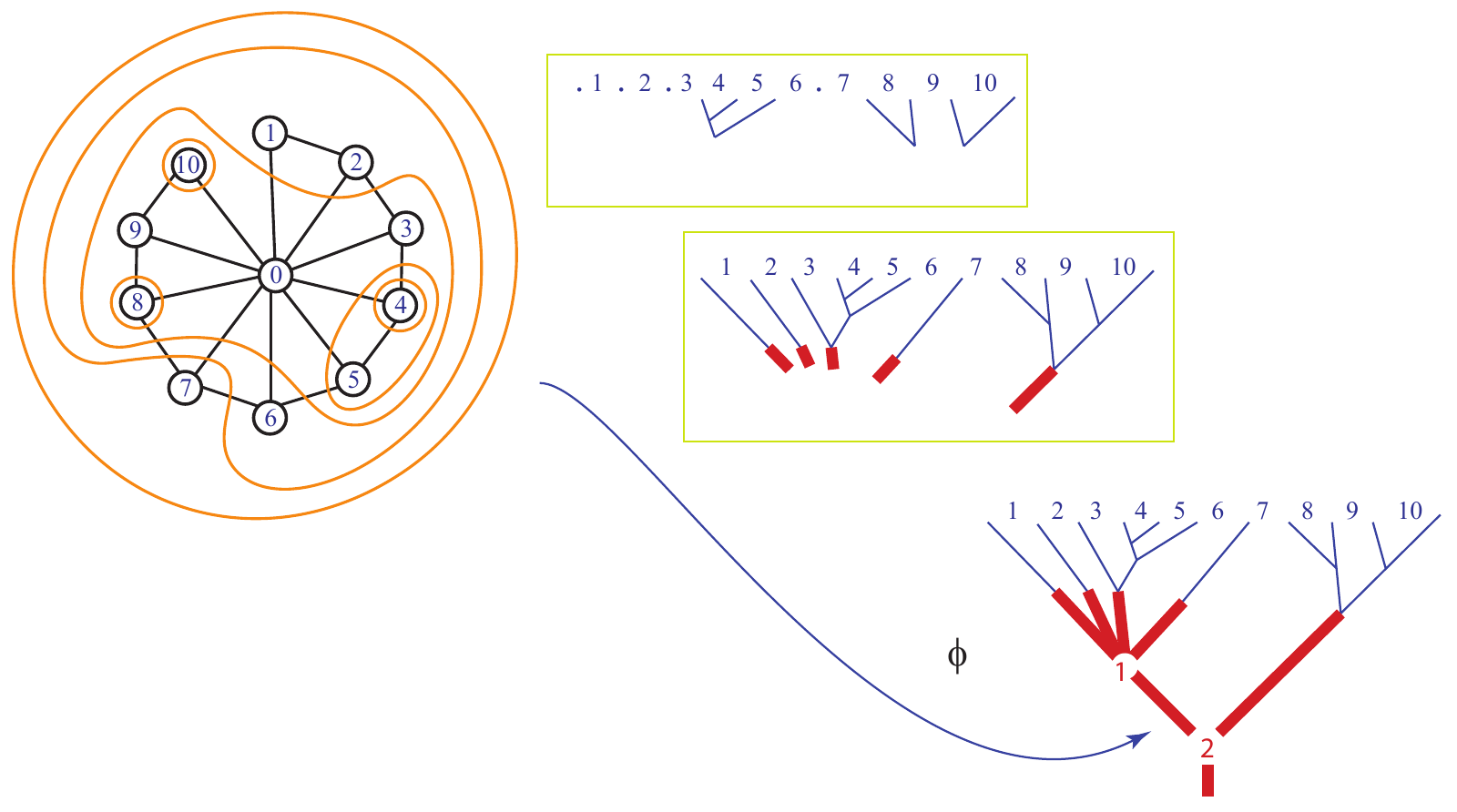}
     \caption{Example of the bijection in Theorem~\ref{ptera}. The steps are shown left to right.}\label{ptera_bij_big}
  \end{figure}

 The fact that this bijection preserves the ordering
follows easily from the definitions. Just note that adding a tube to
a tubing of the fan graph corresponds to growing an internal edge in
the tree. Adding a tube far outside of $t_0$ corresponds to growing
an edge in the painted base. Adding a tube containing node 0 just
inside $t_0$ (so that it becomes the new $t_0$) corresponds to
growing painted edge(s) from a half-painted node. Adding a tube just
inside $t_0$ that does not contain node 0 corresponds to growing
unpainted edge(s) from a half-painted node. Adding a tube further
inside of $t_0$ (that does not contain node 0) corresponds to
growing an edge in the unpainted forest.
\end{proof}
The isomorphism in 3 dimensions is shown pictorially in
Figure~\ref{plethy_flip}.

\begin{figure}[h]\centering
                  \includegraphics[width=\textwidth]{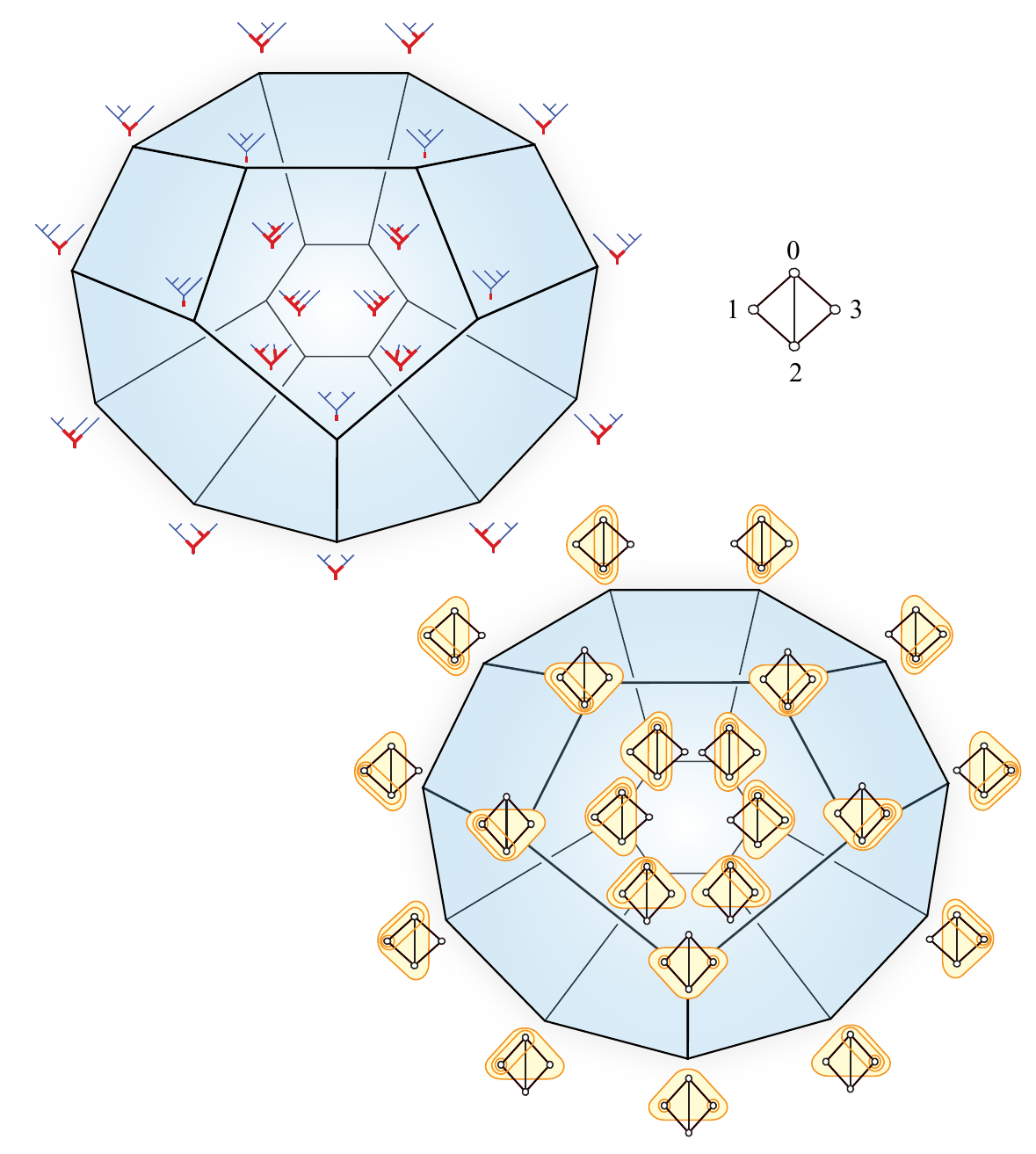}
\caption{Two pictures of the pterahedron, via the bijection from
Theorem~\ref{ptera}.}\label{plethy_flip}
\end{figure}

\subsection{Enumeration}
As well as uncovering the equivalence between the pterahedra and the
fan-graph associahedra, we found some new counting formulas for the
vertices and facets of the pterahedra.

First the vertices of the pterahedra, which are forests of binary
trees grafted to an ordered tree. If there are $k$ nodes in the
ordered, painted portion of a tree, then there are:
\begin{itemize}
\item $k!$ ways to make the ordered portion of this tree with $k$ nodes,
\item $k+1$ leaves of the ordered portion of this tree, and
\item $n-k$ remaining nodes to be distributed among the $k+1$ binary trees that will go on the leveled/painted leaves.
\end{itemize}

\noindent Thus the number of vertices of the pterahedron, labeled by
trees with $n$ nodes, is:

\begin{large}
\begin{center}
$v(n) = \sum\limits_{k=0}^n [~k! \sum\limits_{\substack{\gamma_0 +
\dots + \gamma_k\\= n-k}} ~(\prod\limits_{i=0}^{k}C_{\gamma_i})~]$.
\end{center}
\end{large}
where $C_{j}$ is the $j$th Catalan number. As an example, the number
of trees with $n=4$ is:
\begin{align*}
& 0!~[C_4]\\
&+ 1!~[C_3C_0 + C_2C_1 + C_1C_2 + C_0C_3]\\
&+ 2!~[C_2C_0C_0 + C_1C_1C_0 + C_1C_0C_1 + C_0C_2C_0 + C_0C_1C_1 + C_0C_0C_2]\\
&+ 3!~[C_1C_0C_0C_0 + C_0C_1C_0C_0 + C_0C_0C_1C_0 + C_0C_0C_0C_1]\\
&+ 4!~[C_0C_0C_0C_0C_0]\\
&= 1(14) + 1(14) + 2(9) + 6(4) + 24(1)\\
&= 14 + 14 + 18 + 24 + 24\\
&= 94
\end{align*}

We have computed the cardinalities for $n = 0$ to 9 and they are
shown in Table~\ref{ToY_set_sizes}.

\begin{table}[h]
\caption{The number $v(n)$ vertices of the pterahedra, $n=0$ to
9}\label{ToY_set_sizes}
\begin{center}
\begin{tabular}{|c|c||c|c|}
\hline
~~$n$~~ & ~~$v(n)$~~ & ~~$n$~~ & ~~$v(n)$~~\\
\hline \hline
0 & 1 & 5 & 464 \\
\hline
1 & 2 & 6 & 2652 \\
\hline
2 & 6 & 7 & 17,562 \\
\hline
3 & 22 & 8 & 133,934 \\
\hline
4 & 94 & 9 & 1,162,504 \\
\hline
\end{tabular}
\end{center}
\end{table}
There does not seem to be an entry for this sequence in the OEIS.
\FloatBarrier

Examination of the computations of $v(n)$ leads to an interesting
discovery. If we strip off the factorial factors in $v(n)$
 and build a triangle of just the sums of $C_{\gamma_i}$ products, it appears we are building the Catalan triangle. \\

\begin{center}
\setlength{\tabcolsep}{10pt}
\begin{tabular}{c c c c c c c c c c}
1 \\
1 & 1 \\
2 & 2 & 1 \\
5 & 5 & 3 & 1 \\
14 & 14 & 9 & 4 & 1 \\
42 & 42 & 28 & 14 & 5 & 1 \\
132 & 132 & 90 & 48 & 20 & 6 & 1 \\
429 & 429 & 297 & 165 & 75 & 27 & 7 & 1 \\
1430 & 1430 & 1001 & 572 & 275 & 110 & 35 & 8 & 1 \\
4862 & 4862 & 3432 & 2002 & 1001 & 429 & 154 & 44 & 9 & 1 \\ \\
\end{tabular}
\end{center}

For example,
\begin{center}
$v(4)
=94=0!(\textbf{14})+1!(\textbf{14})+2!(\textbf{9})+3!(\textbf{4})+4!(\textbf{1})$
\end{center}
leads to the values of the $n=4$ row in the triangle above. In fact,
Zoque \cite{Zoque} states that the entries of the Catalan triangle,
often called {\it ballot numbers}, count ``the number of ordered
forests with $m$ binary trees and with total number of $\ell$
internal vertices'' where $m$ and $\ell$ are indices into the
triangle. These forests describe exactly the sets of binary trees we
are grafting onto the leaf edges of individual leveled trees, which
are counted by the sums of $C_{\gamma_i}$ products. Thus we know
that the ballot numbers are equivalent to the sums of $C_{\gamma_i}$
products and can be used in the calculation of $v(n)$ for all values
of $n$.

The formula for the entries of the Catalan triangle
 leads to a simpler formula for $v(n)$,
namely

\begin{large}
\begin{center}
$v(n) = \sum\limits_{k=0}^n k! \frac{(2n-k)!(k+1)}{(n-k)!(n+1)!}$,
\end{center}
\end{large}

Lastly, by considering the Catalan triangle as a matrix as in
\cite{Barry}, we can say that
 the sequence of cardinalities $v(n)$ for all $n$ is the Catalan transform of the factorials
 $(n-1)!$. This means that the ordinary generating function for
 $v(n)$ is: $$\sum_{k=1}^{\infty}(k-1)!\left(\frac{1-\sqrt{1-4x}}{2}\right)^k.$$

Also, it will be helpful to have a formula for the number of facets
for the fan graphs, $F_{m,n}$.
 Recall $F_{m,n}$ is defined to be the graph join of $\bar{K}_m$ the edgeless graph on $m$ nodes, and $P_n$
  the path graph on $n$ nodes. Thus, $\Fmn$ has $m+n$ vertices, $n$ of which comprise a subgraph isomorphic
   to the path graph on $m$ nodes, the other $n$ vertices connected to each of these $m$. Thus, $\Fmn$ has
   $m-1+mn = m(n+1)-1$ edges.

Now, counting tubes in this case is again a matter of counting
subsets of vertices whose induced subgraph is connected. The
structure of $\Fmn$ makes it useful to let $V_m$ denote those
vertices coming from the edgeless graph of $m$ nodes, and likewise
$V_n$ those from the path graph on $n$ nodes.

 It is clear that some tubes are simply tubes of the path graph $P_n$,  hence
  there are at least $\frac{n(n+1)}{2}-1$ tubes. We must not forget that $V_n$ is itself now a
    tube since it is a proper subset of nodes of $\Fmn$. These tubes include every subset of $V_n$ that is a valid tube of $\Fmn$.

 It is simple to see that the only subsets of $V_m$ that are valid tubes are precisely the singletons,
 since no pair of vertices in $V_m$ are connected by an edge. Thus $\Fmn$ has at least $\frac{n(n+1)}{2}+m$ tubes.

 The remaining possibility for tubes must include at least one node from $V_m$ as well as at
  least one node from $V_n$. This produces all (possibly improper) tubes, since any subset of
   $V=V_m \cup V_n$ satisfying this criterion is connected. It is straightforward to see that
    there are exactly $(2^m - 1)(2^n - 1)$ tubes arising in this fashion. Now, however, we must
    subtract 1 from the above since we have allowed ourselves to count $V$ as a tube, although it is not proper.

 Hence we count $$\frac{n(n+1)}{2} + (2^m - 1)(2^n - 1) + m - 1$$ as the number of tubes of $\Fmn$ and
 the number of facets of the corresponding graph associahedron.
For the pterahedra, where $m=1$, the formula becomes:
$$\frac{n(n+1)}{2} + 2^n - 1.$$ Interestingly, this is the same
number of facets as possessed by the multiplihedron $\J(n)$, as seen
in \cite{multi}, where we enumerate the facets by describing their
associated trees.

The number of vertices of the stellahedron is worked out in several
places, including \cite{post2}, where the formula is given:

\[
   v(n)\ =\ \sum_{k=0}^n k! {n \choose k}\ =\
   \sum_{k=0}^n n!/k!\,,
\]
which is sequence A000522 in the OEIS~\cite{Slo:oeis}. This is the
binomial transform of the factorials.

Now for the facets. We will use the following, possibly well-known
\begin{lemma}
The bipartite graph associahedron $\mathcal{K}K_{m,n}$ has
$2^{m+n}+(m+n)-(2^m+2^n)$ facets.
\end{lemma}
To see this, we will count subsets of nodes which give valid tubes.
We will over-count and then correct. Let $K_{m,n}=(V_1\cup V_2,E)$
where $|V_1|=m$ and $|V_2|=n$. Note that the only subsets $S$ of
nodes which do not give valid tubes are $S$ such that $S\subseteq
V_1$ (or $V_2$) with $|S|>1$. These are simply edgeless graphs with
$|S|>1$ nodes, and do not constitute valid tubes. For the moment,
let $S\subseteq V_1$ where $|V_1|=m.$

Let $M=\#\{S\subseteq V_1: |S|>1\}$. Now
$$M=\sum_2^k \begin{pmatrix}m\\k\end{pmatrix}=2^m-(m+1)$$
and by the above, there are $2^n-(n+1)$ ``bad subsets'' that can be
chosen from $V_2$ for a total of $2^m+2^n-(m+n-2)$ bad subsets of
$V=V_1 \cup V_2$.

It follows that we may choose any of $2^{m+n}-2$ proper, nonempty
subsets of $V$, and subtract
 off the bad choices for the total number of tubes. Thus, $K_{m,n}$ has

\begin{eqnarray*}2^{m+n}-2-(2^m+2^n-(m+n-2))&=
2^{m+n}-2-(2^m+2^n)+m+n+2\\
&=2^{m+n}+(m+n)-(2^m+2^n)\end{eqnarray*} tubes.

Note that, for $K_{1,n-1}$, we see that star graphs on $n$ nodes
have
\begin{eqnarray*}
2^{n}+n-(2^{n-1}+2)&=2^n-2^{n-1}+n-2\\
&=2^{n-1}(2-1)+n-2\\
&=2^{n-1}+n-2
\end{eqnarray*}
tubes.

\section{Additional shuffle product on the Stellohedra}\label{s:alg}
\subsection{Preliminaries}

For $n\geq 1$, we denote by $\Sigma_n$ the group of permutations of
$n$ elements. For any set $U=\{ u_1,\dots ,u_n\}$ with $n$ elements,
an element $\sigma \in \Sigma_n$ acts naturally on the left on $U$
and induces a total order $u_{\sigma^{-1} (1)} < \dots < u_{\sigma
^{-1}(n)}$ on $U$.

For nonnegative integers $n$ and $m$, let ${\mbox {Sh}(n,m)}$ denote
the set of $(n,m)$-shuffles, that is the set of permutations
$\sigma$  in the symmetric group $\Sigma_{n+m}$ satisfying that:
$$\sigma (1) < \dots <\sigma (n)\qquad {\rm and}\qquad \sigma(n+1)< \dots <\sigma (n+m). $$
For $n=0$, we define ${\mbox {Sh}(0,m)}:= \{ 1_m\} =: {\mbox
{Sh}(m,0)}$, where $1_m$ is the identity of the group $\Sigma_m$.
More in general, for any composition $(n_1,\dots ,n_r)$ of $n$, we
denote by ${\mbox {Sh}(n_1,\dots ,n_r)}$ the subset of all
permutations $\sigma $ in $\Sigma_n$ such that $\sigma (n_1+\dots
+n_i+1)< \dots  < \sigma (n_1+\dots +n_{i+1})$, for $0\leq i\leq
r-1$.

The concatenation of permutations $\times : \Sigma_n\times
\Sigma_m\hookrightarrow \Sigma_{n+m}$ is the associative product
given by:
\[\sigma \times \tau (i):=\left\{\begin{array}{lr}\sigma (i),& \ {\rm for}\ 1\leq i\leq n,\\
\tau(i-n)+n,&\ {\rm for}\ n+1\leq i\leq n+m,\end{array}\right\} \]
for any pair of permutations $\sigma \in \Sigma_n$ and $\tau\in
\Sigma_m$.

The well-known associativity of the shuffle states that:
\begin{equation}\nonumber {\mbox {Sh}(n+m, r)}\cdot ({\mbox {Sh}(n,m)}\times 1_r) =  {\mbox {Sh}(n,m,r)} = {\mbox {Sh}(n,m+r)}\cdot (1_n\times {\mbox {Sh}(m, r)}),\end{equation}
where $\cdot $ denotes the product in the group $\Sigma_{p+q+r}$.
\bigskip

For $n\geq 1$, recall that the star graph $St_n$ is a simple
connected graph with set of nodes ${\mbox {Nod}(St_n)} =\{ 0,1,\dots
,n\}$ and whose edges are given by $\{0,i\}$ for $i=1,\dots, n.$

That is $St_n$ is the suspension of the graph with $n$ nodes and no
edges.
\medskip
\begin{notation} \label{tubingstello} For any maximal tubing $T$ of $St_n$ such that $T\neq \{ \{1\},\dots ,\{n\}\}$, there exists a unique integer $0\leq r\leq n$, a unique family of integers  $1\leq u_1<\dots < u_r\leq n$ and an order $\{u_{r+1},\dots, u_{n}\}$ on the set $\{1,\dots ,n\} \setminus \{ u_1,\dots , u_r\}$ such that
$$T=\{ \{u_1\},\dots ,\{u_r\}, t_0 ,
 t_0 \cup \{ u_{r+1}\}, \dots , t_0\cup \{u_{r+1},\dots ,u_{n-1}\}\},$$
 where $ t_0 :=\{0, u_1,\dots ,u_r\}$.

We denote such tubing $T$ by ${\mbox {Tub}_r(u_1,\dots ,u_{n})}$,
where $u_n$ is the unique vertex which does not belong to any tube
of $T$. We denote the tubing $\{ \{1\},\dots ,\{n\}\}= {\mbox
{Tub}_n(1,\dots ,n)}$ by ${\mbox {Tub}_n}$.\end{notation} \medskip

\bigskip

Recall this example of a tree splitting from Section~\ref{hopf}.

\includegraphics[width=\textwidth]{treesplit_examp.pdf}

Note that a $k$-fold splitting, which is given by a size $k$
multiset of the $n+1$ leaves, corresponds to a $(k,n)$ shuffle. The
corresponding shuffle is described as follows: $\sigma(i)$ for $i\in
1,\dots,k$ is equal to the sum of the numbers of leaves in the
resulting list of trees 1 through $i$. For instance in the above
example the shuffle is $(3,4,8,10,1,2,5,6,7,9,11).$

In the above Section~\ref{hopf} the product of two painted trees is
described as a sum over splits of the first tree, where after each
split the resulting list of trees is grafted to the leaves of the
second tree. Thus this product can be seen as a sum over shuffles.
In fact if we illustrate the products using the graph tubings, then
shuffles are actually more easily made visible than splittings.

Here we mainly want to show some examples of the products, since we
have already proven the structure. For that purpose we show single
terms in the product, each term relative to a shuffle.
Figures~\ref{f:tree_mult_with_stello}
and~\ref{f:tree_mult_with_ptera} show terms in two sample products,
relative to the given shuffle, and illustrating the splitting as
well. In Figures~\ref{f:tube_mult_with_stello}
and~\ref{f:tube_mult_with_ptera} we show the same sample product
terms, pictured using the tubings on the star graphs and fan graphs.

\begin{figure}[htb!]\centering
                  \includegraphics[width=\textwidth]{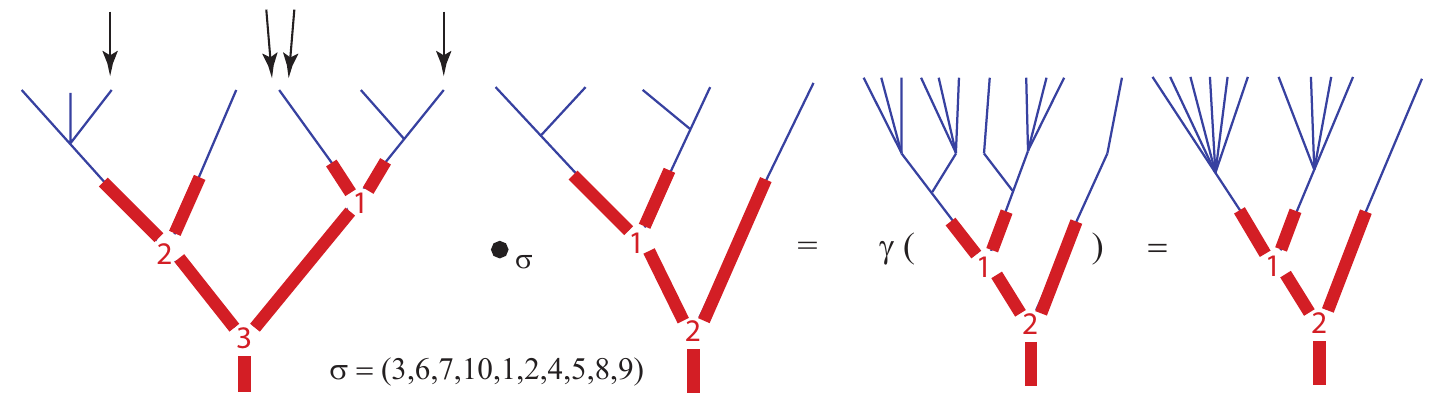}
\caption{One term in the product defined in Section~\ref{hopf},
relative to the given shuffle.}\label{f:tree_mult_with_stello}
\end{figure}

\begin{figure}[htb!]\centering
                  \includegraphics[width=\textwidth]{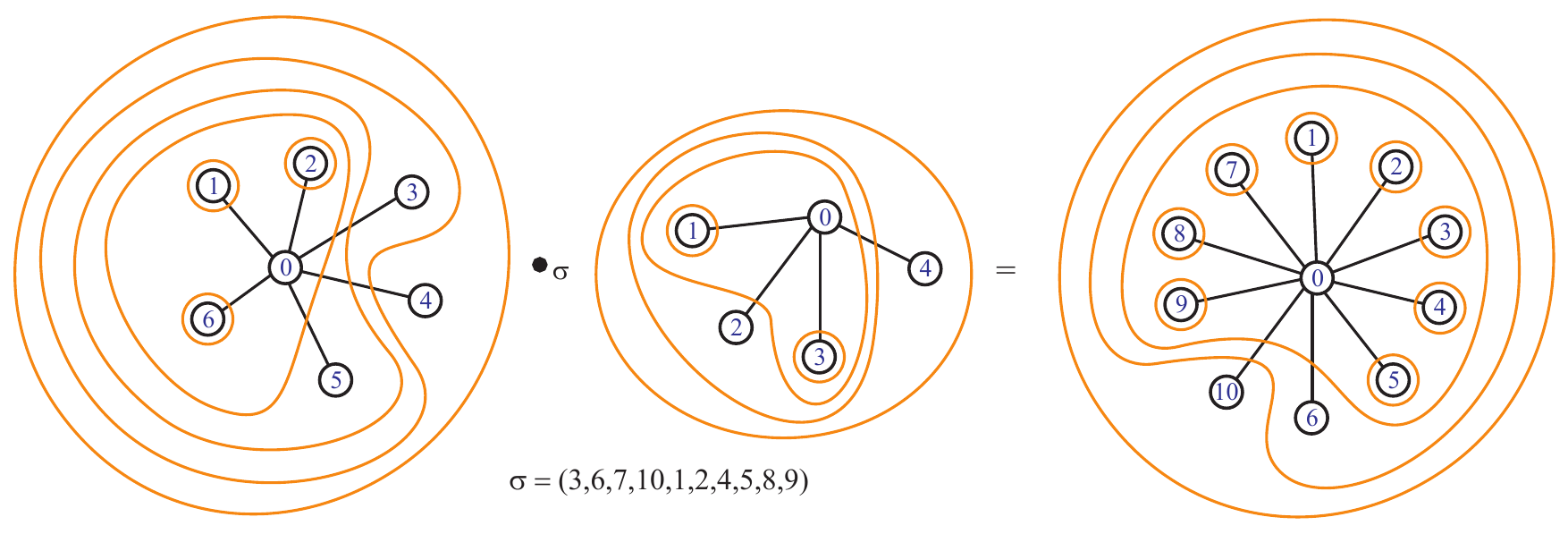}
\caption{The same product as pictured in
Figure~\ref{f:tree_mult_with_stello}, shown here using tubings on
the star graphs.}\label{f:tube_mult_with_stello}
\end{figure}

\begin{figure}[htb!]\centering
                  \includegraphics[width=\textwidth]{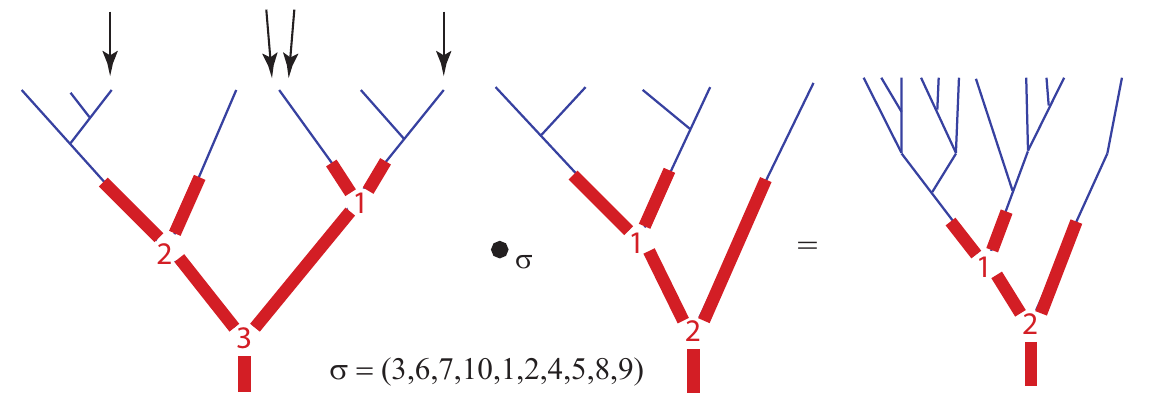}
\caption{One term in the product defined in Section~\ref{hopf},
relative to the given shuffle.}\label{f:tree_mult_with_ptera}
\end{figure}

\begin{figure}[htb!]\centering
                  \includegraphics[width=\textwidth]{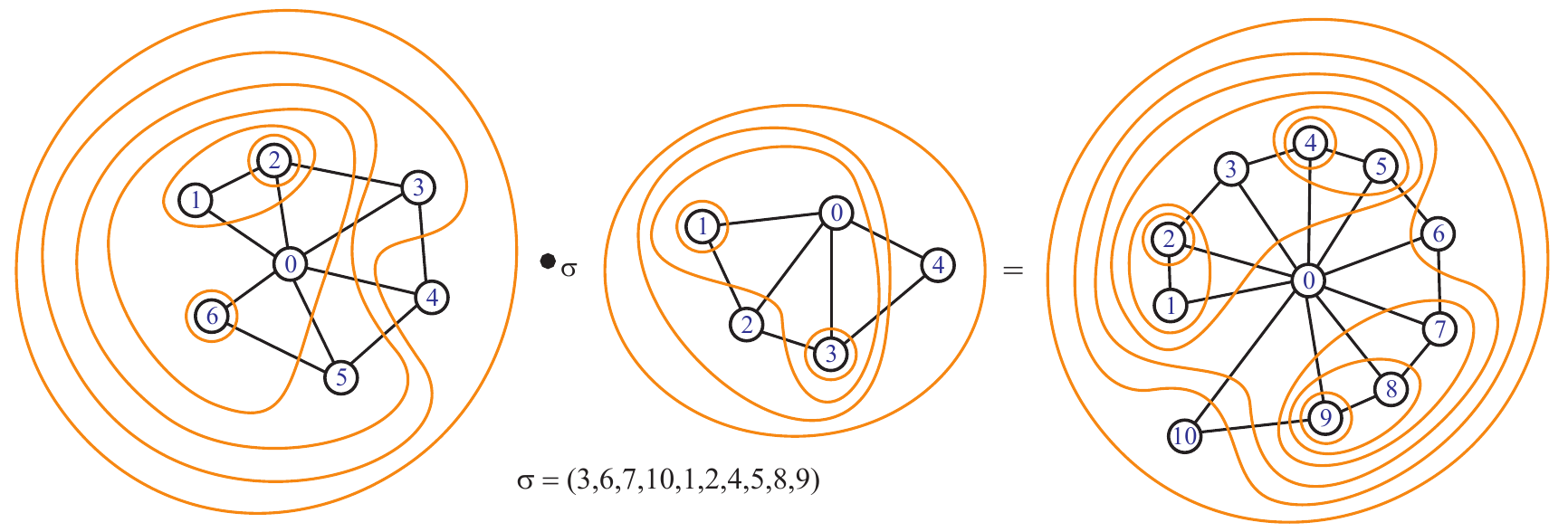}
\caption{The same product as pictured in
Figure~\ref{f:tree_mult_with_ptera}, shown here using tubings on the
star graphs.}\label{f:tube_mult_with_ptera}
\end{figure}

\section{Alternative product on the Stellohedra vertices}

For $n\geq 1$, we denote by $[n]$ the set $\{ 1,\dots ,n\}$. For any
set of natural numbers $U = \{ u_1,\dots ,u_k\}$ and any integer
$r\in {\mathbb Z}$, we denote by $U + r$ the set $\{ u_1+r,\dots ,
u_k+r\}$.
\medskip

Let $G$ be a simple finite graph with set of nodes ${\mbox {Nod}(G)}
= \{ j_1,\dots , j_r\}\subseteq {\mathbb N}$, we denote by $G+n$ the
graph $G$, with the set of nodes colored by ${\mbox {Nod}(G)} + n$,
obtained by replacing the node $j_i$ of $G$ by the node $j_i + n$,
for $1\leq i\leq r$.
\medskip

Let $G$ be a simple finite graph, whose set of nodes is $[n]$, we
identify a tube $t = \{ v_1,\dots , v_k\}$ of $G$ with the tube
$t(h) = \{ v_1+h,\dots ,v_k+h\}$ of $G+h$. For any tubing $T = \{
t_i\}$ of $G$, we denote by $T(h)$ the tubing $\{ t_i(h)\}$ of
$G+h$.
\bigskip
In the present section, we use the shuffle product, which defines an
associative structure on the vector space spanned by all the
vertices of permutohedra, in order to introduce associative products
(of degree ${-}1$) on the vector spaces spanned by the vertices of
stellohedra.
\bigskip

\begin{definition} \label{assstello} Let $T$ be a maximal tubing of $\St_n$ and ${ V}$ be a maximal tubing of $\St_m$ such that $T={\mbox {Tub}_r(u_1,\dots ,u_{n})}$ and $V = {\mbox {Tub}_s(v_1,\dots ,v_{m})}$. For any $(n-r, m-s)$- shuffle $\sigma\in S_{n+m{-}r{-}s}$ define the maximal tubing $T*_{\sigma }{ V}$ of $S_{n+m}$ as follows:
$$T*_{\sigma}{ V} := {\mbox {Tub}_{r+s}(u_1,\dots ,u_r,v_1{ +n},\dots ,v_s{ +n}, w_{\sigma ^{-1}(1)}, \dots ,w_{\sigma^{-1} (n+m{-}(r+s))})},$$
where:
$$w_i:=\left\{\begin{array}{lr}u_{r+i},\ {\rm for}\ 1\leq i\leq n-r,\\
v_{s+i+r{-}n}{ +n},\ {\rm for}\ n-r < i\leq
n+m-r-s.\end{array}\right\}$$ If $T = {\mbox {Tub}_n}$, then $\sigma
= 1_m$ and
$${\mbox {Tub}_n}*_{1_m}{ V} = {\mbox {Tub}_{n+s}(1,\dots ,n,v_1+n,\dots ,v_m+n)}.$$
In a similar way, we have that
$$T*_{1_n}{\mbox {Tub}_m} = {\mbox {Tub}_{r+m}(u_1,\dots ,u_r,n+1,\dots ,n+m,u_{r+1},\dots ,u_n)}.$$
\end{definition}
\medskip

In Figure~\ref{f:maria_mult_with_stello} we illustrate the following
example.
$${\mbox {Tub}_3(1,2,6,5,3,4)*_{\sigma} \mbox{Tub}_2(1,3,2,4)} = \mbox{Tub}_5(1,2,6,7,9,5,8,3,10,4)$$
...where $\sigma = (1,3,5,2,4).$ For comparison see a product with
the same operands in Figure~\ref{f:tube_mult_with_stello}.
\begin{figure}[h]\centering
                  \includegraphics[width=\textwidth]{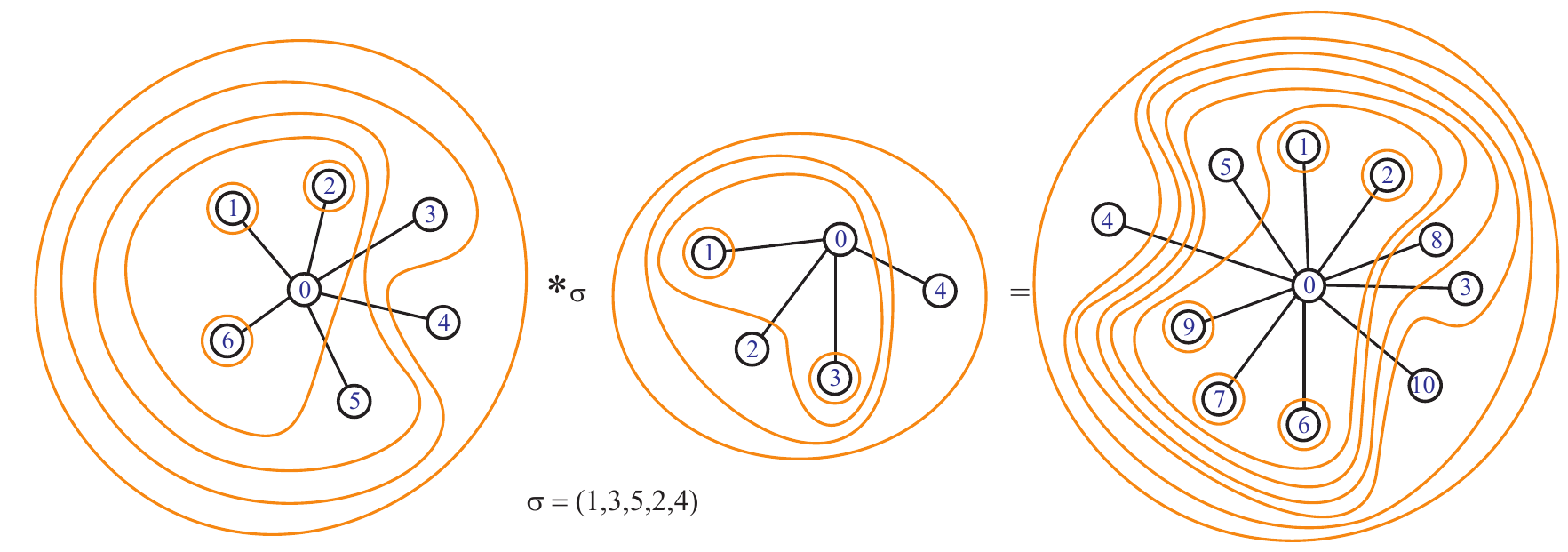}
\caption{Example of Definition~\ref{assstello}.
}\label{f:maria_mult_with_stello}
\end{figure}

Using Definition \ref{assstello},  we define a shuffle product on
the vector space $\KK[{\mathcal {MT}}(\St)]:= \bigoplus _{n\geq
1}\KK[{\mathcal {MT}}(\St_n)]$, where ${\mathcal {MT}}(\St_n)$
denotes the set of maximal tubings on $\St_n$, which correspond to
the vertices of stellohedra, as follows.

\begin{definition} \label{prodstello} Let $T\in {\mathcal {MT}}(\St_n)$ and ${ V}\in {\mathcal {MT}}(\St_m)$ be two maximal tubings. The product $T*{ V}\in {\mathcal {MT}}(\St_{n+m})$ is defined as follows:\begin{enumerate}
\item If $T={\mbox {Tub}_r(u_1,\dots ,u_{n})}\neq {\mbox {Tub}_n}$ and ${ V}= {\mbox {Tub}_s({ v}_1,\dots ,{ v}_{m})}\neq {\mbox {Tub}_m}$, then:
$$T*{ V} := \sum _{\sigma \in {\mbox {Sh}(n-r,m-s)}} T*_{\sigma }{ V},$$
where ${\mbox {Sh}(n-r,m-s)}$ denotes the set of
$(n{-}r,m{-}s)$-shuffles in $S_{n+m{-}(r+s)}$.
\item If $T=  {\mbox {Tub}_n}$ and ${ V} = {\mbox {Tub}_m}$, then:
$$T*{ V} = {\mbox {Tub}_{n+m}}.$$
\item If $T= {\mbox {Tub}_n}$ and ${ V}= {\mbox {Tub}_s({ v}_1,\dots ,{ v}_{m})}\neq {\mbox {Tub}_m}$, then:
$$T*{ V} := {\mbox {Tub}_{n+s}(1,\dots ,n, { v}_1+n,\dots ,{ v}_m+n)}.$$
\item If $T= {\mbox {Tub}_r(u_1,\dots ,u_{n})}\neq {\mbox {Tub}_n}$ and ${ V}= {\mbox {Tub}_m}$, then:
$$T*{ V} := {\mbox {Tub}_{r+m}(u_1,\dots ,u_r, n+1,\dots ,n+m, u_{r+1},\dots , u_{n})}.$$\end{enumerate}
\end{definition}
\medskip

\begin{proposition}\label{assofstell} The graded vector space $\KK[{\mathcal {MT}}(\St)]$, equipped with the product $*$ is an associative algebra.\end{proposition}
\medskip

\begin{proof}
Suppose that the elements $T={\mbox {Tub}_r(u_1,\dots ,u_{n})}$ in
${\mathcal {MT}}(\St_{n})$, $V= {\mbox {Tub}_s(v_1,\dots ,v_{m})}$
belongs to ${\mathcal {MT}}(\St_{m})$ and $W= {\mbox
{Tub}_z(w_1,\dots ,w_{p})}$ belongs to ${\mathcal {MT}}(\St_{p})$
are three maximal tubings, where eventually $T={\mbox {Tub}_n}$, $V
= {\mbox {Tub}_m}$ or $W = {\mbox {Tub}_p}$.

Applying the associativity of the shuffle, we get that:
$$
\begin{array}{lr}(T*V)*W=\\
&\hspace{-1.5in}\sum {\mbox {Tub}_{r+s+z}(u_1,\dots v_r, v_1{
+n},\dots ,w_1{ +n+m},\dots ,x_{\sigma^{-1} (1)},\dots
,x_{\sigma^{-1} (n+m+p{-}(r+s+z))})}\\~~~~~~~~~~~~~~~~~~ =
T*(V*W),\end{array}$$
  where we sum over all permutations $\sigma
\in {\mbox {Sh}(n{-}r,{m}-s,p{-}s)}$ and
$$x_i :=\left\{\begin{array}{lr} u_{r+i},&{\rm for}\ 1\leq i\leq n{-}r,\\
v_{s+i+r{-}n}{ +n},&{\rm for}\ n{-}r< i\leq n+m{-}r{-}s,\\
w_{z+i+r+s{-}n{-}m}{ +n+m},&{\rm for}\ n+m{-}r{-}s < i \leq
n+m+p{-}r{-}s{-}z,\end{array}\right\}$$ which ends the proof.
\end{proof}
\bigskip

\bigskip
\section{Questions}\label{quest}

 There are well-known
extensions of $\ssym$ and $\y Sym$ to Hopf algebras based on all of
the faces of the permutohedron and associahedron. These were first
described by Chapoton, in \cite{chap}, along with a Hopf algebra of
the faces of the hypercubes. We realize the first two Hopf algebras
using the graph tubings in \cite{ForSpr:2010}. They are denoted $\s
S\widetilde{ym}$ and $\y S\widetilde{ym}$ respectively, and so we
refer to Chapoton's algebra of the faces of the cube as $\c
S\widetilde{ym}$.

Immediately the question is raised: how might we relate the
coalgebra $\s S\widetilde{ym} \circ \c S\widetilde{ym}$ to our
algebra of stellohedra faces? How can we relate the Hopf algebra on
$\widetilde{C/C}$ for $C$ the corollas, thus an algebra on the faces
of the hypercube, to Chapoton's Hopf algebra $\c S\widetilde{ym}$?

Further questions arise as we look at the other polytopes in our set
of 12 sequences. (Of course, recall that 4 of them are only
conjecturally convex polytope sequences.) For instance, via our
bijection there is a Hopf algebra based on the weakly ordered
forests grafted to corollas--which we would like to characterize in
terms of known examples.

\section{Acknowledgements}
The authors thank a referee who made many good suggestions about an
early version. The third author is supported via Fondecyt Regular, Project 1171209.
The second author would like to thank the AMS and the
Mathematical Sciences Program of the National Security Agency for
supporting this research through grant H98230-14-0121.\footnote{This
manuscript is submitted for publication with the understanding that
the United States Government is authorized to reproduce and
distribute reprints.} The second author's specific position on the
NSA is published in \cite{freedom}. Suffice it to say here that the
second author appreciates NSA funding for open research and
education, but encourages reformers of the NSA who are working to
ensure that protections of civil liberties keep pace with
intelligence capabilities.

\bibliography{mybib}{}
\bibliographystyle{plain}

\end{document}